\documentclass[12pt]{amsart}%

\usepackage{amsmath, amssymb, amsthm, amsfonts, bbm, dsfont}
\usepackage{graphicx}
\usepackage{float}
\usepackage{wrapfig}
\restylefloat{figure}
\usepackage{mathrsfs}

\usepackage{hyperref}

\usepackage{a4wide}

\numberwithin{equation}{section}
\theoremstyle{plain}
\newtheorem{theorem}[subsubsection]{Theorem}
 \newtheorem{lemma}[subsubsection]{Lemma}
 \newtheorem{proposition}[subsubsection]{Proposition}
 \newtheorem{corollary}[subsubsection]{Corollary}
 \newtheorem{conjecture}[subsubsection]{Conjecture}

 \theoremstyle{definition}

\newtheorem{remark}[subsubsection]{Remark}

\usepackage{stackengine}
\usepackage{mathtools}
\usepackage{amssymb}
\usepackage{dsfont}

\usepackage{mathabx}

\newcommand{\ol}{\overline}

\newcommand{\rmaa}{\mathrm{A}}
\newcommand{\rmoo}{\mathrm{O}}
\newcommand{\rmdd}{\mathrm{D}}

\newcommand{\CC}{\mathbb{C}}
\newcommand{\cc}{\mathbb{C}}

\newcommand{\PP}{\mathbb{P}}

\newcommand{\ZZ}{\mathbb{Z}}

\newcommand{\bb}{\mnfn}
\newcommand{\bbS}{\mathbb{S}}
\newcommand{\bbll}{\mathbb{L}}

\newcommand{\calE}{\mathcal{E}}
\newcommand{\calF}{\mathcal{F}}

\newcommand{\calG}{\mathcal{G}}

\newcommand{\calO}{\mathcal{O}}

\newcommand{\calS}{\mathcal{S}}
\newcommand{\calB}{\mathcal{B}}

\newcommand{\chh}{\mathcal{H}}

\newcommand{\calA}{\mathcal{A}}

\newcommand{\frg}{\mathfrak{g}}
\newcommand{\frp}{\mathfrak{p}}

\newcommand{\frb}{\mathfrak{b}}

\newcommand{\frn}{\mathfrak{n}}

\newcommand{\tfrg}{\tilde{\mathfrak{g}}}

\newcommand{\wtil}{\widetilde}

\newcommand{\ind}{\textup{ind}}
\newcommand{\inc}{\textup{inc}}

\newcommand\Mod{\textup{Mod}}

\newcommand\per{\textup{per}}
\newcommand\xprop{\textup{prop}}

\newcommand\res{\textup{res}}

\newcommand\st{\textup{st}}

\newcommand\Sym{\textup{Sym}}

\newcommand{\Tr}{\textup{Tr}}

\newcommand\Hom{\textup{Hom}}

\newcommand{\RHom}{\mathrm{RHom}}

\newcommand{\Ext}{\textup{Ext}}

\newcommand{\HHH}{\mathrm{HHH}}
\newcommand{\rmHH}{\mathrm{HH}}
\newcommand{\rmH}{\mathrm{H}}

\newcommand\GL{\textup{GL}}
\newcommand\gl{\mathfrak{gl}}

\newcommand{\ad}{\textup{ad}}
\newcommand{\Ad}{\textup{Ad}}

\newcommand{\RR}{\mathbb{R}}
\newcommand{\Id}{\mathrm{Id}}

\newcommand{\quash}[1]{}

\newcommand{\sst}{\mathrm{sst}}

\newcommand{\Br}{\mathfrak{Br}}
\newcommand{\xGr}{\mathfrak{Br}^\flat}

\newcommand{\sbim}{\mathrm{SBim}}
\newcommand{\sbimn}{\sbim_n}
\newcommand{\ctr}{\mathcal{T}r}

\usepackage{tikz}
\usetikzlibrary{shapes,fit,intersections}
\usetikzlibrary{decorations.markings}
\usetikzlibrary{decorations.pathreplacing}
\usetikzlibrary{patterns}
\usetikzlibrary{matrix,arrows}
\usepgflibrary{decorations.pathmorphing}

\usepackage{tikz-cd}

\newcommand{\Wr}{\overline{W}}
\newcommand{\Hilb}{\textup{Hilb}}

\newcommand{\MF}{\mathrm{MF}}
\newcommand{\calXr}{\overline{\mathcal{X}}}
\newcommand{\calX}{\mathcal{X}}
\newcommand{\MFs}{\mathrm{MF}}
\newcommand{\calZ}{\mathcal{Z}}
\newcommand{\calC}{\mathcal{C}}

\newcommand{\FHilb}{\mathrm{FHilb}}

\newcommand{\frh}{\mathfrak{h}}

\newcommand{\CE}{\mathrm{CE}}
\newcommand{\ChE}{Chevalley-Eilenberg }

\newcommand{\odel}{\stackon{$\otimes$}{$\scriptstyle\Delta$}}

\newcommand{\Tqt}{\mathbb{T}_{q,t}}

\newcommand{\thc}{\dddot{\mathrm{Cat}}}

\newcommand{\CH}{\mathsf{CH}}

\newcommand{\HHom}{\mathbb{H}\mathrm{om}}

\newcommand{\bim}{\mathbb{B}\mathrm{im}}

\newcommand{\ho}{\mathrm{Ho}}

\newcommand{\Brn}{\Br_n}
\newcommand{\xGrn}{\Brn^\flat}

\def\Hilb{ \mathrm{Hilb}}
\newcommand{\Tor}{\mathrm{Tor}}

\newcommand{\ti}{\times}

\newcommand{\ot}{\otimes}

\newcommand{\GG}{G}
\newcommand{\fg}{\mathfrak{g}}

\def\mfgl{ \mathfrak{gl}}
\def\mfgln{ \mfgl_n }
\def\tmfgl{\widetilde{\mfgl}}
\def\tmfgln{ \tmfgl_n }

\def\stab{\mathrm{st}}
\def\MFst{\MF^{\stab}}
\def\MFstn{ \MFst_n}
\def\xId{ \mathbbm{1} }

\def\bx{ \mathbf{x}}
\def\by{ \mathbf{y}}

\def\rmmod{ \mathrm{mod}}

\def\rmalg{ \mathrm{alg}}
\def\rmgeo{ \mathrm{geo}}

\def\xCatalg{\Htp(\sbim_n)}

\def\mnfn{ \mathsf{B}}
\def\idosb{ B_{\xId}}

\def\rmalg{\mathrm{alg}}
\def\rmgeo{\mathrm{geo}}

\def\rmD{ \mathrm{D}}

\def\Vn{ V_n }

\def\PhiRbr{ \Phi_R}
\def\PhiRgr{ \Phi_{S}}
\def\Phigr{ \Phi^\flat}
\def\Phibr{ \Phi}

\def\xHalg{\mathrm{HHH}_{\rmalg}}
\def\xHgeo{\mathrm{HHH}_{\rmgeo}}

\def\Dbxy{\rmD(\CC[\bx,\by]-\rmmod)}

\def\Hsbimn{ \Htp(\sbim_n) }

\def\Lbul{ \Lambda^\bullet}

\def\xdgeo{ d_{\rmgeo}}
\def\xdalg{ d_{\rmalg}}

\def\Htp{ \mathrm{Ho}}

\def\xHilbn{ \Hilb_n(\CC^2) }
\def\GLn{ \GL_n }
\def\mfb{ \mathfrak{b}}
\def\Adv#1{ \mathrm{Ad}_{#1}}

\def\lmb{\lambda}
\def\tlmb{ \widetilde{\lmb}}

\def\xiota{\mathrm{res}}
\def\xpi{ \pi}

\title{Soergel bimodules and matrix factorizations}
\author{A. Oblomkov}
\address{
A.~Oblomkov\\
Department of Mathematics and Statistics\\
University of Massachusetts at Amherst\\
Lederle Graduate Research Tower\\
710 N. Pleasant Street\\
Amherst, MA 01003 USA
}
\email{oblomkov@math.umass.edu}

\author{L. Rozansky}
\address{
L.~Rozansky\\
Department of Mathematics\\
University of North Carolina at Chapel Hill\\
CB \# 3250, Phillips Hall\\
Chapel Hill, NC 27599 USA
}
\email{rozansky@math.unc.edu}

\begin{document}


\begin{abstract}
  We establish an isomorphism between the Khovanov-Rozansky triply
  graded link homology and the geometric triply graded homology due to
  the authors.
  Hence we provide an interpretation of
  the Khovanov-Rozansky homology of the closure of a braid
  \(\beta\) as the space of derived sections of a \(\CC^*\times \CC^*\)-
  equivariant sheaf \(\mathcal{T}r(\beta)\) on the Hilbert scheme
  \(\Hilb_n(\CC^2)\), thus proving a version of
  Gorsky-Negut-Rasmussen conjecture \cite{GorskyNegutRasmussen16}. As a consequence we prove that
  Khovanov-Rozansky homology of knots satisfies the \(q\to t/q\)
  symmetry conjectured by Dunfield-Gukov-Rasmussen \cite{DunfieldGukovRasmussen06}.
  We also apply our main result to  compute the Khovanov-Rozansky homology of torus links.
\end{abstract}

\maketitle

\section{Introduction}
\label{sec:introduction}

In this paper we explain why the knot homology theory developed in our previous work \cite{OblomkovRozansky16},\cite{OblomkovRozansky17},\cite{OblomkovRozansky18},\cite{OblomkovRozansky17a},
\cite{OblomkovRozansky18b}
produces the same link homology $\HHH_{\rmalg}$ as the homology $\HHH_{\rmgeo}$ coming from the Soergel bimodule construction \cite{KhovanovRozansky08b}, \cite{Khovanov07}.

\begin{theorem}\label{thm:main}
For any braid $\beta\in\Br_n$ the homologies $\HHH_{\rmalg}(\beta)$ and $\HHH_{\rmgeo}(\beta)$ are canonically isomorphic:
\begin{equation}
\label{eq:iso}
\HHH_{\rmalg}(\beta) \cong \HHH_{\rmgeo}(\beta).
\end{equation}
\end{theorem}


The  algebraic triply-graded homology \(\HHH_{\rmalg}(\beta)\)
of the closure \(L(\beta)\subset \RR^3\) defined in \cite{KhovanovRozansky08b}, \cite{Khovanov07}
is based  on
 Rouquier's
construction of the homomorphism \(\Phi_{R}\) from the braid group \(\Br_n\) to the homotopy category
\(\mathrm{Ho}(\sbim_n)\) of
complexes of Soergel bimodules. We remind the details of the construction in the second half of the introduction.



%

The geometric version of the homology is more
recent and is constructed in the series of papers by the authors:
\cite{OblomkovRozansky16},\cite{OblomkovRozansky18a},\cite{OblomkovRozansky19a}.
We constructed a trace on the braid group \(\Br_n\) with values in the
category of two-periodic complexes of \(T_{qt}=(\CC^*\!\ti\CC^*)\)-equivariant coherent sheaves \(D^{\per}_{T_{qt}}(\Hilb_n(\CC^2))\) on the
Hilbert scheme of points on the plane:
\begin{equation}
\label{eq:geomod}
\ctr: \Br_n\to \rmD^{\per}_{T_{qt}}(\Hilb_n(\CC^2)),\quad  \HHH_{\rmgeo}(\beta)=\RHom(\calO\ot\Lambda^\bullet\calB,\ctr(\beta)),
\end{equation}
where \(\calB\) is the vector bundle on  \(\Hilb_n(\CC^2)\)  which is  dual to the tautological bundle.
We proved that $\HHH_{geo}(\beta)$ is an isotopy invariant of the closure \(L(\beta)\).

In \cite{OblomkovRozansky19a} we have shown that if the closure $L(\beta)$ of a braid \(\beta\) is a knot then the complex of sheaves
\(\ctr(\beta)\) does not change if we switch the factors \(\CC\) in the product \(\CC^2=\CC\ti \CC\) underlying the Hilbert scheme.

\begin{corollary}\label{cor:duality}
  If the closure of \(\beta\) is a knot then the homology \(\HHH_{\rmalg}(\beta)\) is symmetric
  with respect to switching of homological and polynomial gradings:
  \[\HHH_{\rmalg}(\beta)=\HHH_{\rmalg}(\beta)|_{q=t/q}.\]
\end{corollary}
Thus our main result implies the conjecture of \cite{DunfieldGukovRasmussen06} which resisted an algebraic
proof for almost fifteen years.

The existence of the geometric model~\eqref{eq:geomod} for link the Soergel bimodule-based link homology was
\cite{GorskyNegutRasmussen16} and some special classes of braids in \cite{OblomkovRasmussenShende12}, \cite{GorskyOblomkovRasmussenShende14}, \cite{GorskyNegut15}. In particular, our result immediately imply
the formula homology of torus links \(T_{n,k}\) conjectured in the above mentioned papers (for purely algebraic proofs of these conjectures
see \cite{Hogancamp17},\cite{Mellit17}). The  simplest version of the conjecture is in terms of punctual Hilbert scheme \(\Hilb_n(\CC^2,0)\subset \Hilb_n(\CC^2)\):

\begin{corollary}\label{thm:torus-knots}
  For any \(n,k\) we have
  \[(1-\mathbf{q}^2)\cdot\HHH_{alg}(T_{n,nk+1})=H^0(\Hilb_n(\CC^2,0),\det(\calB)^k\ot \Lambda(\calB)).\]
\end{corollary}

For  a general braid \(\beta\) the geometric answer is  more complicated.
However, we have
the following  property of the geometric trace functor \cite{OblomkovRozansky18a}:
\begin{equation}\label{eq:FT}
\mathcal{T}r(\beta\cdot FT)=\mathcal{T}r(\beta)\otimes \det(\calB),
\end{equation}
where $FT$ is the full twist braid.
In general,
we hope that our result would allow to transfer geometric
results of moduli spaces to the algebraic side. For example, it
is natural to expect that the categorical version of
the localization by Halpern-Leistner \cite{HalpernLeistner14} transfers to the categorical
diagonalization of Elias-Hogancamp \cite{EliasHogancamp17}.




In the next subsection we describe the monoidal category of stable matrix factorizations \(\MF_n^{\st}\) that was used in
\cite{OblomkovRozansky16} to provide the geometric realization of the braid group \(\Br_n\).
The key step of our proof is the construction of an additive subcategory and a functor
\[\MF_n^\flat\subset \MF_n^{\st},\quad \mnfn: \MF_n^\flat\to \sbim_n. \]

\begin{theorem}
  For any \(n\) the functor \(\mnfn\) is monoidal and fully-faithful.
\end{theorem}

Establishing the isomorphism between the algebraic and geometric homologies requires a refinement of this theorem. Namely, we have to relate
derived homomorphisms
on the Soergel side with the \(\Lambda^\bullet \calB\) refinement of the homomorphism
space on the matrix factorization side. In the next section
we describe the refinement and  outline
our main argument.

\subsection{Link homology from Soergel bimodules}
Denote $\bx=x_1,\ldots,x_n$ and the same for $\by$. The category of Soergel bimodules \(\sbim_n\) is a monoidal additive subcategory
of \(\CC[\bx,\by]\)-modules that is a Karoubi envelope of tensor category generated by the standard
Soergel bimodules
\(B_i\) \cite{Soergel00}.  The standard bimodule \(B_i\) is the quotient of \(\CC[\bx,\by]\) modulo relations:
\[x_k=y_k, \quad k\ne i,i+1,\quad x^m_i+x^m_{i+1}=y^m_i+y^m_{i+1}, \quad m=1,2.\]


The category $\sbim_n$ has the (derived) duality endo-functor $\vee$. The convolution of bimodules endows $\Dbxy$ with the monoidal structure, the unit object being $\idosb = \CC[\bx,\by]/(\by-\bx)$. The Hochschild homology of a bimodule $B$  is defined as
\(
\rmHH_\bullet(B) = \Ext^\bullet(\idosb^\vee, B).
\)


Soergel \cite{Soergel00} constructed a homomorphism
\[
\PhiRgr\colon\xGr_n\longrightarrow \sbim_n 
\]
from the semi-group $\xGr_n$ of flat braid-graphs to Soergel bimodules.
In \cite{Rouquier04} Rouquier extended this homomorphism from graphs to braids:
\[
\PhiRbr\colon\Br_n\longrightarrow \Hsbimn, 
\]
where $\Br_n$ is the braid group and $\Htp(\sbim_n)$ is the homotopy category over $\sbim_n$.
%
%
The spaces of morphisms in
$\xCatalg$
have three gradings: $q$-grading is the polynomial degree, $t$-grading is the homological degree of the outer (homotopy) category and $a$-grading is the homological degree of the inner (derived) category. The category $\Hsbimn$ is equipped with the (derived) duality endo-functor $\vee$.

Modifying the construction of \cite{KhovanovRozansky08b}, Khovanov defined the triply graded homology of a braid $\beta\in\Br_n$ as a vector space
\begin{equation}
\label{eq:hmalg}
\xHalg(\beta) := \Hom_{\Hsbimn}
\bigl(\PhiRbr(\xId)^\vee,\PhiRbr(\beta)\bigr).
\end{equation}
 In other words, the space $\xHalg(\beta)$ results from, first, replacing the Soergel bimodules in the Rouquier complex $\Phi_R(\beta)$ with their Hochschild homology and, second, taking the homology of the resulting complex:
\[
\xHalg(\beta) := \rmH\bigl(\rmHH(\PhiRbr(\beta) \bigr).
\]
Denote $L(\beta)$ the link constructed by closing a braid $\beta$.
Khovanov proved that $\xHalg(\beta)$ is invariant under Markov moves, hence the triply graded homology is the isotopy invariant of $L(\beta)$:
\(
\xHalg\bigl(L(\beta)\bigr) := \xHalg(\beta).
\)

\subsection{Link homology from Hilbert schemes}
\subsubsection{Matrix factorizations}
In our recent work \cite{OblomkovRozansky16} we construct another triply graded link homology by using a different homomorphism.
Denote  $\mu\colon\tmfgln\rightarrow\mfgln$ the Grothendieck resolution. Explicitly, $\tmfgln = (\GLn\times \mfb)/B$, where $\mfb\subset\mfgln$ is the Borel subalgebra of upper-triangular matrices, $B\subset \GLn$ is the Borel subgroup of upper-triangular matrices, the action of $B$ on $\GLn\times\mfb$ is $h\cdot(g,Y) = (gh^{-1},\Adv{h}Y)$ and $\mu(g,Y) = \Adv{g}Y$.

Denote $\Vn=\CC^n$ the defining representation of $\GLn$.
A triple $(X,Y,v)\in\mfgln\times\mfgln\times\Vn$ is called \emph{stable} if $\CC\langle X,Y\rangle \,v = \Vn$.
Define
\begin{multline}
\label{eq:multst}
(\mfgln\times\tmfgln\times\tmfgln\times\Vn)^\stab
\\
=\{ (X,z_1,z_2,v)\in \mfgln\times\tmfgln\times\tmfgln\times\Vn\;|\; \text{$(X,\mu(z_1),v)$, $(X,\mu(z_2),v)$ are stable} \}.
\end{multline}
Define the polynomial $W\in\CC[\mfgln\times\tmfgln]$ by the formula $W_{Fl}(X,z) = \Tr \bigl(X \mu(z)\bigr)$. In \cite{OblomkovRozansky16} we
consider the category
\[
\MFst_n :=  \MF_{\GL_n}\Bigl( (\mfgln\times\tmfgln\times\tmfgln\times\Vn)^\stab; W_{Fl}(X,z_2) - W_{Fl}(X,z_1)\Bigr),
\]
which is equipped with the duality endo-functor $\vee$.
We define two  homomorphisms
\begin{equation}
\label{eq:homs}
\Phi^\flat\colon\xGr_n\longrightarrow\MFst_n,\qquad\Phi\colon\Br_n\longrightarrow\MFst_n
\end{equation}
and a triply graded homology $\xHgeo$ for graphs \(\gamma\in \xGr_n\) and for braids $\beta\in\Br_n$ by similar formulas:
\begin{align}
\label{eq:homgeo}
\xHgeo(\gamma) & = \Hom_{\MFst_n}\bigl(\Phi^\flat(\xId)^\vee,\Phi^\flat(\beta)\otimes \Lambda^\bullet  V_n\bigr),\\
\label{eq:hmgeo}
\xHgeo(\beta) & = \Hom_{\MFst_n}\bigl(\Phi(\xId)^\vee,\Phi(\beta)\otimes \Lambda^\bullet  V_n\bigr),
\end{align}
(note that by our definition, $\Phi^\flat(\xId) = \Phi (\xId)$).
We proved that $\xHgeo$ is invariant under the Markov moves, thus defining a triply graded link homology
\(
\xHgeo\bigl(L(\beta)\bigr) := \xHgeo(\beta),
\)
for links presented as braid closures.

\subsubsection{Sheaves on a Hilbert scheme}
In order to see the indirect
relation between the category $\MFstn$ and the Hilbert scheme $\xHilbn$, define $\xHilbn$ as the quotient
\begin{equation}
\label{eq:dhs}
\xHilbn = \{(X,Y,v)\in(\mfgln\times\mfgln\times \Vn)^{\stab}\;|\;[X,Y]=0 \}/\GLn.
\end{equation}
Then the matrix $X$ and the vector $v$ from~\eqref{eq:multst} resemble $X$ and $v$ of~\eqref{eq:dhs}, whereas $\mu(z_1)$ and $\mu(z_2)$ resemble $Y$ of~\eqref{eq:multst}.

In~\cite{OblomkovRozansky18a} we presented an alternative construction of the link homology $\xHgeo(\beta)$ in which $\xHilbn$ appears explicitly. We defined the Chern character functor
\[
\CH\colon \MFstn\longrightarrow \rmD^{\per}\bigl( \xHilbn\bigr)
\]
and proved the relation
\[\xHgeo(\beta) = \RHom_{ \rmD^{\per}\bigl( \xHilbn\bigr)} \Bigl(
\calO\ot\Lambda^\bullet\calB,\CH\bigl(\Phibr(\beta)\bigr)\Bigr),\]
where \(\calB\) is the dual of the tautological vector bundle on \(\Hilb_n(\CC^2)\). In other words, the functor $\ctr$ of \eqref{eq:geomod} appears as the composition of the second homomorphism of \eqref{eq:homs} and the Chern character functor: \[\ctr = \CH\circ \Phi.\]


\subsection{Equivalences}
The main result of this paper is the (partly canonical) isomorphism \eqref{eq:iso} between both link homologies.
This isomorphism originates from a special   functor
\begin{equation}
\label{eq:fnB}
\mnfn\colon \MFstn\longrightarrow\sbimn
\end{equation}
defined as a composition of a pull-back and a push-forward. Namely, consider a projection $\lambda$ from the Borel subalgebra $\mfb$ to the Cartan subalgebra $\frh$:
\[\lambda\colon\mfb\rightarrow \frh= \CC^n,\quad Y\mapsto (Y_{11},\dots,Y_{nn}).\]  There is an associate map $\tlmb\colon\tmfgln\rightarrow\CC^n$, $\tlmb(g,Y) = \lambda(Y)$. Consider two maps
\[
\begin{tikzcd}
& (\tmfgln\times\tmfgln\times\Vn)^{\stab}
\ar[dl, hook'
,"\xiota"'] \ar[dr,"\pi_\frh"]
\\
(\mfgln\times\tmfgln\times\tmfgln\times\Vn)^{\stab}
&&
\CC^n\times\CC^n
\end{tikzcd}
\]
where
\[
\xiota(z_1,z_2,v) = (0,z_1,z_2,v),\qquad
\xpi_\frh(z_1,z_2,v) = \bigl(\tlmb(z_1),\tlmb(z_2)\bigr).
\]
Thus we define the functor $\mnfn$ of ~\eqref{eq:fnB} as a composition of a pull-back and a push-forward:
\[\mnfn = \xpi_{\frh*}\circ \xiota^*.\]


In section~\ref{sec:construction-functor} we  construct a monoidal additive subcategory \(\MF_n^\flat\) and  monoidal  functor as image of the homomorphism:
\[\Phi^\flat:\mathfrak{Br}_n^\flat\to \MF_n^\flat.\]


\begin{theorem}

\end{theorem} \label{thm:main-cat} For any \(n\) we have
\begin{enumerate}
\item
The functor $\mnfn$ is monoidal.
\item The functor $\mnfn$ intertwines the homomorphisms $\PhiRgr$ and $\Phi^\flat$:
\[
\begin{tikzcd}
& \xGrn
\ar[ld,"\Phi^\flat"'] \ar[rd,"\PhiRgr"]
\\
\MF_n^{\st}\supset\MF_n^\flat
\ar[rr,"\mnfn"]
&&
\sbim_n
\end{tikzcd}
\]
\item
The functor $\mnfn$ converts the exterior power of $V_n$ into the homological degree of the derived category $\rmD(\CC[\bx,\by]-\rmmod)$: for any two graphs $\gamma_1,\gamma_2\in\xGrn$
\[
\Hom_{\MFstn}\bigl(\Phi^\flat(\gamma_1),\Phi^\flat(\gamma_2)\otimes\Lambda^\bullet V_n \bigr) \cong
\Ext^\bullet_{\Dbxy}\bigl(\PhiRgr(\gamma_1),\PhiRgr(\gamma_2) \bigr)
\]
\item
The functor $\mnfn$ intertwines duality endo-functors up to a `twisting':
\[
\mnfn\bigl( \Phi^\flat(\gamma)^\vee \otimes \det V_n \bigr) = \PhiRgr(\gamma)^\vee[n]_a.
\]
\end{enumerate}

Since $\det \Vn \otimes \Lbul \Vn^\vee = \Lambda^{n-\bullet} \Vn$, this theorem implies the isomorphism between homologies of graph closures:
\begin{corollary}
For any graph $\gamma\in\xGrn$ the functor $\mnfn$ induces an isomorphism
\begin{equation}
\label{eq:griso}
\xHalg(\gamma) \cong \xHgeo(\gamma).
\end{equation}
\end{corollary}

It remains to extend the isomorphism~\eqref{eq:griso} from graphs to braids. Recall the Murakami-Ohtsuki-Yamada construction \cite{MurakamiOhtsukiYamada98} presenting a braid $\beta$ (considered as an element of the Hecke algebra) as a (weighted) alternating sum of associated braid graphs:
\begin{equation}
\label{eq:mcalg}
\beta = \sum_{\gamma\in\xGr(\beta)} (-1)^{s(\beta,\gamma)} \gamma,
\end{equation}
where $s(\beta,\gamma)$ is a $\ZZ$-valued function and  we ignore the powers of $q$ for simplicity. Categorifying this formula, Rouquier~\cite{Rouquier04}  represented a braid $\beta$ by a complex $\Phi_R(\beta)$ of Soergel bimodules $\Phi_S(\gamma)$:
\begin{equation}
\label{eq:mcgeo}
\Phibr(\beta) =
\Bigl(\bigoplus_{\gamma\in\xGr(\beta)} \PhiRgr(\gamma),\xdalg \Bigr).
\end{equation}
We show that the image $\Phi(\beta)$ of the braid $\beta$ admits a similar presentation within the geometric construction as a complex of sheaves associated with graphs:
\[
\Phibr(\beta) =
\Bigl( \bigoplus_{\gamma\in\xGr(\beta)} \Phigr(\gamma),\xdgeo \Bigr)
\]
%
%
%
%
while the functor $\mnfn$ intertwines the differentials:
\[
\begin{tikzcd}
\bigoplus_{\gamma\in\xGr(\beta)} \Phigr(\gamma)
\ar[r,"\xdgeo"]
\ar[d,"\mnfn"]
&
\bigoplus_{\gamma\in\xGr(\beta)} \Phigr(\gamma)
\ar[d,"\mnfn"]
\\
 \bigoplus_{\gamma\in\xGr(\beta)} \PhiRgr(\gamma)
 \ar[r,"\xdalg"]
&
 \bigoplus_{\gamma\in\xGr(\beta)} \PhiRgr(\gamma)
\end{tikzcd}
\]
In order to relate homologies $\xHalg(\beta)$ and $\xHgeo(\beta)$ we substitute the presentations~\eqref{eq:mcalg} and \eqref{eq:mcgeo} into the formulas~\eqref{eq:hmalg} and $\eqref{eq:hmgeo}$.
The difference between the computation of resulting homologies is that the cone of $\PhiRbr(\beta)$ is in the outer homotopy category, hence the homology of $\xdalg$ is computed after the Hochschild homology is applied to the bimodules $\PhiRgr(\gamma)$, whereas the cone of $\Phibr(\beta)$ is in the same matrix factorization category $\MFstn$, so the differential $\xdgeo$ is added to the matrix factorization differentials $D_\gamma$ of the matrix factorizations $\Phigr(\gamma)$. As a result, the homology in~\eqref{eq:hmgeo} is computed with respect to the total differential $\xdgeo + \sum D_\gamma$. However, this homology can be computed by spectral sequence, first taking homology with respect to $\sum D_\gamma$ (which matches the Hochschild homology), then taking the homology with respect to $\xdgeo$ (which matches $\xHalg(\beta)$) and, finally, taking homology with respect to secondary differentials. The counting of $t$-degree indicates that the spectral sequence converges in the second term, and this implies the isomorphism $\xHgeo(\beta)\cong \xHalg(\beta)$.

\subsection{Previous work}
\label{sec:previous-work}

The  relation between the knot homology \(\HHH_{alg}(\beta)\), \(\beta\in \Br_n\) and the sheaves on the Hilbert scheme of points
on the plane \(\Hilb_n(\CC^2)\) was proposed in various forms by several groups. Here we  briefly outline their contributions to the problem.

The  physics prospective on the question originates in the work of Aganagic and Shakirov
\cite{AganagicShakirov12}, the authors proposed formulas for the homology of torus knots in terms of Macdonald polynomials.
Their motivation comes from the fact that the compliment to the torus knot has a \(U(1)\)-symmetry thus one can compute the index
of the usual Chern-Simons theory with respect to this action and upgrade the HOMFLYPT polynomial to the super polynomial.
Using the TQFT perspective on Chern-Simons theory they compute the index in terms of \(\mathrm{SL}_2(\ZZ)\) action
from \cite{Kirillov96}. The  last work explains the connection with the Macdonald polynomials.

The  work of Haiman \cite{Haiman02a} provides a bridge between the theory of Macdonald polynomials and \(K\)-theory of the
Hilbert schemes on the plane.  Gorsky and Negut \cite{GorskyNegut15} exploited  this connection together some more recent developments
in the theory of elliptic Hall algebras in order to provide a geometric set of conjectures that imply the formulas of
\cite{AganagicShakirov12}. More precisely, they
constructed the classes \(S_{n,k}\) in \(\CC^*\ti\CC^*\)-equivariant
\(K\)-theory \(K_{\CC^*\ti\CC^*}(\Hilb_n(\CC^2))\) such that the \(\CC^*\ti \CC^*\) Euler characteristic of
\(S_{n,k}\ot \Lambda^*\calB\) reproduces the stable limit of the Aganagic-Shakirov formulas.
The more recent papers by Elias, Mellit and Hogancamp \cite{EliasHogancamp16}, \cite{Mellit17}, \cite{Hogancamp17} provide an algebraic proof of the conjectures.

From another perspective, a  torus knot \(T_{n,k}\) is a link of the singularity of the toric curve \(x^n=y^k\). The relation between a
Hilbert scheme of points on a singular curve and knot invariants was studied first in  \cite{CampiloDelgadoGuseinZade99} and developed further in \cite{OblomkovShende12}, \cite{OblomkovRasmussenShende12}. In the latter work the isomorphism between the homology of the Hilbert scheme
of points on the curve \(x^n=y^{1+kn}\) and the space of sections of \(\det(\calB)^k\) on \(\Hilb_n(\CC^2,0)\) was observed.

On the other hand, the earlier conjecture of Gorsky \cite{Gorsky12a}
on the Catalan numbers and super-polynomials of \(T_{n,n+1}\) combined  the geometric constructions
of the modules  over the rational Cherednik algebra \cite{GordonStafford06}, \cite{OblomkovYun16} led to
work \cite{GorskyOblomkovRasmussenShende14} in which the statement of theorem \ref{thm:torus-knots} was presented and motivated.

The  first precise conjecture relating \(\HHH_{alg}(\beta)\) to  \(K\)-theory of coherent sheaves on \(\Hilb_n(\CC^2)\)
appears in the work of Gorsky, Negut and Rasmussen  \cite{GorskyNegutRasmussen16}. The authors proposed a conjecture
that relates the knot homology to the \(K\)-theory of the flag Hilbert scheme \(\FHilb_n(\CC^2)\). The conjecture
from \cite{GorskyNegutRasmussen16} identifies the Picard group of the flag Hilbert scheme with the Jucys-Murphy commutative
subgroup of \(\Br_n\).

In  \cite{OblomkovRozansky16} and \cite{OblomkovRozansky18a} we constructed a link homology \(\HHH_{geo}(\beta)\) which is directly related to the Hilbert scheme \(\Hilb_n(\CC^2)\) as well as to the flag Hilbert scheme.
The  invariant \(\HHH_{geo}(\beta)\) has most of the properties predicted in \cite{GorskyNegutRasmussen16} but initial it was not clear why in would be related to \(\HHH_{alg}\). This paper provides a proof the equality between the two link homologies.

Finally, for a braid $\beta$, Gorsky and Hogancamp \cite{GorskyHogancamp17}  construct a complex
of sheaves on the
isospectral Hilbert scheme with the property that there is a spectral sequence connecting its homology
and \(\HHH_{alg}(\beta)\).

\subsection{The structure of the paper}
In section~\ref{sec:prel-matr-fact}
we recall the definition of the category of equivariant matrix factorizations
from \cite{OblomkovRozansky16}. We also explain our construction of the
equivariant push-forward for matrix factorizations and explain basic
properties of the push-forward. In section~\ref{sec:construction-functor}
we provide the details of the construction of our main functor
\(\bb\). We prove that \(\bb\) is monoidal and intertwines the
induction functors. Also in this section the value of \(\bb\) of
the elementary braid graph is computed.

In section~\ref{sec:relations-mfflat} we prove the MOY relations in the
category of matrix factorizations that correspond to the braid relations.
In section~\ref{sec:markov-moves} we prove the MOY relations that correspond
to the strand removing Markov moves. The MOY relations allow us
to use Hao Wu induction strategy \cite{Wu08} in section~\ref{sec:properties-functor} where we prove our main categorical theorem
\ref{thm:main-cat}. In the same section we show how the categorical
theorem implies the comparison theorem~\ref{thm:main}.
Finally, in section~\ref{sec:appl-furth-direct} we prove corollaries that we discuss
in the introduction. We also state conjectures suggested by our results.

{\bf Acknowledgments}
We would like to thank Dmitry Arinkin, Tudor Dimofte, Eugene Gorsky, Sergey Gukov, Tina Kanstrup, Ivan Losev, Roman Bezrukavnikov and Andrei Negu{\c t} for useful discussions. The authors also would like to thank UC Davis for hospitality
during the July 2019 FRG conference, most of the outline of this paper
was written during this visit.
The work of A.O. was supported in part by  the NSF CAREER grant DMS-1352398, NSF FRG grant DMS-1760373 and Simons Fellowship.
The work of L.R. was supported in part by  the NSF grant DMS-1108727.

\section{Preliminaries on matrix factorizations}
\label{sec:prel-matr-fact}

The goal of this section is to collect the definitions and results on matrix factorizations that are used in the
main body of the paper.

Let \(\calZ\) be an affine manifold and \(W\in \CC[\calZ]=R\) then a matrix factorization \cite{Eisenbud80} is a pair
\[ M=M_1\oplus M_2, \quad D \in \Hom(M_1,M_0)\oplus \Hom(M_0,M_1),\quad D^2=W,\]
where we assume that \(M_i\) are free \(R\)-modules.

Given \(\calF=(M,D)\) and \(\calG=(M',D')\) the space morphisms \(\Hom(\calF,\calG)\) consists of \(\ZZ_2\)
graded \(R\)-linear maps \(\phi\in \Hom_R(M,M')\) that respects the differentials:
\[D'\circ \phi=\phi\circ D.\]

A homotopy \(h\) between the morphisms \(\phi,\psi\in \Hom_R(M,M')\) shifts the \(\ZZ_2\) grading by \(1\) and
\(\phi-\psi=D'\circ h-h \circ D\).  Thus the matrix factorizations are elements of the homotopy category
\[\MF(\calZ,W)\]
and it was shown by Orlov \cite{Orlov04} that this category is triangulated.

There is a natural generalization of the category of matrix factorizations to the case when \(X\) is a quasi-projective
manifold or even Artin stack \cite{PolishchukVaintrob11}. However, in our theory the spaces that host matrix factorizations
are usually quotients of affine manifolds by group actions. Hence we prefer to work with {\it equivariant matrix factorizations}
which were introduced in  \cite{OblomkovRozansky16}.

Before we define the First let us remind the construction of the Chevalley-Eilenberg complex.

\subsection{Chevalley-Eilenberg complex}
\label{sec:chev-eilenb-compl}

Suppose that $\frh$ is a Lie algebra. Chevalley-Eilenberg complex
 $\CE_\frh$ is the complex $(V_\bullet(\frh),d)$ with $V_p(\frh)=U(\frh)\otimes_\CC\Lambda^p \frh$ and differential $d_{ce}=d_1+d_2$ where:
 \def\dtheta{d}
 $$ d_1(u\otimes x_1\wedge\dots \wedge x_p)=\sum_{i=1}^p (-1)^{i+1} ux_i\otimes x_1\wedge\dots \wedge \hat{x}_i\wedge\dots\wedge x_p,$$
 $$ d_2(u\otimes x_1\wedge\dots \wedge x_p)=\sum_{i<j} (-1)^{i+j} u\otimes [x_i,x_j]\wedge x_1\wedge\dots \wedge \hat{x}_i\wedge\dots\wedge \hat{x}_j\wedge\dots \wedge x_p,$$

 Let us denote by $\Delta$ the standard map $\frh\to \frh\otimes \frh$ defined by $x\mapsto x\otimes 1+1\otimes x$.
 Suppose $V$ and $W$ are modules over the Lie algebra $\frh$ then we use notation
 $V\odel W$ for  the $\frh$-module which is isomorphic to $V\otimes W$ as a vector space, the $\frh$-module structure being defined by  $\Delta$. Respectively, for a given $\frh$-equivariant matrix factorization $\calF=(M,D)$ we denote by $\CE_{\frh}\odel \calF$
 the $\frh$-equivariant matrix factorization $(CE_\frh\odel\calF, D+d_{ce})$. The $\frh$-equivariant structure on $\CE_{\frh}\odel \calF$ originates from the
 left action of $U(\frh)$ that commutes with right action on $U(\frh)$ used in the construction of $\CE_\frh$.

 A slight modification of the standard fact that $\CE_\frh$ is the resolution of the trivial module implies that \(\CE_\frh\stackon{$\otimes$}{$\scriptstyle\Delta$} M\) is a free resolution of the
$\frh$-module $M$.

\subsection{Equivariant matrix factorizations}
\label{sec:equiv-matr-fact}

Let us assume that there is an action of the Lie algebra \(\frh\) on \(\calZ\) and \(F\) is a \(\frh\)-invariant function.
Then we can construct the following triangulated category \(\MFs_{\frh}(\calZ,W)\).

The objects of the category are  triples:
\[\mathcal{F}=(M,D,\partial),\quad (M,D)\in\MFs(\calZ,W) \]
where $M=M^0\oplus M^1$ and $M^i=\CC[\calZ]\otimes V^i$, $V^i \in \Mod_{\frh}$,
$\partial\in \oplus_{i>j} \Hom_{\CC[\calZ]}(\Lambda^i\frh\otimes M, \Lambda^j\frh\otimes M)$ and $D$ is an odd endomorphism
$D\in \Hom_{\CC[\calZ]}(M,M)$ such that
$$D^2=F,\quad  D_{tot}^2=F,\quad D_{tot}=D+d_{ce}+\partial,$$
where the total differential $D_{tot}$ is an endomorphism of $\CE_\frh\odel M$, that commutes with the $U(\frh)$-action.


Note that we do not impose the equivariance condition on the differential $D$ in our definition of matrix factorizations. On the other hand, if $\calF=(M,D)\in \MFs(\calZ,F)$ is a matrix factorization with
$D$ that commutes with $\frh$-action on $M$ then $(M,D,0)\in \MFs_\frh(\calZ,F)$.


Given two $\frh$-equivariant matrix factorizations $\calF=(M,D,\partial)$ and $\tilde{\calF}=(\tilde{M},\tilde{D},\tilde{\partial})$ the space of morphisms $\Hom(\calF,\tilde{\calF})$ consists of
homotopy equivalence classes of elements $\Psi\in \Hom_{\CC[\calZ]}(\CE_\frh\odel M, \CE_\frh\odel \tilde{M})$ such that $\Psi\circ D_{tot}=\tilde{D}_{tot}\circ \Psi$ and $\Psi$ commutes with
$U(\frh)$-action on $\CE_\frh\odel M$. Two maps $\Psi,\Psi'\in \Hom(\calF,\tilde{\calF})$ are homotopy equivalent if
there is \[ h\in  \Hom_{\CC[\calZ]}(\CE_\frh\odel M,\CE_\frh\odel\tilde{M})\] such that $\Psi-\Psi'=\tilde{D}_{tot}\circ h- h\circ D_{tot}$ and $h$ commutes with $U(h)$-action on  $\CE_\frh\odel M$.

 Given two $\frh$-equivariant matrix factorizations $\calF=(M,D,\partial)\in \MFs_\frh(\calZ,F)$ and $\tilde{\calF}=(\tilde{M},\tilde{D},\tilde{\partial})\in \MFs_\frh(\calZ,\tilde{F})$
 we define $\calF\otimes\tilde{\calF}\in \MFs_\frh(\calZ,F+\tilde{F})$ as the equivariant matrix factorization $(M\otimes \tilde{M},D+\tilde{D},\partial+\tilde{\partial})$.

 \subsection{Push forwards, quotient by the group action}
\label{sec:push-forwards}

The technical part of \cite{OblomkovRozansky16} is the construction of push-forwards of equivariant matrix factorizations. Here we state the main
results, the details may be found in section 3 of \cite{OblomkovRozansky16}. We need push forwards along projections and embeddings. We also use  the
functor of taking quotient by group action for our definition of the convolution algebra.

The projection case is more elementary.
Suppose we have a smooth projection \(\pi:\calZ\to \calX\)
both \(\calZ \) and \(\mathcal{X}\) have \(\frh\)-action and
the map \(\pi:\mathcal{Z}\rightarrow\mathcal{X}\) is \(\frh\)-equivariant. Then
for any $\frh$ invariant element $w\in\CC[\calX]^\frh$ there is a functor
\[\pi_{*}:\quad  \MFs_{\frh}(\calZ, \pi^*(w))\rightarrow \MFs_{\frh}(\mathcal{X},w)
\]
which is induced by the push-forwards of the modules
underlying the matrix factorizations.


We define an embedding-related push-forward in the case when the subvariety $\calZ_0\xhookrightarrow{j}\calZ$
is the common zero of an ideal $I=(f_1,\dots,f_n)$ such that the functions $f_i\in\CC[\calZ]$ form a regular sequence. We assume that the Lie algebra $\frh$ acts on $\calZ$ and $I$ is $\frh$-invariant. Then there exists an $\frh$-equivariant Koszul complex $K(I)=(\Lambda^\bullet \CC^n\otimes \CC[\calZ],d_K)$ over $\CC[\calZ]$ which has non-trivial homology only in degree zero. Then in section~3 of \cite{OblomkovRozansky16} we define the push-forward functor
\[
j_*\colon\quad \MFs_{\frh}(\calZ_0,W|_{\calZ_0})\longrightarrow
\MFs_{\frh}(\calZ,W),
\]
for any $\frh$-invariant element $W\in\CC[\calZ]^\frh$.

Finally, let us discuss the quotient map. The complex \(\CE_\frh\) is a resolution of the trivial \(\frh\)-module by free modules. Thus the correct derived
version of taking \(\frh\)-invariant part of the matrix factorization \(\mathcal{F}=(M,D,\partial)\in\MFs_\frh(\calZ,W)\), \(W\in\CC[\calZ]^\frh\) is
\[\CE_\frh(\mathcal{F}):=(\CE_\frh(M),D+d_{ce}+\partial)\in\MFs(\calZ/H,W),\]
where \(\calZ/H:=\mathrm{Spec}(\CC[\calZ]^\frh )\) and use the general definition of \(\frh\)-module \(V\):
\[\CE_\frh(V):=\Hom_\frh(\CE_\frh,\CE_\frh\odel V).\]

\subsection{Base change}
\label{sec:base-change}

Unlike push-forward functor, the pull-back of matrix factorizations is defined for any
regular map \(f: \calX\to \calZ\). If we assume that both \(\calX,\calZ\) have \(\frh\)
action and \(f\) is \(\frh\)-equivariant then pull-back of the \(\CC[\calZ]\) modules
induces the functor
\[f^*: \MF_\frh(\calZ,W)\to \MF_\frh(\calX,W).\]

Just as in the case of coherent sheaves we have the smooth base change isomorphism.
In more details, suppose we have affine manifolds \(\calX,\calX',\calS,\calS\) and
the corresponding potentials \(W_{\calX}, W_{\calX'},W_{\calS},W_{\calS'}\) that
fit into the commuting diagrams:
\[\begin{tikzcd}
    \calX'\ar[r,"g'"]\ar[d,"f'"]&\calX\ar[d,"f"]\\
    \calS'\ar[r,"g"] &\calS
  \end{tikzcd},\quad \quad
\begin{tikzcd}
    \MF(\calX',W_{\calX'})\ar[d,"f'_*"]&\MF(\calX,W_{\calX})\ar[l,"g^{\prime *}"]\ar[d,"f_*"]\\
   \MF( \calS',W_{\calS'}) &\ar[l,"g^*"]\MF(\calS,W_{\calS})
  \end{tikzcd}.
\]

If \(g\) is a flat map then we have an natural transformation
that identifies the functors:
\[g^*\circ f_* = f'_*\circ g^{\prime*}.\]

The identity also holds in the equivariant setting.
\section{Construction of the functor}
\label{sec:construction-functor}

In this section we recall a construction of the functor:
\[\bb:\MF_n\to \rmD^{\per}(\cc^n\ti\cc^n).\]
and prove that the functor is monoidal and compute the functor on the some important
collection of matrix factorizations.  The  matrix factorization that we study are the
key ingredients for our braid group realization from \cite{OblomkovRozansky16}.

From now on we fix the conventions for the most used groups and Lie algebras:
\[G_n=\GL_n,\quad \mathfrak{g}_n=\gl_n,\quad \frh_n\subset\frn_n\subset \frb_n\subset \gl_n.\]
When the rank of the group is clear from the context we suppress the lower
index.

\subsection{Motivation and the main definition}
\label{sec:motiv-main-defin}

In our previous work \cite{OblomkovRozansky18d} we introduced  the three-category \(\thc_{\gl}\), the objects in the category are labeled by the positive integers.
The two-category of morphisms has objects:
\[\mathrm{Obj}(\HHom(\mathbf{n},\mathbf{m}))=\{Z,W\in \cc[\fg_n\ti\fg_m\ti Z]^{\GG_n\ti\GG_m}\},\]
here we assume that the space \(Z\) is acted on by  \(\GG_n\ti \GG_m\).

The composition of the morphisms is described in the cited paper. The one-category of morphisms between \((Z',W'),(Z'',W'')\in \HHom(\mathbf{n},\mathbf{m})\) has
objects:
\[\mathrm{Obj}(\Hom((Z',W'),(Z'',W'')))=\MF_{\GG_n\ti\GG_m}(\fg_n\ti Z'\ti Z'' \ti \fg_m,W'-W'').\]

There are two particular objects that play central role in our constructions:
\[\rmaa_n=(\wtil{\fg_n},W_{Fl}),\quad \rmoo_n=((\fg_n\ti \CC^n)^{st},W_{pt})\in \HHom(\mathbf{n},\mathbf{0}),\]
\[W_{Fl}(X,g,Y)=\Tr(X\Ad_gY),\quad X\in\fg_n,\quad Y\in \frb_n,\quad g\in GG_n,\]
\[W_{pt}(X,Y)=\Tr(XY), \quad X,Y\in \fg_n,\]
here we use the identification \(\wtil{\fg}_n=\GG_n\ti\frb/B\) and \((\fg_n\times\CC^n)^{st}\) stands for the stable sublocus that consists of the pairs
\((Y,v)\) such that \(\CC[Y]v=\CC^n\).

The one-category of the endomorphisms of the object \(A_n\)  is a dg category of the equivariant matrix factorizations:
\[\Hom(A_n,A_n)=\MF_{\GG_n\ti B^2}(\fg_n\ti (\GG_n\ti \frb)^2,W),\]
\[ W(X,g_1,Y_1,g_2,Y_2)=\Tr(X(\Ad_{g_1}Y_1-\Ad_{g_2}Y_2)),\]
we abbreviate  the notation for this category by \(\MF_n\).

The composition of elements of \(\Hom(A_n,A_n)\) endows the category \(\MF_n\) with the monoidal structure which we denote by \(\star\) to keep
our notations consistent with \cite{OblomkovRozansky16}. In the last mentioned paper we constructed a group homomorphism from the braid group.
\[\Phi: \Br_n\to (\MF_n,\star).\]

Thus we have an action of the braid group on category \(\Hom(\rmaa_n,\rmoo_n)\), that is for every element \(\beta\in \Br_n\) we obtain a functor:
\[\Phi_\beta:\Hom(\rmaa_n,\rmoo_n)\to \Hom(\rmaa_n,\rmoo_n).\]

The category \(\Hom(\rmaa_n,\rmoo_n)=\MF_{\GG_n\ti B_n}(\frg_n\ti \GG_n\ti \frb_n\ti \frb_n,W_A-W_O)\) has a simpler model:

\begin{proposition}
  For any \(n\) we have an isomorphism of the linear categories:
  \[\Hom(\rmaa_n,\rmoo_n)=\rmdd^{\per}(\CC^n).\]
\end{proposition}
\begin{proof}
We use Kn\"orrer periodicity, the potential in our matrix factorizations is quadratic:
\[W_{Fl}-W_{pt}=\Tr(X(Y_1-\Ad_gY_2)).\]

Thus we conclude that
\[\Hom(\rmaa_n,\rmoo_n)=\rmdd^{\per}_{\GG_n\ti B_n }((\GG_n\ti \frb\ti \cc^n)^{st})=\rmdd^{\per}_{B_n}((\frb\ti \cc^n)^{st}).\]

To complete our argument we observe that we have a natural slice to the \(B_n\)-action:
\begin{equation}\label{eq:kappa}
  \varkappa:\CC^n\to (\frb\ti \cc^n)^{st},\quad \varkappa(y)=(Y(y),e_n),\end{equation}
\[Y(y)_{ij}=\delta_{i-j}y_i+\delta_{j-i-1}.\]
Since all points of the slice have a trivial stabilizer the statement follows.
\end{proof}

Let us emphasize that we choose double grading of the variable as
\[\deg (Y_1)_{ij}=\deg (Y_2)_{ij}=\mathbf{q}^2,\quad \deg(X_{ij})=\mathbf{t}^2\mathbf{q}^{-2}. \]

Thus a general theorem implies that for a braid \(\beta\in \Br_n\) we have a two-periodic complex of bimodules \(\calC_\beta\) that is a Fourier-Mukai
kernel for \(\Phi_\beta\). Let us denote by \(\bim\) the abelian category of the \(R_n\)-bimodules, \(R_n=\cc[y_1,\dots,y_n]\). Respectively, we have constructed
the functor of the dg categories:
\[\bb: \MF_n\to \rmdd^{\per}(\CC^n\ti \CC^n).\]

The argument of the previous proposition implies an elementary description of the above functor. The Knorrer periodicity used in the proof amounts to the pull-back along
the embedding:
\[\res:\quad\GG_n\ti \frb\ti \GG_n\ti\frb\to\frg\ti \GG_n\ti \frb\ti \GG_n\ti\frb, \]
that is induced by the embedding of the zero into the Lie algebra \(\fg_n\).

The  stability condition in the definition of \(\rmoo_n\) translates in the stability condition in the definition below:
\[\bigg(\GG_n\ti \frb\ti \GG_n\ti \frb\ti \CC^n\bigg)^{\st}=\{(g_1,Y_1,g_2,Y_2,v)|\CC[\Ad_{g_i}Y]v=\CC^n\}.\]

Respectively, the inclusion of \(0\) inside \(\frg\) combined with the projection along \(\CC^n\) gives us
map:
\[\res: \bigg(\GG_n\ti \frb\ti \GG_n\ti \frb\ti \CC^n\bigg)^{st}\to \frg\ti \GG_n\ti \frb\ti \GG_n\ti\frb. \]

There is a natural projection from \(\frb_n\) to \(\frh_n\)
that extracts the diagonal part of the matrix. In the introduction we fixed notation
\(\lambda\) for this map.

With these notations we define the \(\GG_n\ti B_n\ti B_n\)
equivariant map:
\[\pi_\frh: \bigg(\GG_n\ti \frb\ti \GG_n\ti \frb\ti \CC^n\bigg)^{st}\to \frh_n\ti\frh_n,
\quad (g_1,Y_1,g_2,Y_2,v)\mapsto (\frh(Y_1),\frh(Y_2)).\]

As we mentioned in  the introduction, the functor from above can be constructed in terms of \(\res\) and \(\pi_\frh\) as follows.
For any \(\calC\in \MF_n\) we have
  \[\bb(\calC)=\CE_{\frn^2}(\pi_{\frh*}\circ\res^*(\calC))^{T^2\ti\GG_n}.\]

  From now on we use the last formula as definition of the functor and in the rest of the paper we do not use
  any TQFT results from \cite{OblomkovRozansky18b}.

\subsection{Monoidal properties}
\label{sec:monoidal-properties}

Our motivational construction of functor \(\bb\) is strongly suggests that the functor  is monoidal. The monoidal structure on \(\rmD^{\per}(\CC^n\ti\CC^n)\)
is an extension of the monoidal structure on the \(R_n=\CC[x_1,\dots,x_n]\)-bimodules:
\[B_1\star B_2=B_1\ot_{R_n}B_2.\]
We show that this monoidal structure is compatible with our functor \(\bb\).

\begin{proposition}\label{prop:monoid}
  For any \(\calC_1,\calC_2\in \MF_n\)
  we have:
  \[\bb(\calC_1\star\calC_2)=\bb(\calC_1)\star\bb(\calC_2).\]
\end{proposition}
\begin{proof}
  Let us recall that the convolution for the matrix factorizations is defined with push-pull construction
  that evolves the space:
  \[\calX_{con}=\frg\ti (G_n\ti \frb)^3\]
  and the \(G_n\ti B^3\)-equivariant projections
  \[\pi_{ij}:\calX_{con}\to \calX_n.\]
  The convolution product of two matrix factorizations \(\calC_1\) and \(\calC_2\) is defined as
  \[\calC_1\star\calC_2=\pi_{13*}(\CE_{\frn^{(2)}}(\pi_{12}^*(\calC_1)\ot\pi_{23}^*(\calC_2))^{T^{(2)}}),\]
   where \(\frn^{(i)}\) and \(T^{(i)}\) stands for the \(i\)-th copy of the corresponding group in inside \(B^3\).

  Thus the object \(\bb(\calC_1\star\calC_2)\) can be described in terms of the open piece inside
  \(\calX_{con}\ti \CC^n\) with the following stability conditions:
  \[\bigg(\calX_{con}\ti \CC^n\bigg)^{\sst}_{(1),(3)}=\{(X,g_1,Y_1,g_2,Y_2,g_3,Y_3,v)| \CC[\Ad_{g_i}Y_i]v=\CC^n,i=1,3\}.\]

  We can naturally extend the maps \(\pi_{ij}\) to the maps between \(\calX_{con}\ti \CC^n\) and \(\calX_n\ti \CC^n\).
  After the extension we have:
  \[\pi_{13}^{-1}\Bigg(\res^{-1} \Bigg(\bigg( (\GG_n\ti \frb)^2\ti \CC^n\bigg)^{\st}\Bigg)\Bigg)=\bigg(\calX_{con}\ti \CC^n\bigg)^{\sst}_{(1),(3)}.\]

  The element \(\bb(\calC_1\star\calC_2)\) can be rewritten as result of application of the
  functor \(\CE_{\frn^2}(\cdot)^{T^2}\) to the pull-back:
    \[\res^*\bigg(\pi_{13*}\CE_{\frn^{(2)}}\bigg( (\pi_{12}^*( \calC_1)\ot\pi_{23}^*(\calC_2))_{(1),(3)}\bigg)^{T^{(2)}}\bigg),\]
    where the subscript \((1),(3)\) indicates that we work with the matrix factorization on \((\calX_{con}\ti\CC^n)^{\sst}_{(1),(3)}\).
    Next we observe that the \(\pi_{12}\) and \(\pi_{23}\) images of the strongly stable spaces also has
    a stability-condition description:
    \[\pi_{12}\bigg((\calX_{con}\ti \CC^n)^{\sst}_{(1),(3)}\bigg)=(\calX_n\ti \CC^n)_{(1)}^{\sst},\quad
      \pi_{23}\bigg((\calX_{con}\ti \CC^n)^{\sst}_{(1),(3)}\bigg)=(\calX_n\ti \CC^n)_{(2)}^{\sst},\]
    \[(\calX_n\ti\CC^n)^{\sst}_{(i)}=\{(X,g_1,Y_1,g_2,Y_2,v)|\CC[\Ad_{g_i}Y_i]v=\CC^n.\}\]

    On the critical locus of \(W\) we have \(\Ad_{g_1}Y_1=\Ad_{g_2}Y_2\). Thus by the shrinking lemma  in \cite{OblomkovRozansky16}
    the restriction from \((\calX_n\ti\CC^n)_{(i)}^{\sst}\) to
    \[\bigg(\calX_n\ti \CC^n\bigg)^{\sst}_{(1),(2)}=\{(X,g_1,Y_1,g_2,Y_2,v)|\CC[\Ad_{g_i}Y_i]v=\CC^n,i=1,2\},\]
    induces the equivalence of the categories.

  Thus  the element \(\bb(\calC_1\star\calC_2)\) can be rewritten as result of application of the
  functor \(\CE_{\frn^2}(\cdot)^{T^2}\) to the pull-back:
    \[\res^*\bigg(\pi_{13*}\CE_{\frn^{(2)}}\bigg( (\pi_{12}^*( \calC_1)_{(1),(2)}\ot\pi_{23}^*(\calC_2)_{(1),(2)})_{(1),(2),(3)}\bigg)^{T^{(2)}}\bigg),\]
    where the subscript \((1),(2)\) indicates that we apply
    the pull-back \(\pi_{ij }\) to the the matrix factorization on \((\calX_n\ti\CC^n)^{\sst}_{(1),(2)}\). Respectively,
    the subscript \((1),(2),(3)\) indicates that we work with the matrix factorizations on
    \[\bigg(\calX_{con}\ti \CC^n\bigg)^{\sst}_{(1),(2),(3)}=\{(X,g_1,Y_1,g_2,Y_2,g_3,Y_3,v)| \CC[\Ad_{g_i}Y_i]v=\CC^n,i=1,2,3\}.\]

    Finally, let us observe that \(B^{(2)}\)-action on \((\calX_{con}\ti \CC^n)_{(1),(2),(3)}^{\sst}\) is free. Hence on this
    space the functor \(\CE_{\frn^{(2)}}(\cdot)^{T^{(2)}}\) is equivalent to the functor of restriction to the slice
    to \(B^{(2)}\)-action:
    \[\bigg(\calX_{con}\ti \CC^n\bigg)_{(1),\varkappa,(3)}^{\sst}\subset  \bigg(\calX_{con}\ti \CC^n\bigg)_{(1),(2),(3)}^{\sst}\]
    that is defined by the condition:
    \[(Y_2,g_2v)\in \varkappa(\CC^n),\]
    where \(\varkappa\) is defined by \eqref{eq:kappa}.

    The  last observation implies that the element \(\bb(\calC_1\star\calC_2)\) can be rewritten as result of application of the
  functor \(\CE_{\frn^2}(\cdot)^{T^2}\) to the pull-back:
    \[\res^*\bigg(\pi_{13*}\bigg( \pi_{12}^*( \calC_1)_{(1),\varkappa}\ot\pi_{23}^*(\calC_2)_{\varkappa,(2)}\bigg)_{(1),\varkappa,(3)}\bigg),\]
    where the subscripts \((1),\varkappa\) and \(\varkappa,(2)\) indicate that we apply the pull-back to the
    matrix factorizations on the spaces
    \[\bigg(\calX_2\ti \CC^n\bigg)^{\sst}_{(1),\varkappa},\quad \bigg(\calX_2\ti\CC^n\bigg)^{\sst}_{\varkappa,(2)}\]
    which are the images of \((\calX_{con}\ti \CC^n)_{(1),\varkappa,(3)}^{\sst}\) under projections \(\pi_{12}\) and \(\pi_{23}\), respectively.

    Since the spaces \((\calX_n\ti \CC^n)^{\sst}_{(1),\varkappa}\) and \((\calX\ti \CC^n)^{\sst}_{\varkappa,(2)}\) are slices to \(B\)-actions, we
    have:
    \[\bb(\calC_1)=\CE_{\frn}\left(\res^*(\calC_1)_{(1),\varkappa}\right)^T,\quad \bb(\calC_2)=\CE_{\frn}\left(\res^*(\calC_2)_{\varkappa,(2)}\right)^T,\]
    where the subindices \((1),\varkappa\) and \(\varkappa,(2)\) indicate the restriction to the \(B\)-slices
   \((\calX_n\ti \CC^n)^{\sst}_{(1),\varkappa}\) and \((\calX\ti \CC^n)^{\sst}_{\varkappa,(2)}\).

   Thus to complete our proof we observe that
   \begin{multline*}
     \CE_{\frn^2}\Bigg(\res^*\bigg(\pi_{13*}\bigg( \pi_{12}^*( \calC_1)_{(1),\varkappa}\ot\pi_{23}^*(\calC_2)_{\varkappa,(2)}\bigg)_{(1),\varkappa,(3)}\bigg)\Bigg)^{T^2}=\\
     \pi_{13*}\Bigg(\pi_{12}^*\bigg(\CE_{\frn}(\res^*(\calC_1)_{(1),\varkappa})^T\bigg)\ot\pi_{23}^*\bigg(\CE_{\frn}(\res^*(\calC_2)_{\varkappa,(2)})^T\bigg)\Bigg).
    \end{multline*}

\end{proof}

In the rest of the text we work with the stable 
of the category that is defined to be
\[\MF_n^{\st}=\MF_{\GG_n\times B^2}(\calX_n^{\st},W), \quad\calX_n^{\st}\xhookrightarrow{j_{\st}}\calX_n=\fg_n\times (\GG_n\times \frb)^2\times V_n,\]
where \((X,g_1,Y_1,g_2,Y_2,v)\) is stable if \[\CC^n=\CC\langle X,\Ad_{g_i}Y_i\rangle v.\]

It is shown in \cite{OblomkovRozansky16} that the pull-back map \(j^*_{\st}\) is monoidal.

The functor \(\bb\) naturally extends to \(\MF_n^{\st}\) and  the previous proposition holds for the functor:
\[\bb:\MF_n^{\st}\to\mathrm{D}^{\per}(\CC^n\ti\CC^n).\]
Moreover, since the stability condition restrict naturally:
\[\res^{-1}(\calX_n^{\st})=\bigg(\GG_n\ti\frb\ti\GG_n\ti\frb\ti \CC^n\bigg)^{\st}\]
we have
\begin{corollary}
  For any \(n\) there is a commuting diagram of monoidal functors:
  \[\begin{tikzcd}
      \MF_n\arrow[r,"\bb"]\arrow[rd,"j^*_{\st}"'] & \mathrm{D}^{\per}(\CC^n\ti\CC^n)\\
      &\MF_n^{\st}\arrow[u,"\bb"]
    \end{tikzcd}
  \]
\end{corollary}

\subsection{Knorrer reduction for $\mathsf{B}$}
\label{sec:knorr-reduct-mathbbb}

In this section we explain how the Knorrer periodicity allows us to simplify the functor \(\bb\). Computationally, it is easier to
work with the simplified form of the functor.

The  \(\GG_n\) action on the space \(\calX\) is free. We choose a slice to this action  as
\[\calX^{\circ}=\fg_n\ti \frb_n\ti \GG_n\ti \frb_n \]
that is the image of \(B^2\)-equivariant projection
\begin{equation}\label{eq:proj-circ}
  \calX\to \calX^\circ,\quad (X,g_1,Y_1,g_2,Y_2)\mapsto (\Ad_{g_1}^{-1}X,Y_1,g_1^{-1}g_2,Y_2).\end{equation}
Respectively, we have the induced \(B^2\)-invariant potential:
\[W^\circ(X_1,Y_1,g_{12},Y_2)=\Tr(X_1(Y_1-\Ad_{g_{12}}Y_2)).\]
and we have an equivalence of categories:
\[\MF_n\simeq \MF_n^\circ=\MF_{B^2}(\calX^\circ,W^\circ).\]

Next we observe that the potential \(W^\circ\) has a quadratic summand:
\[W^\circ=\overline{W}^\circ+W_{kn}, \quad W_{kn}(X_1,Y_1,g_{12},Y_2)=\Tr((X_1)_{--}((Y_1)_{++}-(\Ad_{g_{12}}Y_2)_{++})),\]
\[ \overline{W}^\circ(X_1,Y_1,g_{12},Y_2)=\Tr((X_1)_+(Y_1-\Ad_{g_{12}}Y_2)).\]

Above and everywhere in the text we use notations for the projections on the upper triangular subspaces \(\frb,\frn\) and
the lower-triangular \(\frb_-,\frn_-\):
\[X=X_++X_{--},\quad X_+\in \frb,\quad X_{--}\in \frn_-,\]
\[X=X_{++}+X_-,\quad X_{++}\in \frn,\quad X_-\in \frb_{-}.\]

The  potential is \(W_{kn}\) is quadratic of \((X_1)_{--}\) and \( (Y_1)_{++}-(\Ad_{g_{12}}Y_2)_{++}\) and the embedding:
\[j_{kn}:\quad \frb\ti \frb\ti \GG_n\ti\frb\to \fg_n\ti\frb\ti\GG_n\ti \frb,\]
is \(B^2\)-equivariant and regular. Also the  \(B^2\)-equivariant projection
\[\pi_{kn}:\quad \frb\ti \frb\ti \GG_n\ti\frb\to\frb\ti \frh\ti\GG_n\ti \frb=\overline{\calX}^\circ,\]
\[\pi_{kn}:(X_1,Y_1,g_{12},Y_2)\mapsto (X_1,\lambda(Y_1),g_{12},Y_2)\]
is smooth with fibers affine spaces that have the coordinates \((Y_1)_{++}-(\Ad_{g_{12}}Y_2)_{++}\).
Hence Knorrer functor
\[KN=j_{kn*}\circ\pi_{kn}^*: \quad\overline{\MF}_n^\circ=\MF_{B^2}(\overline{\calX}^\circ,\overline{W}^\circ)\to \MF_n, \]
is an equivalence of categories. The inverse functor \(\pi_{kn*}\circ j_{kn}^*\) is conjugate to \(KN\).

Thus we want to construct a version of the functor \(\bb\) for the reduced category
\[\overline{\bb}=\bb\circ KN: \quad \overline{\MF}_n^\circ\to \rmD^{\per}(\CC^n\ti \CC^n).\]

First let us simplify the construction of the functor \(\bb\) by eliminating the \ChE step. That
could be achieved with use of the slices \(\CC^n_\varkappa\) for \(B\)-action on the space \(\fg_n\ti \CC^n\):
\[\CC^n_\varkappa=\varkappa(\CC^n)\subset \fg_n\ti \CC^n.\]

Thus we  can define \(B^2\)-slice in the strongly stable locus:
\[\bigg(\calX^\circ\ti \CC^n\bigg)^{\sst}_{\varkappa,\varkappa}=\{(X_1,Y_1,g_{12},Y_2,v)|(Y_1,v)\in \CC^n_\varkappa, (Y_2,g_{12}^{-1}v)\in \CC^n_\varkappa\}.\]

Let us denote by \(i_{\varkappa,\varkappa}\) the embedding of last space into  \(\calX^\circ\ti \CC^n\). Thus we have a simplified
formula for the functor:
\[\bb=i^*_{\varkappa,\varkappa}\circ \res^*\circ \pi_{\frh,*}.\]

The maps underlying the Kn\"orrer functor and the functor \(\bb\) fits into the commutative diagram:
\[\begin{tikzcd}
  &\calX^\circ_n\ti \CC^n& \arrow[l,"i_{\varkappa,\varkappa}"](\calX^\circ_n\ti \CC^n)^{\sst}_{\varkappa,\varkappa}& (\frb\ti \GG_n\ti \frb\ti \CC^n)_{\varkappa,\varkappa}^{\st}\arrow[r,"\pi_{\frh}"]\arrow[l,"\res"] & \frh^2\\
  \overline{\calX}^\circ_n & \frb^2\ti\GG_n\ti\frb\ti\CC^n\arrow[u,"j_{kn}"]\arrow[l,"\pi_{kn}"] &
  & (\frh\ti\GG_n\ti\frb\ti\CC^n)_{\varkappa,\varkappa}^{\st}\arrow[u,"\bar{j}"]\arrow[ll,"\tilde{i}_{\varkappa,\varkappa}"] \arrow[ur,"\pi_\frh"']&
  \end{tikzcd}
\]
In the last diagram  we set
\[(\frb\ti\GG_n\ti\frb\ti \CC^n)^{\st}_{\varkappa,\varkappa}=\{(Y_1,g,Y_2,v)|(Y_1,v)\in\CC^n_\varkappa,(Y_2,g_2^{-1}v)\in\CC^n_\varkappa\}\]
the    the closed embedding
\[\bar{j}:\quad (\frh\ti\GG_n\ti\frb\ti \CC^n)^{\st}_{\varkappa,\varkappa}\to (\frb\ti\GG_n\ti\frb\ti \CC^n)^{\st}_{\varkappa,\varkappa}\]
is defined by the equation:
\[\bar{j}(h,g,Y_2,v)=(h+(\Ad_gY_2)_{++},g,Y_2,v).\]
The inclusion  \(\tilde{i}_{\varkappa,\varkappa}\) is uniquely defined by the commutativity of the diagram.

Now we observe that the composition of the maps:
\[\bar{i}_{\varkappa,\varkappa}=\pi_{kn}\ti \mathrm{Id}_{\CC^n}\circ \tilde{i}_{\varkappa,\varkappa}:\quad   (\frh\ti\GG_n\ti\frb\ti\CC^n)_{\varkappa,\varkappa}^{\st}\to \overline{\calX}_n^\circ\ti \CC^n,\]
defines  a natural embedding of  \((\frh\ti\GG_n\ti\frb\ti\CC^n)_{\varkappa,\varkappa}^{\st}\)
into \(\overline{\calX}_n^\circ\). That is the
map \(\tilde{i}_{\varkappa,\varkappa}\) is the section of the projection \(\pi_{kn}\ti \mathrm{Id}_{\CC^n}\).

The image of \(j_{kn}\circ\tilde{i}_{\varkappa,\varkappa}\) consists of \((X_1,Y_1,g,Y_2,v)\in (\calX^\circ_n\ti \CC^n)^{\sst}_{\varkappa,\varkappa}\)
in the vanishing locus of the ideal with generators
\[(X_1)_{--},\quad (Y_1)_{++}-(\Ad_gY_2)_{++}.\]

Since  the push-forward \(j_{kn*}\) is defined in terms of Koszul matrix factorization of the last two group of elements, we conclude that
\[i_{\varkappa,\varkappa}^*\circ j_{kn*}\circ \pi_{kn}^*=\bar{j}_*\circ \tilde{i}_{\varkappa,\varkappa}^*\circ \pi_{kn}^*.\]

Hence we obtain a simplified formula for the functor:
\[\overline{\bb}=\pi_{\frh*}\circ\bar{i}_{\varkappa,\varkappa}^*.\]

The further simplification comes from the following observation:

\begin{proposition}\label{prop:simple-slice} For any \(n\) we have:
\[(\frh\ti\GG_n\ti\frb\ti\CC^n)_{\varkappa,\varkappa}^{\st}=\CC^n\ti U_-\ti \CC^n\]
  where \(U_-\) is the group of strictly-lower triangular matrices.
\end{proposition}
\begin{proof}
  We need to show that if \[(Y_1,v),(Y_2,g^{-1}v)\in \CC^n_\varkappa\quad \mbox{ and }\quad
Y_1-\Ad_gY_2\in \frb_-\]
then \(g\in U_-\).

By the first assumption we have:
\[v=e_n,\quad Y_i=\frh(Y_i)+J,\quad g_{in}=\delta_n^i.\]
where \(J\) is the Jordan block of size \(n\).

The second assumption defines subgroup \(G'\subset \GG_n\).
The Lie algebra \(\frg'\) of this group is defined by
\[\ad_\alpha Y_2\in \frb_-.\]

An element \(\alpha\in \frg'\) uniquely decomposes into the sum
\[\alpha=\alpha_{--}+\alpha_+,\quad \alpha_+\in \frb, \quad \alpha_{--}\in \frn_{--}.\]
The defining equation of \(\frg'\) only constraints \(\alpha_+\) since
\[\ad_{\alpha_{--}}(\frh(Y_2)+J)\in \frb_-.\]

Moreover, the condition on \(\alpha_+\) is equivalent to
\[\ad_{\alpha_+}(J)=0\]
Hence \(\alpha_+\) is a polynomial of \(J\):
\[\alpha_+=\sum_{i=0}^{n-1}c_i J^i.\]
However, by the first assumption \((\alpha)_{in}=0\), \(i<n\) and thus \(c_i=0\), \(i>0\) and we have shown
that \(\frg'=\frn_-\).
\end{proof}

Let us summarize the discussion of the section by reminding that the embedding
\[j':\quad \frb\ti\frh\ti\GG_n\ti\frb=\overline{\calX}^\circ_n\to\calX_n,\quad j'(X_1,Y_1,g,Y_2)\mapsto (X_1,Y_1+(\Ad_gY_2)_{++},g,Y_2),\]
and the potential \(\overline{W}^\circ\) on \(\overline{\calX}^\circ_n\) is obtained from
\(W \) by pull-back along \(j'\).
The inverse to the Kn\"orrer equivalence is given by:
\[KN^{-1}=\pi_{kn*}\circ j_{kn}^*: \quad \MF_n\to \overline{\MF}^\circ_n=\MF_{B^2}(\overline{\calX}^\circ,\overline{W}^\circ).\]

As we have shown the reduced functor \(\overline{\bb}\), which is obtain by pre-composing \(\bb\) with \(KN\) has  the following
simple decryption:
\begin{theorem}\label{thm:simple}
  The reduced functor
  \[ \overline{\bb}:\quad \overline{\MF}_n^\circ=\MF_{B^2}(\overline{\calX}^\circ_n,\overline{W}^\circ)\to D(\CC^n\ti \CC^n),\]
  is defined by
  \[\overline{\bb}=\pi_{\frh*}\circ\bar{i}^*_{\varkappa,\varkappa} \]
  where
  \[\bar{i}_{\varkappa,\varkappa}:\quad \CC^n\ti U_-\ti \CC^n\to \overline{\calX}^\circ_n, \quad \bar{i}_{\varkappa,\varkappa}(\vec{y}_1,g,\vec{y}_2)=(0,\frh(\vec{y_1}),g,\frh(\vec{y}_2)+J),\]
  and \(\pi_{\frh}\) is the projection to \(\CC^n\ti\CC^n\).
\end{theorem}

\subsection{Unit element}
\label{sec:unit-el}

In this section we apply the results of the previous section and check that \(\bb\) send the convolution unit to the unit.
Let us first recall that \(R_n\) has a natural structure of \(R_n\) bimodule and is naturally a unit of the
bimodule monoidal structure.

To describe the unit on the geometric side we introduce \(B^2\)-equivariant space
\[\overline{\calX}^\circ_n(B_n)=\frb\ti \frh\ti B_n\ti \frb,\quad \bar{j}_B:\overline{\calX}^\circ_n(B_n)\to \overline{\calX}^\circ_n,\]
where the last map is induced by the natural inclusion of \(B_n\) inside \(\GG_n\).
The pull-back of the potential \(\overline{W}^\circ\) is quadratic:
\[\bar{j}_B^*(\overline{W}^\circ)=\sum_{i=1}^n X_{1,ii}(Y_{1,ii}-Y_{2,ii}).\]
In particular, the category \(\MF_{B^2}(\overline{\calX}_n^\circ(B_n), \bar{j}_B^*(\overline{W}^\circ))\)
contains a canonical Koszul matrix factorization:
\[K_\Delta=\begin{bmatrix} X_{1,11} & Y_{1,11}-Y_{2,11}\\ \vdots &\vdots \\
    X_{1,nn}& Y_{1,nn}-Y_{2,nn}
  \end{bmatrix}
\]

Respectively, the unit in the monoidal  category  \(\MF_n\) is defined by :
\[\calC_{\parallel}=KN(\overline{\calC}_\parallel)\in\MF_n,\quad \overline{\calC}_\parallel=j_{B*}(K_\Delta) \]

\begin{proposition}\label{prop:unit-br}
  For any \(n\) we have
  \[\bb(\calC_\parallel)=R_n.\]
\end{proposition}
\begin{proof}
  By above discussion it is enough to show that
  \[\overline{\bb}(\overline{\calC}_\parallel)=R_n.\]
  The diagram of  maps that participate the construction of \(\overline{\bb}\) could be
  completed to the commuting diagram:
  \[
    \begin{tikzcd}
      \overline{\calX}_n^\circ(B_n)\arrow[r,hook,"\bar{j}_B"]& \overline{\calX}_n^\circ&\\
      \CC^n\ti\CC^n\arrow[u,"\tilde{i}_{\varkappa,\varkappa}"]\arrow[r,"\tilde{j}_B"]&\arrow[u,"\bar{i}_{\varkappa,\varkappa}"]\CC^n\ti U_-\ti \CC^n
      \arrow[r,"\pi_\frh"]& \CC^n\ti \CC^n.
    \end{tikzcd}
  \]
  where
  \[\tilde{i}_{\varkappa,\varkappa}(Y_1,Y_2)=(0,Y_1,1,Y_2+J_n).\]
  The map \(\tilde{j}_B\) is uniquely determined by the commutativity of the diagram and
  \[\pi_\frh\circ \tilde{j}_B=id.\]
  Thus we can apply the base change formula to obtain
  \[\overline{\bb}(\overline{\calC}_\parallel)=\pi_{\frh*}\circ\bar{i}^*_{\varkappa,\varkappa}\circ \bar{j}_{B*}(K_\Delta)=
    \pi_{\frh*}\circ \tilde{j}_{B*}\circ \tilde{i}_{\varkappa,\varkappa}^*(K_\Delta)=\tilde{i}_{\varkappa,\varkappa}^*(K_\Delta).\]
  Since \(K_\Delta|_{X=0}\)  is the Koszul complex of the diagonal in \(\CC^n\ti \CC^n\) the statement follows.
\end{proof}

\subsection{Elementary braid diagrams}
\label{sec:elem-braid-diagr}

For our argument we would need to calculate the functor \(\bb\) on the matrix factorizations that represent the braid diagrams.

Let us recall that the basic building blocks for the Soergel bimodules are
the singular Soergel bimodules:
\[B_i=R_n\otimes_{R_n^i} R_n,\]
where \(R_n=\CC[y_1,\dots,y_n]\) and \(R_n^i=R_n^{(i,i+1)}\) is the subring
of polynomials invariant with respect to switching \(y_i\) and \(y_{i+1}\).
The bimodule \(R_n\) is the unit of the convolution.

In this section we would like to discuss the case \(n=2\) for both
Soergel bimodules and for matrix factorizations.

The analog of the bimodule \(B_1\in \sbim_2\) is the element \(\calC_\bullet\) of the category
\(\MF_n\) which is easier to describe in the different model \(\overline{\MF}_n\) of the category
\(\MF_n\). The category \(\overline{\MF}_n\) is defined as category of matrix factorizations
on the space \[\overline{\calX}_n=\frn\times \GG_n\times \frb,\] with the potential
\[\overline{W}(X,g,Y)=\Tr(X\Ad_gY).\]

The  left and right \(B\)-actions on \(G\) could be extended in a unique to preserve the potential.
Respectively, we use the following notations for the \(B^2\)-equivariant category of matrix factorizations:
\[\overline{\MF}_n=\MF_{B^2}(\overline{\calX}_n,\overline{W}).\]

As it is shown in \cite{OblomkovRozansky19a} there is a natural Kn\"orrer equivalence functor:
\[KN: \overline{\MF}_n\to \MF_n.\]

If \(n=2\) the potential \(\overline{W}\) factors:
\[\overline{W}(X,g,Y)=x_{12}g_{21}\tilde{y}\det(g)^{-1},\quad \tilde{y}=g_{22}(y_{11}-y_{22})-g_{21}y_{12}.\]
Thus we can define three  Koszul matrix factorizations:
\[\overline{\calC}_\parallel=[g_{21},x_{12}\tilde{y}/\det(g)],\quad \overline{\calC}_\bullet=[x_{12}g_{21},\tilde{y}/\det(g)],\quad
\overline{\calC}_+=[x_{12},g_{12}\tilde{y}/\det(g)].\]

Using the Kn\"orrer functor we get the elements:
\[\calC_\parallel=KN(\overline{\calC}_\parallel),\quad \calC_\bullet=KN(\overline{\calC}_\bullet),\quad \calC_+=KN(\overline{\calC}_+).\]

Below we show the basic case of our main result:

\begin{proposition}\label{prop:el-braid}
  If \(n=2\) then
  \[\bb(\calC_\bullet)=B_1.\]
\end{proposition}

\begin{proof}
  Since the \(\GG_n\)-action on the space \(\calX^{\st}\) is free it is more convenient to
  work with the \(\GG_n\)-quotient  space with the potential
  \[\calX^{\circ,\st}\subset\fg_n\times\frb\times G\times \frb\times V,\quad W^\circ(X,Y^1,g_{12},Y^2,v)=
    \Tr(X(\Ad_{g_{12}}Y_1-Y_2)),\]
  and the stability condition \(\langle X,Y_2\rangle v=\CC^n\), \(\langle X,\Ad_{g_{12}}Y_1\rangle v=\CC^n\) and \(B^2\) action:
  \[(b_1,b_2)\cdot v= b_2 v,\quad  (b_1,b_2)\cdot Y_i=\Ad_{b_i}Y_i,\quad (b_1,b_2)\cdot X=\Ad_{b_2}X,\quad (b_1,b_2)\cdot g_{12}=b_2g_{12}b_1^{-1}.\]

  To distinguish two \(B\)-actions we write \(B^2=B^{(1)}\times B^{(2)}\).

  We can use the \(B^{(2)}\)-action on \(V\) to put the vector \(v\) in the standard position
  \[v_0=\begin{bmatrix} 0\\1\end{bmatrix}.\]

  Let \(G'\subset G\) be a open \(B^{(1)}\)-equivariant locus inside the group:
  \[G'=\{\begin{bmatrix} a&b\\c&d\end{bmatrix}| a\ne 0\}.\]

  The pull-back along the inclusion map \[j':\fg_n\times\frb\times G'\times\frb\times v_0\to\fg_n\times\frb\times G\times\frb\times v_0\]
  is induces the equivalence of categories of \(B^{(1)}\)-equivariant matrix factorizations.

  Indeed, if \(g\in G\setminus G'\) then for any \(Y\in \frb\) there is vanishing of one of the matrix
  coefficients of the conjugate:
  \[\big(\Ad_g(Y)\big)_{12}=0.\]
  On the other hand on the critical locus of \(W^\circ\) we have
  \[\Ad_{g_{12}}Y_1=Y_2.\]
  Thus if \(g_{12}\in G\setminus G'\) the critical relations imply that \(Y_2\) is diagonal. On the other hand \(v=v_0\)
  thus \(\CC[Y_2]v\ne \CC^2\) and we get a contradiction to the strong stability.

  The \(B^{(1)}\)-action and \(T^{(2)}\)-actions  allows us to move the point of \(\calX^{\st}\) to the standard position:
  \[g_{12}=\begin{bmatrix} a_{11}&0\\a_{21}&1\end{bmatrix},\quad Y_i=\begin{bmatrix}y^{(i)}_{11}& 1\\0&y^{(i)}_{22}\end{bmatrix},\quad v=v_0\]
  This closed locus has a trivial stabilizer inside \(B^2\).

  By direct computation we obtain:
  \[\Ad_{g_{12}}Y_1=\begin{bmatrix}
      y_{11}^{(1)}-a_{21}& a_{11}\\(y^{(1)}_{11}-y^{(1)}_{22}-a_{21})a_{21}/a_{11}&a_{21}+y_{22}^{(1)},
    \end{bmatrix}
  \]
  respectively the potential is equal:
  \begin{multline*}W^\circ=x_{11}(y^{(1)}_{11}-y^{(2)}_{11}-a_{21})+x_{21}(a_{11}-1)+x_{22}(y^{(1)}_{22}-y^{(2)}_{22}+a_{21})\\
   + x_{12} a_{21 }(y^{(1)}_{11}-y^{(1)}_{22}-a_{21}).
  \end{multline*}

  The first line of the formula for the potential is exactly the quadratic term that is responsible for the the Kn\"orrer functor
  \(KN\). Thus obtain a presentation of the matrix factorizations as curved Koszul complexes:
  \[ \CE_{\frn^2}(\calC_\bullet)^{T^2}=Q\otimes [x_{12},a_{21}(y_{11}^{(1)}-y^{(1)}_{22}-a_{21})],\]
  \[Q= [x_{11},y^{(1)}_{11}-y^{(2)}_{11}-a_{21}]\otimes[x_{21},a_{11}-1]\otimes[x_{22},y^{(1)}_{22}-y^{(2)}_{22}+a_{21}].\]

  The functor \(\bb\) sets variables \(x_{ij}\) to zero. Thus  \(\bb(\calC_\bullet)\) is a Koszul complex for the regular sequence  \[a_{21}(y_{11}^{(1)}-y_{22}^{(1)}-a_{21}),\quad y^{(1)}_{11}-y^{(2)}_{11}-a_{21},\quad a_{11}-1,\quad y^{(1)}_{22}-y^{(2)}_{22}+a_{21}.\]
  This sequence is equivalent to the sequence
  \[(y_{22}^{(2)}-y_{22}^{(2)})(y_{11}^{(2)}-y^{(1)}_{22}),\quad y_{11}^{(1)}+y_{22}^{(1)}-y_{11}^{(2)}-y_{22}^{(2)},\quad a_{11},
    \quad y^{(1)}_{22}-y^{(2)}_{22}+a_{21}.\]

  The first two terms of the sequence define the bimodule \(B_1\) and the statement follows.
  \end{proof}

\subsection{Induction functors}
\label{sec:induction-functors}

The presentation of \(\CC^n\) as product \(\CC^k\ti\CC^{n-k}\) induces the induction functor:
\[\ind_k:\sbim_k\ti\sbim_{n-k}\to\sbim_n.\]

Similar functor was constructed for the category of matrix factorization,
we provide the details below.
In the next section we use the induction functor  to define
the generators \(\calC_\bullet^{(i)}\) of \(\MF_n^\flat\). In this section we show the compatibility of
these induction functors:

\begin{proposition}
  We have:
  \[\bb\circ \ind_k=\ind_k\circ \bb\ti \bb.\]
\end{proposition}

Before proceed with the proof let us remind the construction for the induction functors for categories \(\MF_n\).
The induction functor is easier to describe in the reduced case:
\[\overline{\ind}_k: \overline{\MF}^\circ_{k}\ti \overline{\MF}_{n-k}^\circ\to \overline{\MF}_n^\circ.\]
To define the functor we introduce the intermediate space:
\[\overline{\calX}^\circ_n(P_k)=\frb_n\ti\frh_n\ti P_k \ti \frb_n,\quad P_k\subset\GG_n.\]
where \(P_k\) the standard parabolic subgroups.

The last space has a natural \(B_n^2\)-action.
Since the standard  group homomorphisms:
\[P_k\to \GG_k\ti \GG_{n-k},\quad P_k\to \GG_n\]
are \(B^2\)-equivariant we get the corresponding \(B^2\)-maps and the induction functor:
\[\begin{tikzcd}
    \overline{\calX}_n^\circ(P_k)\arrow[d,"\bar{p}_k"]\arrow[r,"\bar{i}_k"]& \overline{\calX}_n^\circ\\
    \overline{\calX}_k^\circ\ti\overline{\calX}_{n-k}^\circ&
  \end{tikzcd},\quad \overline{\ind}_k=\bar{i}_{k*}\circ\bar{p}_k^*.
\]

To define the functor in the non-reduced version of the category we introduce the intermediate space:
\[\calX_n^\circ(P_k)\xhookrightarrow{i_k}\calX_n^\circ,\quad \calX_n^\circ(P_k)=\frp_k\ti \frb_n\ti P_k\ti \frb_n,\]
where \(\frp_k=\mathrm{Lie}(P_k)\).

Analogously to the reduced case we have \(B^2_n\)-equivariant maps \(p_k,i_k\) which allow us to define the induction
functor:
\[\begin{tikzcd}
    \calX_n^\circ(P_k)\arrow[d,"p_k"]\arrow[r,"i_k"]& \calX_n^\circ\\
    \calX_k^\circ\ti\calX_{n-k}^\circ&
  \end{tikzcd},\quad \ind_k=i_{k*}\circ p_k^*.
\]

Let us list the properties of these functors.

\begin{proposition}
  For any \(n\) and \(1\le k< n\) we have
  \begin{enumerate}
  \item \(\ind_k\circ KN_k\ti KN_{n-k}=KN_n\circ \overline{\ind}_k\).
    \item \(\ind_k(\calF\star\calF'\ti \calG\star\calG')=\ind_k(\calF\ti \calG)\star\ind_k(\calF'\ti\calF')\).
  \end{enumerate}
\end{proposition}

The second property is shown in \cite[Proposition 6.2]{OblomkovRozansky16} the first property is an easy modification of
\cite[Proposition 6.1]{OblomkovRozansky16}. It is also shown in \cite{OblomkovRozansky16} that the induction functor is
transitive:

\begin{proposition}\label{prop:ind}\cite[Proposition 6.4]{OblomkovRozansky16}
  The following functors
  \[\MF_k^\circ\ti \MF_{m-k}^\circ\ti\MF_{n-m}^\circ\to \MF_n^\circ\]
  are naturally isomorphic:
  \[\ind_k\circ \Id\ti \ind_{m-k}\simeq \ind_m\circ\ind_k\ti \Id.\]
\end{proposition}

\begin{proof}[Proof of ]
  By proposition~\ref{prop:ind}  it is enough to show
  \[\overline{\bb}\circ \overline{\ind}_k=\overline{\ind}_k\circ\overline{\bb}\ti \overline{\bb}.\]
  To prove the last statement we observe there is a commutative diagram that includes all the map that participate in both sides
  of the equation:
  \[\begin{tikzcd}[column sep=small]
      \CC^k\ti U^-_k\ti\CC^k\ti \CC^{n-k}\ti U^-_{n-k}\ti \CC^{n-k}\arrow[ddd,bend right= 80,"\bar{i}_{\varkappa,\varkappa}\ti\bar{i}_{\varkappa,\varkappa}"]\arrow[d,equal]\arrow[dr,"\tilde{i}_k"]\arrow[rr,"\pi_\frh\ti \pi_\frh"]&
     &\CC^k\ti\CC^k\ti \CC^{n-k}\ti\CC^{n-k}\arrow[d,equal]\\
      \CC^n\ti U^-_k\ti U^-_{n-k}\ti \CC^n\arrow[d,"\tilde{i}_{\varkappa,\varkappa}"]\arrow[r,"\tilde{i}_k"]&\CC^n\ti U^-_n\ti \CC^n\arrow[r,"\pi_\frh"]\arrow[d,"\bar{i}_{\varkappa,\varkappa}"]&\CC^n\ti\CC^n\\
      \overline{\calX}_n^\circ(P_k)\arrow[r,"\bar{i}_k"]\arrow[d,"\bar{p}_k"]&\overline{\calX}_n^\circ&\\
      \overline{\calX}_k^\circ\ti \overline{\calX}_{n-k}^\circ&&
    \end{tikzcd},
  \]
  where the new maps are
  \[\tilde{i}_{\varkappa,\varkappa}(Y_1,g',g'',Y_2)=(0,Y_1,g'\ti g'',Y_2+J_n),\]
  and the map \(\tilde{i}_\varkappa\) is induced by the inclusion of \(U^-_k\ti U^-_{n-k}\to U^-_n\).

  Thus we can use base-change relation to show:
  \begin{multline*}
    \overline{\bb}\circ \overline{\ind}_k=\pi_{\frh*}\circ \bar{i}_{\varkappa,\varkappa}^*\circ \bar{i}_{k*}\circ \bar{p}_k=
    \pi_{\frh*}\circ\tilde{i}_{k*}\circ \tilde{i}_{\varkappa,\varkappa}^*\circ \bar{p}_k^*=    \pi_{\frh*}\circ\tilde{i}_{k*}\circ \bar{i}_{\varkappa,\varkappa}^*\ti \bar{i}_{\varkappa,\varkappa}^*\\=\pi_{\frh*}\ti \pi_{\frh*}\circ \bar{i}_{\varkappa,\varkappa}^*\ti \bar{i}_{\varkappa,\varkappa}^*=\ind_k\circ \overline{\bb}\ti \overline{\bb}.
  \end{multline*}
\end{proof}
\subsection{Generators of $\MF_n^{\flat}$}
\label{sec:generators-mf_nflat}

Let us set as in \cite{OblomkovRozansky16}:
\[\calC_\bullet^{(i)}=\ind_i(\calC''_\parallel\ti
  \ind_2(\calC_\bullet\ti\calC'_\parallel))\in \MF_n,\]
where \(\calC'_\parallel=\calC_\parallel\in \MF_{n-i-1} \) and
\(\calC''_\parallel=\calC_\parallel\in \MF_{i-1}\).

Thus combination of  propositions~\ref{prop:ind}, \ref{prop:unit-br}
and \ref{prop:el-braid} implies

\begin{proposition}
  For any \(i\in 1,\dots,n-1\) we have:
  \[\bb(\calC^{(i)}_\bullet)=B_i.\]
\end{proposition}

\section{MOY relations for matrix factorizations}
\label{sec:relations-mfflat}

In this section we discuss the properties of the functor \(\bb\). In context of our previous work we assign
a matrix factorization to a braid graph  \(D\in\Br_n^\flat\)
on \(n\) strands.

The building blocks of our category of graphs are elementary
graphs \(D^{(i,i+1)}_\bullet\). We label the strand and show only
non-trivial part of the graph, the strands \(j\)-th strand for \(j\ne i,i+1\) go up.

\begin{wrapfigure}{r}{0.35\textwidth}
  \vspace{-50pt}
  \begin{center}
    \includegraphics[width=0.32\textwidth, height=0.32\textwidth]{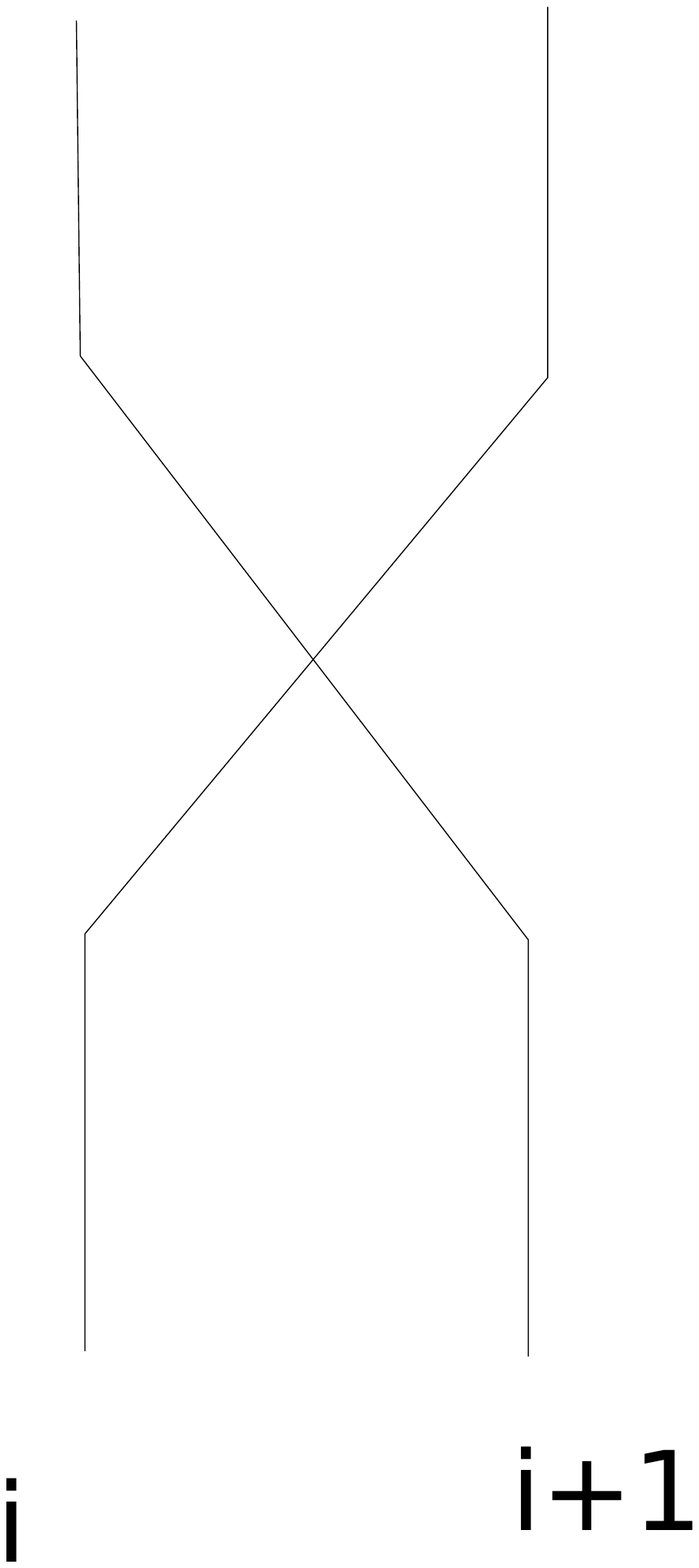}
  \end{center}
  \vspace{-30pt}
  \caption{$D^{(i,i+1)}$}
\end{wrapfigure}

We compose the graphs vertically and use notation \(\circ\) for the composition. The set of all possible
compositions of \(D_\bullet^{(i,i+1)}\), \(i=1,\dots,n-1\)
we denote by \(\mathfrak{Br}_n^\flat\).

There is a natural homomorphism
\[ \Phi^\flat:\quad \Br_n^\flat\to \MF_n, \] \[\Phi^\flat(D^{(i,i+1)}_\bullet)=\calC_\bullet^{(i)}.\]
Respectively, \(\Phi^\flat\) sends composition \(\circ\)
to \(\star\).

The graphs \(\mathfrak{Br}_n^\flat\) are particular
cases of so-called MOY graphs \cite{MurakamiOhtsukiYamada98}. To be more precise the
graphs from \(\mathfrak{Br}_n^\flat\) are braid graphs. In \cite{MurakamiOhtsukiYamada98} it is
explained how one can the braid diagram in term of the braid graphs. After resolution
the braid relations become so called MOY relations between the braid diagrams. The main result of this section is a proof of the braid MOY relations inside
\(\MF_n\).

\begin{lemma}\label{lem:MOY1}
  For any \(1\le i\le n-2\) we have:
  \[\calC_\bullet^{(i+1)}\star\calC_\bullet^{(i)}\star\calC_\bullet^{(i+1)}\oplus
    \mathbf{q}^2\calC_\bullet^{(i)}=
  \calC_\bullet^{(i)}\star\calC_\bullet^{(i+1)}\star\calC_\bullet^{(i)}\oplus
  \mathbf{q}^2\calC_\bullet^{(i+1)}\]
\end{lemma}

The skein-relation (or quadratic Hecke relation) becomes the
MOY relation depicted below and prove the relation also
holds in category \(\MF_n\);
\begin{lemma}\label{lem:MOY2}
  For any \(1\le i \le n-1\) we have
  \[\calC_\bullet^{(i)}\star \calC_\bullet^{(i)}=\calC_\bullet^{(i)} \oplus\mathbf{q}^2\calC^{(i)}_\bullet.\]
\end{lemma}

\subsection{Coherent sheaf version}
\label{sec:coher-sheaf-vers}

The categories \(\MF_n\) that we study are closely related to the categories from the work of Riche \cite{Riche08} and further
developed in the work of Bezrukavnikov and Riche \cite{BezrukavnikovRiche2012}.

The relation between the
category of Bezrukavnikov and Riche and matrix factorizations
were studied by Arkhipov and Kanstrup \cite{ArkhipovKanstrup15},
\cite{ArkhipovKanstrup15a} in the context of derived algebraic
geometry. Here we would like to present a link between the categories
studied by Bezrukavnikov-Riche and the equivariant matrix
factorization in the sense of \cite{OblomkovRozansky16}.

We use this link to give a more geometric proofs of some arguments.
It is also possible to do all proofs entirely in the matrix factorization
setting but the reference to the work of Bezrukavnikov and Riche might
be helpful for a reader familiar with the coherent sheaf version of
the construction.

The results of this section hold in the generality of
\cite{BezrukavnikovRiche2012}, that is the
group \(\GL_n\) can be replaced by any complex reductive
group \(G\) with Lie algebra \(\frg\). The category \(\MF_n\)
also can be defined in this generality.

The Grothendieck alteration \(\tilde{\frg}\to \frg\) has a quotient description \(\tilde{\frg}=G\ti\frb/B\). The Steinberg
variety \[St=\tilde{\frg}\ti_\frg\tilde{\frg}\subset \tilde{\frg}\ti\tilde{\frg}\] has a property that
the projections \(\pi_i:\tilde{\frg}\ti_\frg\tilde{\frg}\to \frg\), \(i=1,2\) has compact fibers.
Thus \(\calO_{St}\) is an element of the category \(\rmD^b_{prop}(\tilde{\frg}\ti\tilde{\frg})\).
The  last category consists of the complexes \(\calC^\bullet\) of coherent sheaves on \(\tilde{\frg}\times\tilde{\frg}\) such that the
 homology \(\mathcal{H}^*(\calC^\bullet )\) are sheaves with the support proper over each copy of \(\frg\).

Using the projections \(\pi_{ij}:\tilde{\frg}^3\to \tilde{\frg}^2\) Riche defines
the convolution product
\[\star: \rmD^b_{prop}(\tilde{\frg}^2)\ti \rmD^b_{prop}(\tilde{\frg}^2)\to \rmD^b_{prop}(\tilde{\frg}^2).\]
In \cite{Riche10} the homomorphism \(\Br_n^{aff}\to (\rmD^{b}_{prop}(\tilde{\frg}^2),\star)\) is constructed.

The  category that is most relevant for our model is the category \(T_{q}\)-equivariant  complexes
\(    \mathcal{D}_{G}^{b,T_{q}}(\tilde{\frg}^2)\)
that \(G\)-equivariant. The  torus \(T_{q}=\CC^*_q\) acts on the space \(\tfrg\) by scaling of
the \(\frg\) with weight \(\mathbf{q}\).

We can enhance the category by  enlarging the torus \(T_q\)
to \(T_{qt}=\CC_q\ti \CC_t\)
The  other factor \(\CC^*_t\) acts trivially on \(\tfrg^2\) but the
 differentials in the complexes have degree \(\mathbf{t}\).

There is a natural folding functor \(\mathrm{Per}: \rmD_{G}^{b,T_{qt}}(\tfrg^2)\to \rmD_G^{\per,T_{qt}}(\tfrg^2)\) that folds
the homological grading to the \(2\)-periodic grading. The  homological grading could be restored from the \(2\)-periodic grading
since the differential in the two-periodic complexes is of degree \(\mathbf{t}\).

Since \(\tilde{\frg}^2\) is smooth and every coherent sheaf has a finite projective resolution, this category is equivalent to the category of the strictly \(G\times B^2\)-equivariant matrix factorizations
with zero potential
\[\MF^{str}_{G \ti B^2}(\frb\ti G\ti\frb\ti G,0)\simeq \rmD^{\per,T_{qt}}_{G}(\tilde{\frg}^2). \]

An equivariant matrix factorization \((M,D,\partial)\) is strict if
\(D\) is equivariant and \(\partial=0\). Thus a pull-back of
a strictly equivariant matrix factorization is a strictly
equivariant. Similarly, the push-forward along the
smooth map and group quotient from section \ref{sec:push-forwards}
preserve strict equivariance. Thus the subcategory
\[ \MF^{str}_{G \ti B^2}(\frb\ti G\ti\frb\ti G,0)\subset
  \MF_{G \ti B^2}(\frb\ti G\ti\frb\ti G,0)\]
is a monoidal subcategory (the monoidal structure is \(\star\)).

The  pull-back along the inclusion \[i:\quad \frb\ti G\ti \frb\ti G\to \frg\ti \frb\ti G\ti \frb\ti G,\quad
i(z)=(0,z)\] annihilates the potential \(W\). Since the inclusion is clearly a regular lci inclusion
by the results from \cite{OblomkovRozansky16} there is a well-defined functor:
\[i_*: \rmD^{b,T_{q}}_G(\tilde{\frg}^2)\to \MF_n.\]

Let us also use notation \(i_*\) for the composition of the folding functor and the push-forward. If we specify the
domain of our functor there is no chance for a confusion.

The subcategory  \(\mathrm{D}^{b,T_{q}}_{G,prop}(\tilde{\frg}^2)\) that consists of the complexes with the homological supports
proper with respect to the projections \(\pi_i:\tilde{\frg}^2\to \tilde{\frg}\) is the analog of the category
\(\mathcal{D}_{prop}(\tilde{\frg}^2)\) and thus has a convolution product \(\star\).

\begin{proposition}
The push-forward functor
\(i_*\) is monoidal:
\[i_*(\calB)\star i_*(\calB')=i_*(\calB\star\calB').\]
\end{proposition}


\subsection{Outline of proof of MOY relations}
\label{sec:outline-proof}
The geometry of irreducible components of the Steinberg varieties
their behavior under convolution product
are well-studied (see for example \cite{ChrisGinzburg}).
Thus we can use the previous proposition to transfer some
of these result into matrix factorization setting and prove
the MOY relation from the beginning of the section. We outline
our approach in this section.

There are natural Koszul complexes that are sent to the generators \(\calC^{(i)}_\bullet\) by \(i_*\)
\begin{equation}\label{eq:calB}
    \calB^{(i)}_\bullet\in \rmD^{b,T_{q}}_{G,\xprop}(\tilde{\frg}^2),\quad i_*(\calB^{(i)}_\bullet)=\calC_\bullet^{(i)}.\end{equation}

The  \(\calB^{(i)}\) is a Koszul complex of the subvariety \(St^{(i)}\subset \widetilde{\fg}_n\ti\widetilde{\fg}_n \) that is a \(B^2\)-quotient of the subvariety
\(\widetilde{St^{(i)}}\). The  subvariety \(\widetilde{St^{(i)}}\) consists of quadruples \( (Y_1,g_1,Y_2,g_2)\) that satisfy
equations:
\[g_{12}=g^{-1}_1g_2\in P_i,\quad \Ad_{g_1}Y_1=\Ad_{g_2}Y_2.\]

Here and everywhere below we define a parabolic subgroup \(P_I\subset G_n\), \(I\subset\{1,\dots,n-1\}\) as a group with the
Lie algebra
\(\mathrm{Lie}(P_I)\subset \frg\) generated by \(\frb\) and \(E_{i+1,i}\), \(i\in I\).

The  variety \(\widetilde{St^{(i)}}\) is a complete intersection and the  \(G\ti B^2\)-equivariant Koszul complex provides
a resolution of its structure sheaf:
\[\calB^{(i)}_\bullet=K(\widetilde{St^{(i)}})\in \rmD_{G,\xprop}^{b,T_{q}}(\tilde{g}^2).\]

The  equation \eqref{eq:calB} is almost immediate from our construction of the induction functor. Thus we need to show two propositions which imply
the MOY relations.

\begin{proposition}\label{lem:buble}
  For \(\calB_\bullet^{(1)}\in \rmD_{G_2,prop}^{b,T_{q}}(\tilde{\frg}_2^2)\) we have:
  \[\calB_\bullet^{(1)}\star\calB_\bullet^{(1)}=\calB_\bullet^{(1)}\oplus\mathbf{q}^2\cdot\calB_\bullet^{(1)}.\]
\end{proposition}

For  the next proposition we need a slight generalization of the main generators. We define
\(\widetilde{St^{(i,i+1)}}\subset \tfrg^2\) as variety of quadruples \((g_1,Y_1,g_2,Y_2)\) that
satisfy the system of equations:
\[g_{12}=g_1^{-1}g_2\in P_{i,i+1},\quad \Ad_{g_1}Y_1=\Ad_{g_2}Y_2.\]

This variety is a complete intersection and the corresponding \(G\ti B^2\)-equivariant Koszul
complex defines the element:
\[\calB^{(i,i+1)}=K(\widetilde{St^{(i,i+1)}})\in \rmD_{G,\xprop}^{\per,T_{q}}(\tfrg^2).\]

Respectively, we have the corresponding element of the category of matrix
factorization:
\[\calC^{(i,i+1)}=i_*(\calB^{(i,i+1)}).\]

\begin{proposition}\label{lem:triple-dot}
  For \(\calC_\bullet^{(1)},\calC_\bullet^{(2)},\calC_\bullet^{(1,2)}\in \MF_3\) we have:
  \[\calC_\bullet^{(1)}\star\calC_\bullet^{(2)}\star\calC_\bullet^{(1)}=\calC_\bullet^{(1)}\oplus\calC_\bullet^{(1,2)}.\]
\end{proposition}

The  first proposition is an easy computation. The  second of the proposition could be proven by
a laborious computation but we provide a proof of the lemma that relies on a computations of the extension
groups and two lemmas.

\begin{lemma}\label{lem:short-ex}
  There is a short exact sequence of sheaves on \(\tfrg^2\):
  \[0\to \calB_\bullet^{(1)}\to \calB_\bullet^{(1)}\star\calB_\bullet^{(2)}\star\calB_\bullet^{(1)}\to \calB^{(1,2)}\to 0. \]
\end{lemma}

The push forward \(i_*\) sends the short exact sequence from the proposition to the
short exact sequence in \(\MF_3\). Thus to complete our argument we
need to compute the extension groups:


\begin{lemma}
  For  \(\calC_\bullet^{(1,2)}\in\MF_3\) we have the following vanishing of the
  extension groups:
  \[\Ext^i(\calC_\bullet^{(1,2)},\calC_\parallel)=0, \quad i>0.\]
\end{lemma}

To derive the proposition from the lemmas we need to explain the relation between the extension groups and convolution product.
It leads us to study of duality on our category.  Given two characters \(\lambda,\mu\in T^\vee\) and
a module \(M\)  with \(B^2\) action
\[(b',b'',m)\mapsto (b',b'')\cdot m, \quad b',b''\in B, m\in M,\] we denote
\(M\langle \lambda,\mu\rangle\) the module \(M\) with the twisted \(B^2\)-action:
\[(b',b'',m)\mapsto \lambda(b')\mu(b'')(b',b'')\cdot m, \quad b',b''\in B, m\in M.\]

Respectively, we use the same convention for twisting objects from \(\MF_n\) and \(\rmD_{G_n}^{b,T_{q}}(\tfrg^2)\). Let us fix
\(\rho\) for the half sum of positive roots of \(G_n\). Using these notations let us list the properties of the \(B^2\)-twisting.

\subsection{Duality}
\label{sec:duality}

In this section we describe the action of the duality functor on the categories \(\MF_n\) and \(\rmD_{G_n}(\tfrg_n^2)\).
Just as in the case  of matrix factorizations we can introduce the
an additive subcategory of \(\rmD^{\flat}\subset \rmD^{b,T_{q}}_G(\tfrg^2)\) that is spanned by all
products of elements \(\calB^{(i)}\) and by the products of the direct summands of these products.

The  both categories \(\MF_n\) and \(\rmD_{G_n}(\tfrg_n)\) have duality functor:
\[\calF\mapsto\calF^\vee=\chh\mathrm{om}(\calF,\calO).\]

The  duality preserves the objects of the both graph  categories:

\begin{proposition}\label{prop:dual}
  For any \(\calA\) from \(\MF_n^\flat\) or \(\rmD^\flat\) we have:
  \[\calA^\vee=\calA.\]
\end{proposition}

The  statement is almost immediate from the observation below
and the self-duality of the generators \(\calB^{(i)}_\bullet\)
and \(\calC^{(i)}\) which is discussed next.
\begin{proposition}
  For any \(\calA_1,\calA_2\) from \(\MF_n\) or \(\rmD_{G_n}(\tfrg_n^2)\) we have
  \[(\calA_1\star\calA_2)^\vee=\calA_1^\vee\star\calA_2^\vee.\]
\end{proposition}
\begin{proof}
  Since \(\tfrg_n\) is holomorphic symplectic it has a trivial canonical class. Hence the statement follows
  from the Serre duality:
  \[\calA_1^\vee\star\calA_2^\vee=\pi_{13*}(\pi_{12}^*(\calA_1)^\vee\ot\pi_{23}^*(\calA_2)^\vee)=\pi_{13*}(\pi_{12}^*(\calA_1)\ot\pi_{23}^*(\calA_2))^\vee,\]
 for the second equality we used the Serre duality.
\end{proof}

Let us denote the unit object in both categories as \(\mathbbm{1}\):
\[\mathbbm{1}=\calC_\parallel\in \MF_n,\quad\mathbbm{1}=\calB_\parallel\in \rmD_{G_n}(\tfrg_n^2).\]

\begin{corollary}\label{cor:cycle}
  For any \(\calA_1,\calA_2\) from \(\MF_n\) or \(\rmD_{G_n}(\tfrg_n^2)\) we have
  \[\Ext^i(\calA_1,\calA_2)=Tor_i(\calA_1\star\calA_2,\mathbbm{1}).\]
\end{corollary}

The  only thing that we check to complete a proof of proposition~\ref{prop:dual} is the action of the duality on the generators.
For  that let us fix notations for the twisting of \(B^2\)-action.

Given two characters \(\lambda,\mu\in T^\vee\) and
a module \(M\)  with \(B^2\) action
\[(b',b'',m)\mapsto (b',b'')\cdot m, \quad b',b''\in B, m\in M,\] we denote
\(M\langle \lambda,\mu\rangle\) the module \(M\) with the twisted \(B^2\)-action:
\[(b',b'',m)\mapsto \lambda(b')\mu(b'')(b',b'')\cdot m, \quad b',b''\in B, m\in M.\]

Respectively, we use the same convention for twisting objects from \(\MF_n\) and \(\rmD_{G_n}(\tfrg^2)\). Let us fix
\(\rho\) for the half sum of positive roots of \(G_n\).

\begin{proposition}For any \(\calA\) from  \(\rmD_{G_n}(\tfrg_n)\) (or \(\MF_n\)) and \(\lambda,\mu\in T^\vee\) we have:
  \begin{enumerate}
  \item \(\calA\star\mathbbm{1}\langle \lambda,\mu\rangle=\calA\langle 0,\lambda+\mu\rangle\)
  \item  \(  \mathbbm{1}\langle\lambda,\mu\rangle\star\calA=\calA\langle \lambda+\mu,0\rangle\).
 \item \(\mathbbm{1}\langle\lambda,\mu\rangle=\mathbbm{1}\langle\lambda+\mu,0\rangle=
    \mathbbm{1}\langle 0,\lambda+\mu\rangle\).
   \item \(\mathbbm{1}^\vee=\mathbbm{1}\).
  \end{enumerate}

\end{proposition}
\begin{proof}
  The three properties are proven in \cite{OblomkovRozansky16} so we explain a proof of last statement.
  Let us fix coordinates \((g_1,Y_1,g_2,Y_2)\) on \((G\ti \frb)^2\).
  Then the element \(\mathbbm{1}\) is a Koszul complex of the equations:
  \[g_2^{-1}g_1\in B,\quad (\Ad_{g_2^{-1}g_1}Y_1)_{ij}=(Y_2)_{ij},\quad i\ge j\]
  The sum of the \(B^2\) weights of the first group of  equations  is \(\langle -\rho,-\rho\rangle\)  and of the
  second is \(\langle 0,2\rho\rangle\). To compute the dual of \(\mathbbm{1}\) we need to invert the Borel weights
  in the Koszul complex, hence  by the third property
  \[\mathbbm{1}^\vee=\mathbbm{1}\langle \rho,-\rho\rangle=\mathbbm{1}.\]
\end{proof}

Finally, let us compute the dual of \(\calB_\bullet^{(i)}\in \rmD^\flat_n\). The  corresponding variety \(\widetilde{St}^{(i)} \)
can be described by two equivalent systems of equations. The first system is
\[ g_2^{-1}g_1\in P_i,\quad (\Ad_{g_2^{-1}g_1}Y_1)_{kj}=(Y_2)_{kj},\quad k\ge j, \quad (\Ad_{g_2^{-1}g_1}Y_1)_{i+1,i}=0.\]
The other system is
\[ g_1^{-1}g_2\in P_i,\quad (\Ad_{g_1^{-1}g_2}Y_2)_{kj}=(Y_1)_{kj},\quad k\ge j, \quad (\Ad_{g_1^{-1}g_2}Y_2)_{i+1,i}=0.\]

The sum of the weights of the first system of equations is \(\langle-\rho+\epsilon_{i+1}, \rho-\epsilon_{i+1}\rangle\),
respectively the sum of the weights of the equations of the second system is
\(\langle \rho-\epsilon_{i+1},-\rho+\epsilon_{i+1}\rangle\). Thus the Koszul complexes of these two systems
are dual to each other. Since these two Koszul complexes are two presentations of the element \(\calB_\bullet^{(i)}\) we
obtain:
\[(\calB_\bullet^{(i)})^\vee=\calB_\bullet^{(i)}.\]

The same argument implies the self-duality of the corresponding matrix factorization:
\[(\calC_\bullet^{(i)})^\vee=\calC_\bullet^{(i)}.\]

\subsection{Two strands}
\label{sec:two-strands}

Before we proceed with the proof of lemma~\ref{lem:buble} let us write explicitly the generators of the Koszul
complex \(\calB^{(1)}_\bullet\in \rmD_{G,prop}^{b,T_{q}}\). If the space \(\frb\ti G\ti\frb\ti G\)  has coordinates \((Y_1,g_1,Y_2,g_2)\) then
collection of equations:
\begin{equation}\label{eq:gen-B}
  (\Ad_{g_2^{-1}g_1}Y_1)_{21}=0,\quad (\Ad_{g_2^{-1}g_1}Y_1)_{ij}=(Y_2)_{ij}, \quad ij=11,22,12, \end{equation}
defines the variety \(\widetilde{St^{(1)}}\) as complete intersection. Thus these elements provide a choice of
generators for the Koszul complex \(\calB^{(1)}_\bullet\).

\begin{proof}[Proof of Lemma~\ref{lem:buble}]
  The convolution space \((\frb\ti G)^3\) has coordinates \((Y_1,g_1,Y_2,g_2,Y_3,g_3)\).
  In these coordinates the pull-back \(\pi_{12}^*(\calB^{(1)}_\bullet)\)   is a Koszul complex with generators as in
  \eqref{eq:gen-B}. For  the generators of \(\pi_{23}^*(\calB^{(1)}_\bullet)\) we choose a system of generators of
  the Koszul complex as follows:
  \begin{equation}\label{eq:gen-B23}
    (\Ad_{g_2^{-1}g_3}Y_3)_{21}=0,\quad     (\Ad_{g_2^{-1}g_3}Y_3)_{ij}=(Y_2)_{ij},\quad ij=11,22,12.
  \end{equation}
  Thus the combination of the systems \eqref{eq:gen-B} and \eqref{eq:gen-B23} is equivalent to the system:
  \begin{align*}
    \Ad_{g_2^{-1}g_1}Y_1&=\Ad_{g_2^{-1}g_3}Y_3,&\quad  (\Ad_{g_2^{-1}g_1}Y_1)_{ij}&=(Y_2)_{ij}, \quad ij=11,22,12, \\
 & &(\Ad_{g_2^{-1}g_1}Y_1)_{21}&=0.\end{align*}

  The equations in the first line of the last system allow us to eliminate variable \(Y_2\). Also the first group
  of equations in the first line is equivalent to the system
  \[\Ad_{g_1}Y_1=\Ad_{g_3}Y_3.\]
  Thus the convolution of interest is given by
  \[\CE_\frn(K[(\Ad_{g_2^{-1}g_1}Y_1)_{21}])^T\otimes \calB_\bullet^{(1)}.\]
  Finally let us observe that the element
  \[(\Ad_{g_2^{-1}g_1}Y_1)_{21}=a_{21}(a_{22}(y_{22}-y_{11})+a_{21}y_{12}),\quad g=(a_{ij})\]
  has weight \((2,0)\) with respect to the action of the \(T=\CC^*\ti \CC^*\) since the elements
  \(a_{ij}\) have weight \(e_i\).

  Since \(G/B=\mathbb{P}^1\) we need to compute the push-forward along the projection of the two step complex
  \[\pi_{\PP^1}([\calO\to\mathbf{q}^2\calO(-2)]),\quad \pi_{\PP^1}:\mathbb{P}^1\ti \frb\ti G\to \frb\ti G.\]
  Since \(\pi_{\PP^1*}(\calO)=\calO\) and \(\pi_{\PP^1*}(\calO(-2))= \calO[1]\)
  the statement follows.
\end{proof}

\subsection{Three stands}
\label{sec:three-stands}

Before we prove the lemma~\ref{lem:triple-dot} let us provide the details on the complexes that appear in the statement of the lemma.

Let us fix coordinates on the space \(\frb_3\ti G_3\ti \frb_3\ti G_3\) as \((Y_1,g_1,Y_2,g_2)\).
A minimal set of equations defining \(\widetilde{St^{(1)}}\subset \frb_3\ti G_3\ti \frb_3\ti G_3\) is
\[(g_1^{-1}g_2)_{ij},\quad ij=31,32,\quad (\Ad_{g_1}Y_1)_{21}=(\Ad_{g_2}Y_2)_{21}.\]
These equations define \(\widetilde{St^{(1)}}\) as a complete intersection.

Similarly, a minimal set of equations for \(\widetilde{St^{(2)}}\) is
\[(g_1^{-1}g_2)_{ij}=0,\quad ij=21, 23,\quad (\Ad_{g_1}Y_1)_{32}=(\Ad_{g_2}Y_2)_{32}.\]

Finally, the variety \(\widetilde{St^{(1,2)}}\) has a description as a complete intersection:
\begin{equation}\label{eq:St12}\Ad_{g_1}Y_1=\Ad_{g_2}Y_2.\end{equation}

\subsubsection{}
\label{sec:triple-double}

In the next arguments we need a convolution statements that is parallel to the lemma~\ref{lem:buble}.

\begin{lemma}\label{lem:triple-double} For \(\calB_\bullet^{(1)},\calB_\bullet^{(1,2)}\) we have
  \[\calB_\bullet^{(1)}\star \calB_\bullet^{(1,2)}=\calB_\bullet^{(1,2)}\star\calB_\bullet^{(1)}=\calB_\bullet^{(1,2)}\oplus \mathbf{q}^2\calB_\bullet^{(1,2)}.\]
\end{lemma}
\begin{proof}
  We show the last equality since the others are analogous.
  The convolution space \(\calX_{conv}\subset (\frb\ti G)^3\) has coordinates \((Y_1,g_1,Y_2,g_2,Y_3,g_3)\) with
  constraint \(g_2^{-1}g_3\in P_1\)
  In these coordinates the pull-back \(\pi_{12}^*(\calB^{(1)}_\bullet)\)   is a Koszul complex with generators as in
  \eqref{eq:St12}. For  the generators of \(\pi_{23}^*(\calB^{(1)}_\bullet)\) we choose a system of generators of
  the Koszul complex as follows:
  \begin{equation}\label{eq:gen-B23-tripl}
    (\Ad_{g_2^{-1}g_3}Y_3)_{21}=0,\quad     (\Ad_{g_2^{-1}g_3}Y_3)_{ij}=(Y_2)_{ij},\quad i\le j.
  \end{equation}
  Thus the combination of the systems \eqref{eq:St12} and \eqref{eq:gen-B23-tripl} is equivalent to the system:
  \begin{align*}
    \Ad_{g_2^{-1}g_1}Y_1&=\Ad_{g_2^{-1}g_3}Y_3,&\quad  (\Ad_{g_2^{-1}g_3}Y_3)_{ij}&=(Y_2)_{ij}, \quad ij=11,22,12, \\
 & &(\Ad_{g_2^{-1}g_3}Y_3)_{21}&=0.\end{align*}

  The equations in the first line of the last system allow us to eliminate variable \(Y_2\). Also the first group
  of equations in the first line is equivalent to the system
  \[\Ad_{g_1}Y_1=\Ad_{g_3}Y_3.\]
  Thus the convolution of interest is given by
  \[\CE_\frn(K[(\Ad_{g_2^{-1}g_3}Y_3)_{21}])^T\otimes \calB_\bullet^{(1,2)}.\]
  Finally let us observe that the element
  \[(\Ad_{g_2^{-1}g_3}Y_3)_{21}=a_{21}(a_{22}(y_{22}-y_{11})+a_{21}y_{12}),\quad g_2^{-1}g_3=(a_{ij})\]
  has weight \((2,0)\) with respect to the action of the \(T=\CC^*\ti \CC^*\) since the elements
  \(a_{ij}\) have weight \(\epsilon_i\).

  Since \(P_1/B=\mathbb{P}^1\) we need to compute the push-forward along the projection of the two step complex
  \[\pi_{\PP^1}([\calO\to\mathbf{q}^2\calO(-2)]),\quad \pi_{\PP^1}:\mathbb{P}^1\ti \frb\ti G\to \frb\ti G.\]
  Since \(\pi_{\PP^1*}(\calO)=\calO\) and \(\pi_{\PP^1*}(\calO(-2))=\calO[1]\)
  the statement follows.

\end{proof}


\begin{lemma} For \(\calB_\bullet^{(1)}\in \rmD_{G_3,prop}^{b,T_q}(\tfrg^2_3)\) we have
  \[\calB_\bullet^{(1)}\star\calB_\bullet^{(2)}=\calO_{Z_{12}},\]
  where \(Z_{12}\subset \tfrg^2\) is a complete intersection defined by
  \[(g_1^{-1}g_2)_{31}=0,\quad \Ad_{g_2}^{-1}\Ad_{g_1}(Y_1)_{ij}=0,\quad ij=32,21,31.\]
  \[(Y_2)_{ij}=(\Ad_{g_2}^{-1}\Ad_{g_1}(Y_1))_{ij},\quad ij=11,22,33,12,23,13.\]
\end{lemma}

\begin{proof}
  The statement follows immediately from the observation that the projection: \(\pi_{13}:\pi^{-1}_{12}({St}^{(1)})\cap\pi^{-1}_{23}({St}^{(2)})\to Z_{12}\) is
  an isomorphism of reduced schemes.

\end{proof}


Let us discuss a geometric description of \(Z_{12}\). It is piece of the Steinberg variety \(\widetilde{St}^{(1,2)}\).
The Steinberg variety \(St^{(1,2)}\subset \tfrg^2_3\) has six irreducible components \cite{ChrisGinzburg}:
\[St^{(1,2)}=\bigcup_{w\in S_3} \Gamma_w.\]
The generic point of \(\Gamma_w\) consists of triples \((F',F'',X)\), \(X\in \frg_3\) and \(F',F''\) are in relative position \(w\).

The small Steinberg varieties are unions of subsets of irreducible components:
\[St^{(1)}=\Gamma_1\cup \Gamma_{s_1},\quad St^{(2)}=\Gamma_1\cup \Gamma_{s_2}.\]

Let use notation \(\widetilde{\Gamma}_{w}\) for the corresponding subvariety of \((G_3\ti \frb_3)^2\).
The computation above shows that
\[Z_{12}=\Gamma_1\cup \Gamma_{s_1}\cup \Gamma_{s_2}\cup \Gamma_{s_1s_2}.\]

\begin{proof}[Proof of Lemma~\ref{lem:short-ex}]
  We need to compute the product
\[\calO_{Z_{12}}\star \calB^{(1)}_\bullet=\bigg(\calB_\bullet^{(1)}\star\calB_\bullet^{(2)}\bigg)\star\calB_\bullet^{(1)}.\]

In geometric terms, the convolution is the push-forward:
\[\pi_{13*}(\calO_{Z_{123}}),\quad Z_{123}=\pi_{12}^{-1}(Z_{12})\cap \pi_{23}^{-1}(St^{(1)})\subset \tfrg_3^3.\]

Let us present the variety \(Z_{123}\) as union of two collections of irreducible components:
\[Z_{123}=Z'_{123}\cup Z''_{123},\quad Z'_{123}=\pi_{12}^{-1}(\Gamma_1\cup \Gamma_{s_1})\cap \pi_{23}^{-1}(\Gamma_{s_1}),\]
\[Z''_{123}=\pi_{12}^{-1}(Z_{12})\cap \pi_{23}^{-1}(\Gamma_1)\bigcup \pi_{12}^{-1}(\Gamma_{s_2}\cup\Gamma_{s_1s_2})\cap \pi_{23}^{-1}(\Gamma_{s_1}).\]

The  intersection of the  graph closures \(\Gamma_{w'}\),  \(\Gamma_{w''}\) is sent by the projection
\[\mu_1:(G\ti \frb)^2\to \frb\] to the locus defined by the equations:
\[w'\cdot \lambda(Y)= w''\cdot \lambda(Y).\]

Thus the intersection  \(Z'_{123}\cap Z''_{123}\) is a divisor inside \(Z'_{123}\), the defining equation of the divisor is
\((Y^1_{22}-Y^1_{33})(Y^3_{11}-Y^3_{22})\). The vanishing locus of the  first factor of the equation is the intersection
\(\pi_{12}^{-1}(\Gamma_1\cup \Gamma_{s_1})\cap \pi_{12}^{-1}(\Gamma_{s_2}\cup \Gamma_{s_1s_2})\) and the second
defines the intersection \(\pi_{23}^{-1}(\Gamma_1)\cap\pi_{23}^{-1}(\Gamma_{s_1})\). Hence, we have the short
exact sequence:
\[0\to \mathbf{q}^4\calO_{Z'_{123}}\to \calO_{Z_{123}}\to \calO_{Z''_{123}}\to 0.\]

As a last step of the proof we need to compute the push forward \(\pi_{13*}\) of the last short exact sequence.
For  that we first notice that \[\pi_{13*}(\calO_{Z'_{123}})=\calB_\bullet^{(1)}\star \calO_{\Gamma_{s_1}}.\]
Thus the computation of the last convolution product can done in the category \(\rmD_{\GG_2}(\widetilde{\fg}_2^2)\)  since
both \(\calB_\bullet^{(1)}\) and \(\calO_{\Gamma_{s_1}}\) are induced from \(G_2\subset G_3\). The  argument analogous
to the previously discussed lemma~\ref{lem:triple-double}  yields:
\[\calB_\bullet^{(1)}\star\calO_{\Gamma_{s_1}}=\mathbf{q}^{-2}\calB_\bullet^{(1)}.\]

Next let use the following general fact about the graph closures \(\Gamma_w\). Suppose that \(w',w''\in S_n\) such that
\(\ell(w'w'')=\ell(w')+\ell(w'')\) then the projection \(\pi_{13}:\tfrg^3\to \tfrg^2\) provides an isomorphism
\[\pi_{12}^{-1}(\Gamma_{w'})\cap\pi_{23}^{-1}(\Gamma_{w''})\to \Gamma_{w'w''}.\]

Thus the fact implies that in our situation \(\pi_{13*}\) restricts to the isomorphism: \(Z''_{123}\to St.\) Hence we have
shown that \(\pi_{13*}(\calO_{Z'_{123}})=\calB_\bullet^{(1)}\) and \(\pi_{13*}(\calO_{Z''_{123}})=\calO_{St}\) and the statement of the
lemma follows.

\end{proof}

\subsubsection{Computation of the extension group}
\label{sec:comp-extens-group}

The  statements from the previous section imply the statements for the matrix factorizations, since
\[i_*(\calB^{(i)}_\bullet)=\calC_\bullet^{(i)},\quad i_*(\calB_\bullet^{(1,2)})=\calC_\bullet^{(1,2)}.\]

To show that the short exact sequence (or may be triangle since we are in the triangulated category) of
the matrix factorizations
\[  0\to \calC_\bullet^{(1)}\to \calC_\bullet^{(1)}\star\calC_\bullet^{(2)}\star\calC_\bullet^{(1)}\to \calC^{(1,2)}\to 0. \]
split  we need to compute the corresponding extension group.

\begin{proposition}
  For \(\calC_\bullet^{(1)},\calC_\bullet^{(1,2)}\in \MF_3\) we have:
  \[\Ext^{>0}(\calC^{(1,2)}_\bullet,\calC_\bullet^{(1)})=0,\quad \Ext^0(\calC^{(1,2)}_\bullet,\calC_\bullet^{(1)})=\CC[y_1,y_2,y_3]^{\oplus 2}.\]
\end{proposition}
\begin{proof}
  First we observe that \(\widetilde{St}\subset (G\ti \frb)^2\) is defined by
  \begin{equation}\label{eq:St-eq}
    \Ad_{g_1}Y_1=\Ad_{g_2}Y_2,\end{equation}
  where \((g_1,Y_1,g_2,Y_2)\) are the coordinates on \((G\ti \frb)^2\). The last system of equations is
  \(B^2\)-invariant. Thus corresponding Koszul matrix factorization:
  \[\calC^{(1,2)}=K^W[\Ad_{g_1}Y_1-\Ad_{g_2}Y_2]\]
  is self-dual and
  \[\Ext^*(\calC^{(1,2)}_\bullet,\calC^{(1)}_\bullet)=\Tor^*(\calC^{(1,2)}_\bullet,\calC^{(1)}_\bullet)=
    \Tor^*(\calC^{(1,2)}_\bullet\star\calC^{(1)}_\bullet,\calC_\parallel)=\Tor^*(\calC^{(1,2)}_\bullet,\calC_\parallel)^{\oplus 2}.\]
  The last equality follows from \ref{lem:triple-double}.

  The  last group can be computed by the standard method:
  \[\Tor^*(\calC^{(1,2)}_\bullet,\calC_\parallel)=\CE_{\frn^2}(K^W[\Ad_{g_1}Y_1-\Ad_{g_2}Y_2]\otimes K^{-W}[g_1g_2^{-1}\in B, Y_1-Y_2])^{T^2\ti G},\]
  where we used used the standard notation for the Koszul matrix factorizations on the
  space \(\frg\ti (G\ti \frb)^2\) with coordinates \((X,g_1,Y_1,g_2,Y_2)\).
  As it is explained in section 13  of \cite{OblomkovRozansky16} the last group is equal to
  the pull-back
  \[\CE_{\frn}(j_e^*(K^W[\Ad_{g_1}Y_1-\Ad_{g_2}Y_2]))^{T},\]
  where \(j_e: \frg\ti \frb\to \frg\ti (G\ti \frb)^2\)  is the \(B\)-equivariant embedding:
  \[j_e(X,Y)=(X,e,Y,e,Y).\]
  Since the pull-back \(j_e^*\) turns the system \eqref{eq:St-eq} to the trivial system we obtain
  \[\CE_{\frn}(j_e^*(K^W[\Ad_{g_1}Y_1-\Ad_{g_2}Y_2]))^{T}=\CE_{\frn}(K[X])^T=H_{Lie}^*(\frn,\CC[\frb])^T.\]

  The last group is the homology of the dg algebra \(\calA=\CC[\frb]\otimes \Lambda^*(\theta_{ij})_{i<j}\) with the differential
  \[D=\sum_{i<j}\theta_{ij}E_{ij}\cdot,\]
  where \(E_{ij}\cdot \) is the action of \(E_{ij}\) on \(\calA\). The  element \(E_{ij}\) acts of \(\CC[\frb]\) by the taking
  derivative along the corresponding vector field generated by the adjoint action of \(\frn\) on \(\frb\) and it acts
  on \(\Lambda^*[\theta_{ij}]\) by the Lie bracket. Thus elements \(\theta_{ij}\) have the \(T\)-weight \(\epsilon_i-\epsilon_j\).
  Hence all elements in the algebra \(\calA\) have the positive \(T\)-grading and thus \(T\)-invariant part of the complex is
  \(\CC[y_1,y_2,y_3]\).
\end{proof}







\section{Markov-type MOY relations}
\label{sec:markov-moves}

\subsection{Markov relations}
\label{sec:markov-relations}

In this section we prove the last two statements that are needed for the comparison of Hom spaces. To state the main result of this section
it is convenient to assemble the traces \(\Tr^i\) from the section~\ref{sec:properties-functor} into one triply graded
vector space:
\[\Tr^*(\calF)=\oplus_i \Tr^i(\calF),\quad \calF\in \MF_n^{st}.\]

Let us recall that each of the spaces \(\Tr^i(\calF)\) is a doubly graded vector space. We call these two grading \(q,t\)-gradings
and use notation \(\mathbf{q}\cdot\), \(\mathbf{t}\cdot\) for  the functors that shift corresponding grading.
The additional grading \(*\) in \(\Tr^*\)  name \(a\)-grading and use \(\mathbf{a}\cdot \) for the corresponding grading shifting
functor.

In section~\ref{sec:properties-functor} we introduce notation \(D_\bullet^{(i,i+1)}\) for the generators of free monoid of planar diagram.
We denote by \(\Br_n^\flat\) the free monoid of the diagram on \(n\) strings. Respectively, to each \(D\in\Br_n^\flat\)
we attach an element \(\calF_D\in \MF_n^\flat\). We use same notation for the pull-back of \(\calF_D\) on the stable locus.

The main statement of this section is analog of the second Markov moves for the plane diagrams:

\begin{proposition}\label{prop:Markov-2}
  Suppose \(D\in \Br_n^\flat\) is element of the sub-monoid \[\langle D_\bullet^{(1,2)},\dots, D_\bullet^{(n-2,n-1)}\rangle.\]

  Respectively, \(D'\in \Br_{n-1}^\flat\) is the diagram \(D\) with the last strand removed.
  Then we have:
  \begin{enumerate}
  \item \(\Tr^*(\calF_D)=\frac{1-\mathbf{q}^2\mathbf{a}}{1-\mathbf{q}^2}\cdot\Tr^*(\calF_{D'})\).
    \item \(\Tr^*(\calF_{D\circ D^{(n-1,n)}_\bullet})=\frac{1-\mathbf{a}}{1-\mathbf{q}^2}\cdot\Tr^*(\calF_{D'})\).
  \end{enumerate}
\end{proposition}

  \subsection{Geometric Markov moves}
\label{sec:geom-mark-moves}

Before we prove the proposition let us recall key points of the construction of \(\Tr^i\) from \cite{OblomkovRozansky16}.
  Our proof is very much analogous to the proof of the Markov move \cite{OblomkovRozansky16}.

  In \cite{OblomkovRozansky16} we work with the free Hilbert  scheme \(\Hilb_{1,n}^{free}\) instead of the usual Hilbert scheme of points
  on plane \(\Hilb_n(\CC^2)\). We define the trace in terms of two-periodic sheaves on the this free Hilbert scheme and
  in subsequent paper \cite{OblomkovRozansky18a} we explain that the trace from \cite{OblomkovRozansky16} could be interpreted
  in terms of \(\Hilb_n(\CC^2)\). We remind the details of this relation in the next section but for now let us
  recall the basic geometric constructions from \cite{OblomkovRozansky16}.

  The free nested Hilbert scheme $\FHilb_n^{free}$ is a $B\times \CC^*$-quotient of the sublocus
\[\widetilde{\FHilb}_{n}^{free}\subset \frb_n\times\frn_n\times V_n\] of the cyclic triples
$$\widetilde{\FHilb}_{n}^{free}=\{(X,Y,v)|\CC\langle X,Y\rangle v=V_n\},$$
here \(V_n=\CC^n\).
The usual nested Hilbert scheme  $\FHilb_{n}$ is the subvariety of $\FHilb^{free}_{n}$, it is defined by the condition that
 $X,Y$ commute.

\begin{remark}
  There is a bit of confusion  in the notations, what we denote here by \(\FHilb_n\)   is denoted in \cite{OblomkovRozansky16}
  by \(\Hilb_{1,n}\) and by \(\FHilb_n(\cc)\) in \cite{GorskyNegutRasmussen16}.
\end{remark}

The torus $T_{qt}=\CC^*\times\CC^*$ acts on $\FHilb_{n}^{free}$ by scaling the matrices. We denote by $\rmD^{\per}_{T_{qt}}(\FHilb_{n}^{free})$ a derived category of the two-periodic complexes of the $T_{qt}$-equivariant quasi-coherent sheaves on $\FHilb_{n}^{free}$. Let us also denote by $
\mathcal{B}^\vee$ the descent of the
trivial vector bundle $V_n$ on $\widetilde{\FHilb}^{free}_{n}$ to the quotient $\FHilb^{free}_{n}$. Respectively, \(\mathcal{B}\) stands for the dual of
\(\mathcal{B}^\vee\). Below we construct for every $\beta\in \Br_n$ an
element $$\mathbb{S}_\beta\in \rmD^{\per}_{T_{qt}}(\FHilb_{n}^{free})$$ such that space of hyper-cohomology of the complex:
$$\mathbb{H}^k(\mathbb{S}_\beta):=\mathbb{H}(\mathbb{S}_\beta\otimes \Lambda^k\mathcal{B}) $$
defines an isotopy invariant.

\begin{theorem}\cite{OblomkovRozansky16}\label{thm:mainOR} For any $\beta\in\Br_n$ the doubly graded space
$$ H^k(\beta):=\mathbb{H}^{(k+\textup{writh}(\beta)-n-1)/2}(\mathbb{S}_\beta)$$
is an isotopy invariant of the braid closure $L(\beta)$.
\end{theorem}

The variety $\widetilde{\FHilb}_{n}^{free}$ embeds inside $\overline{\calX}_n^\circ$ via a map $j_e:(X,Y,v)\rightarrow (X,diag(Y),e,Y,v)$. The diagonal copy $B=B_\Delta\hookrightarrow B^2$
respects the embedding $j_e$ and since $j_e^*(\Wr)=0$, we obtain a functor:
$$ j_e^*: \MF_{B_n^2}(\calXr^{st},\Wr)=\overline{\MF}_n^{st}\rightarrow \MF_{B_\Delta}(\widetilde{\FHilb}_{n}^{free},0).$$
Respectively, we get a geometric version of "closure of the braid" map:
$$\mathbb{L}^{\st}: \MF_{B_n^2}(\calXr^{st},\Wr)=\overline{\MF}_n^{st}\to \rmD^{\per}_{T_{qt}}(\FHilb_{n}^{free}).$$

As we explained earlier, there is an isomorphism of categories:
\[KN^\circ: \overline{\MF}_n\to \MF_n^\circ.\]
Moreover in \cite{OblomkovRozansky16} we show that the category \(\overline{\MF}_n\) has a natural monoidal structure \(\bar{\star}\)
and \(KN^\circ\) is monoidal. We also constructed the corresponding homomorphism \(\overline{\Phi}_n\) from the braid group that
fits into the commuting diagram:
\[\begin{tikzcd}
    \Br_n\arrow[r,"\overline{\Phi}_n"]\arrow[rd,"\Phi_n"']&(\overline{\MF}_n,\bar{\star})\arrow[d,"KN^\circ"]\\
    &(\MF_n^\circ,\star)
  \end{tikzcd}
\]

The main result of \cite{OblomkovRozansky16} could be restated in more  geometric term via geometric trace map:
$$ \mathcal{T}r: \Br_n\rightarrow \rmD^{\per}_{T_{qt}}(\FHilb_{n}^{free}), \quad \mathcal{T}r(\beta):=\oplus_k \mathbb{L}^{\st}(\overline{\Phi}_n(\beta)\otimes \Lambda^k\mathcal{B}).$$

The  above mentioned complex \(\bbS_\beta\)  is the complex \(\mathbb{L}^{st}(\overline{\Phi}_n(\beta))\). The differentials in
the complex \(\bbS_\beta\) are of degree \(t\) thus the differentials are invariant with respect to the anti-diagonal
torus \(T_a\).

For  any braid graph \(D\in \Br_n^\flat\) we define
\[\bbS_D=\mathbb{L}^{\st}(KN^\circ(\calF_D)),\quad \mathbb{H}^k(D)=\mathbb{H}(\bbS_D\ot \Lambda^k\calB).\]



In our proof of the Markov moves we  use nested nature of the scheme \(\FHilb_n\) to define the intermediate map:
\begin{equation}\label{eq:pi-proj}
\pi: \FHilb^{free}_n\to \cc\times \FHilb^{free}_{n-1},
\end{equation}
where the first component of the map \(\pi\) is \(x_{11}\) and the second component is just forgetting of the first
rows and rows of the matrices \(X,Y\) and the first component of the vector \(v\). Let us also fix notation for the line bundles on
 \(\FHilb_n^{free}\): we denote by \(\calO_k(-1)\) the line bundle induced from the twisted trivial bundle \(\calO\otimes\chi_k\).
It is quite elementary to
show

\begin{proposition}
  The fibers of the map \(\pi\) are projective spaces \(\mathbb{P}^{n-1}\) and
  \begin{enumerate}
  \item \(\calB_n/\pi^*(\calB_{n-1})=\calO_n(-1).\)
  \item \(\calO_n(-1)|_{\pi^{-1}(z)}=\calO_{\mathbb{P}^{n-1}}(-1).\)
  \end{enumerate}
\end{proposition}

We can combine the last proposition with the observation that the total homology \(H^*(\mathbb{P}^{n-1},\calO(-l))\) vanish
if \(l\in (1,n-1)\) and is one-dimensional for \(l=0,n\):

\begin{corollary}
  For any \(n\) we have:
  \begin{itemize}
  \item \(\pi_*(\Lambda^k\calB_n)=\Lambda^k\calB_{n-1}\)
  \item \(\pi_*(\calO_n(m)\ot \Lambda^k\calB_n)=0\) if \(m\in [-n+2,-1]\).
    \item \(\pi_*(\calO_n(-n+1)\ot \Lambda^k\calB_n)=\Lambda^{k-1}\calB_{n-1}[n]\)
  \end{itemize}
\end{corollary}

The geometric version of the Markov moves is the following
\begin{theorem} For any \(D\in \Br_{n}^\flat\), \(D'\in\Br_{n-1}^\flat \) as in the statement of
  proposition \ref{prop:Markov-2} we have
  \[\mathbb{H}^k(D\circ D^{(n-1,n)}_\bullet)=(\mathbb{H}^k(D')\oplus \mathbb{H}^{k-1}(D'))\otimes \CC[x_{11}]\]
   \[ \mathbb{H}^k(D)=(\mathbb{H}^k(D')\oplus \mathbf{q}^2\cdot\mathbb{H}^{k-1}(D'))\otimes \CC[x_{11}].\]
\end{theorem}

Our proof of the theorem follows the method of \cite{OblomkovRozansky16} where similar statement is shown in section 13.
To make this paper self-contain we present a slightly more geometric proof of the main technical step of
the proof. We dedicate the next section to these results.

\subsection{Induction and closure}
\label{sec:induction-closure}

In this section we explore interaction between the induction, convolution and the closure functor \(\bbll\). Before we state
our main result we need to extend some of the definitions from the reduced category \(\overline{\MF}_n\) to \(\MF^\circ_n\).
In particular, we define:
\[\bbll^\circ:\MF^{\circ}_n\to \rmD^{\per}_{\Tqt\ti B^2}(\frn_n\ti \frb_n),\quad \bbll^\circ= j_e^*\circ (KN^\circ)^{-1},\]
\[\bbll=j_e^*:\overline{\MF}_n\to \rmD^{\per}_{\Tqt\ti B^2}(\frn_n\ti \frb_n)\]


To state our first result let us recall notation for the parabolic subgroups: \(P_I\subset \GG_n\), \(I\subset\{1,\dots,n-1\}\),
is the subgroup such that  \(\mathrm{Lie}(P_I)=\mathfrak{p}_I\)  is generated by \(\frb\) and \(E_{i,i+1}\), \(i\notin I\).
We also define \(P_{\ge k}\subset P_{k,\dots,n-1}\) as a kernel of homomorphism \(P_{k,\dots,n-1}\to (\CC^*)^{n-k}\), that
is the projection on the last \(n-k\) diagonal elements.
Using this group we can  define the inclusion functor:
\[\begin{tikzcd}\calX^\circ(P_{\ge k})\arrow[r,"i_{\ge k}"]\arrow[d,"p_{\ge k}"]&\calX_n^\circ\\
    \calX_{k}^\circ&\end{tikzcd}, \inc_{\ge k}=i_{\ge k,*}\circ p_{\ge k}^*,\]
where here and everywhere below
\[\calX^\circ_n(P)=\mathrm{Lie}(P)\ti\frb_n\ti P\ti\frb_n.\]

There is a natural relation between the inclusion and induction functors:
\begin{proposition}
  For any \(n\) and \(k\le n-1\) we have:
  \[\ind_k(\calF\ti \calC_\parallel)=\inc_{\ge k}(\calF).\]
\end{proposition}
\begin{proof}
  The formula follows from the construction of the unit matrix factorization
  as push-forward \(\calC_\parallel=i_{\ge 1,*}(\calO)\), since \(P_{\ge 1}=U\). The  map \(i_{\ge 1}\) naturally
  fits into the commutative diagram:
  \[\begin{tikzcd}
      \calX^\circ_n(P_{\ge k})\arrow[r,"i'"]\arrow[d,"p'"]&\calX^\circ(P_k)\arrow[r,"i_k"]\arrow[d,"p_k"]&\calX_n^\circ\\
      \calX^{\circ}_k\ti\calX_{n-k}(U_{n-k})\arrow[r,"1\ti i_{\ge 1}"]\arrow[d,"p''"]&\calX^\circ_k\ti\calX^\circ_{n-k}&\\
      \calX^\circ_k&&
    \end{tikzcd}.
  \]
  Now we can use the base-change relations:
  \begin{multline*}
    \ind_k(\calF\ti\calC_\parallel)=i_{k*}\circ p_k^*(\calF\ti\calC_\parallel)=i_{k*}\circ p_k^*\circ (1\ti i_{\ge 1})_*\circ (p'')^*(\calF)=
    i_{k*}\circ i'_*\circ (p')^*\circ (p'')^*(\calF)\\
    =i_{\ge k,*}\circ p_{\ge k}^*(\calF)=\inc_{\ge k}(\calF),\end{multline*}
  where we used the base-change and \(p''\circ p'=p_{\ge k}\), \(i_k\circ i'=i_{\ge k}\).
\end{proof}

To connect with the convolution operation with the closure operation \(\bbll^\circ\) we need to discuss the Knorrer functor
\(KN^\circ\). Observe that the potential \(\bar{W}^\circ\) has a quadratic summand:
\[\bar{W}^\circ(X,Y_1,g,Y_2)=\Tr(XY_1)-\Tr(X_{++}\Ad_gY_2).\]
Thus we have an isomorphism of categories \(\overline{\MF}_n^\circ\simeq\overline{\MF}_n\). To obtain the isomorphism
\(KN^\circ\) we need to compose the last isomorphism with \(KN\). The  composition fits into the commuting diagram:
\[\begin{tikzcd}
    \frg_n\ti \GG_n\ti \frb_n&&\arrow[ll,"p^\circ"]\calX_n^\circ\\
    \frb_n\ti\GG_n\ti\frb_n\arrow[u]&\arrow[l,"p'"]\overline{\calX}_n^\circ&\arrow[l,"\pi_{kn}"]\arrow[u,"j_{kn}"]\frb_n^2\ti \GG_n\frb\\
    \overline{\calX}_n\arrow[u]\arrow[uu,bend left=70,"j^\circ"]&\frb_n\ti \GG_n\ti \frb_n\arrow[u,"\bar{j}^\circ"]\arrow[l,"\bar{p}^\circ"]&
  \end{tikzcd},
\]
where \(p^\circ(X,Y_1,g,Y_2)=(X,g,Y_2)\) and \(j^\circ \) is the natural inclusion of \(\overline{X}_n=\frn_n\ti \GG_n\ti \frb_n\).
Thus the base-change relation implies that
\[KN^\circ=p^\circ_*\circ j^{\circ,*}.\]


Using base-change formula we can simplify the formula for the functor \(\bbll^\circ\). We can describe the functor in terms
of maps:
\[j_e^\circ: \frn_n\ti \frb_n\ti\frb_n\to\calX^\circ,\quad \rho_e:\frn_n\ti \frb_n\ti \frb_n\to \frn_n\ti \frb_n, \]
\[j_e^\circ(X,Y_1,Y_2)=(X,Y_1,1,Y_2),\quad \rho_e(X,Y_1,Y_2)=(X,Y_2).\]

\begin{proposition} For any \(\calF\in \MF^\circ_n\) we have:
  \[\bbll^\circ(\calF)=\rho_{e*}\circ j_e^{\circ*}(\calF).\]
\end{proposition}
\begin{proof}
  The above maps and the functor \(KN^\circ\) fit into commuting diagram:
  \[\begin{tikzcd}
      &\frg_n\ti \GG_n\ti \frb_n&\calX^\circ_n\arrow[l,"p^\circ"]\\
      \frn_n\ti \frb_n\arrow[r,"j_e"]&\overline{\calX}_n\arrow[u,"j^\circ"] &\frn\ti\frb_n\ti\frb_n\arrow[ll,bend left=10,dashed,"\rho_e"]\arrow[u,dashed,"j_e^\circ"]
    \end{tikzcd}
  \]
  Thus we can apply the base-change to complete the proof:
  \[\bbll^\circ=j_e^*\circ j^{\circ,*}\circ p^\circ_*=\rho_{e*}\circ j_e^{\circ,*}.\]
\end{proof}

We need a slight generalization of the previous formula that explain how to compute a closure of the convolution
of two elements.

First, let us recall how we define the monoidal structure on \(\MF_n^\circ\):
\[\calX_{con}^\circ=\frg\ti \frb\ti \GG_n\ti \frb\ti \GG_n\ti \frb,\quad \pi_{ij}^\circ:\calX_{con}^\circ\to \calX_n^\circ,\]
\[\pi_{12}^\circ(X,Y_1,g_{12},Y_2,g_{23},
  Y_3)=(X,Y_1,g_{12},Y_2),\,\, \pi_{13}^\circ(X,Y_1,g_{12},Y_2,g_{23},
  Y_3)=(X_1,Y_1,g_{12}g_{23},Y_3),
\] \[\pi_{23}^\circ(X,Y_1,g_{12},Y_2,g_{23},Y_3)=
(\Ad_{g_{12}}^{-1}(X),Y_2,g_{23},Y_3),\]
\[\calF\star\calG=\pi_{13*}^\circ(\CE_{\frn}(\pi_{12}^*(\calF)\otimes\pi_{23}^*(\calG))^T).\]

To state the next statement we set
\[j^\circ_{*,G}:\frn_n\ti\frb_n^3\ti \GG_n\to\calX_n^\circ,\quad
\rho_G:\frn_n\ti\frb_n^3\ti \GG_n\to\frn_n\ti \frb_n.\]
\[j^\circ_{l,G}(X,Y_1,Y_2,Y_3,g)=(X,Y_1,g,Y_2),\quad
  j^\circ_{r,G}(X,Y_1,Y_2,Y_3,g)=(\Ad_g^{-1}X,Y_2,g^{-1},Y_3),
\]
\[\rho_G(X,Y_1,Y_2,Y_3,g)=(X,Y_3).\]

\begin{proposition}
  For any \(\calF,\calG\in \MF^\circ\) we have:
\[\bbll^\circ(\calF\star\calG)=\rho_{G*}(\CE_{\frn}(j^{\circ,*}_{l,G}(\calF)\ot j^{\circ,*}_{r,G}(\calG))^T),\]
  where the \(B\)-action in the quotient is
  \[b\cdot(X,Y_1,Y_2,Y_3,g)=(X,Y_1,\Ad_b(Y_2),Y_3,gb^{-1}).\]
\end{proposition}
\begin{proof}
  The statement follows from the base-change formula applied to
  the commuting diagram of maps:
  \[\begin{tikzcd}
      &\frn_n\ti \frb_n^3\ti \GG_n\arrow[r,dashed,"j_{e,G}"]\arrow[d,dashed,"p_{e,G}"]
      &
  \calX^\circ_{con}\arrow[d,"\pi_{13}^\circ"]\arrow[r,"\pi_{12}^\circ\ti\pi_{23}^\circ"]
      &\calX_n^\circ\ti \calX_n^\circ\\
      \frn_n\ti \frb_n&\frn_n\ti\frb_n^2\arrow[l,"\rho"]\arrow[r,"j_e"]&\calX_n^\circ&
    \end{tikzcd},
  \]
  where the new maps are
  \[j_{e,G}(X,Y_1,Y_2,Y_3,g)=(X,Y_1,g,Y_2,g^{-1},Y_3),\quad p_{e,G}(X,Y_1,Y_2,Y_3,g)=(X,Y_1,Y_3).\]
  To complete the proof one can use the base-change and
  \[\rho\circ p_{e,G}=\rho_{G},\quad  \pi_{12}^\circ\ti\pi_{23}^\circ\circ
  j_{e,G}=j^\circ_{l,G}\ti j^\circ_{r,G}.\]
\end{proof}

Now we need combine the previous computation with the
inclusion functor. The pair of maps needed for this
combination is:
\[j_{l,\ge k}^\circ: \frn_n\ti \frb_n^3\ti P_{\ge k}\to \calX^\circ_n,\quad p_{l,\ge k}^\circ:\frn_n\ti \frb_n^3\ti P_{\ge k}\to \calX^\circ_k, \]
where \(j_{l,\ge k}^\circ\) is a restriction of \(j_{l,G}^\circ\) to the subspace
and \(p_{r,\ge k}\) is a restriction of \(j_{r,G}^\circ\) post-composed
with the map \(\calX_n(P_{\ge k})\to \calX_k^\circ\).

\begin{proposition}
  For any \(\calF\in \MF_n^\circ\), \(\calG\in \MF_k^\circ\) we have:
\[\bbll^\circ(\calF\star \inc(\calG))=\rho_{G*}(\CE_{\frn}(j^{\circ,*}_{l,\ge k}(\calF)\ot p^{\circ,*}_{r,\ge k}(\calG))^T),\]
\end{proposition}
\begin{proof}
  There is a unique map \(\phi\) such  that the diagram
  below commutes:
  \[\begin{tikzcd}
      &\calX_n^\circ\ti \calX_n^\circ&\calX^\circ_n(P_{\ge k})\ti\calX^\circ_n\arrow[r]\arrow[l,"j_{\ge k\ti 1}"]&\calX^\circ_k\ti\calX_n^\circ\\
      \frn_n\ti\frb_n&\arrow[u,"j_{l,G}^\circ\ti j_{r,G}^\circ"]\frn_n\ti \frb_n^3\ti \GG_n\arrow[l,"\rho_G"]&\frn_n\ti\frb_n^3\ti P_{\ge k}\arrow[l,dashed,hook]\arrow[ur,"j_{l,\ge k}\ti p_{r,\ge k}"']\arrow[u,dashed,"\phi"]&
    \end{tikzcd}
  \]
  Thus the statement follow from the base-change formula.
\end{proof}

The  closure functor in reduced category \(\overline{\MF}_n\) is simpler than in the full category \(\MF^\circ\). Thus we use the
the matrix factorizations from the reduced category in our main lemma and we give a version of the previous result in the reduced
category.

To state the reduced version of the result let us recall that we the automorphism of the space \(\overline{\calX}_n\):
\[ \frn_n\ti \GG_n\ti \frb_n\ni (X,g,Y)\to (X,g,-Y)\]
induces an equivalence of categories \(\overline{\MF}_n\) and \(\overline{\MF}_n^*=\MF_{B^2}(\overline{\calX}_n,-\overline{W})\).
For  \(\calF\in \overline{\MF}_n\) we denote by \(\calF^*\) the corresponding element of \(\overline{\MF}_n^*\).

Next we fix a pair of natural maps from the space \(\overline{\calX}_n(P_{\ge k})=\frn_n\ti P_{\ge k}\ti \frb_n\):
\[\bar{j}_{l,\ge k}:\overline{\calX}_n(P_{\ge k})\to \overline{\calX}_n,\quad \bar{p}_{r,\ge k}:\overline{\calX}_n(P_{\ge k})\to \overline{\calX}_k\]

\begin{proposition}\label{prop:last-redu}
  For any \(\overline{\calF}\in \overline{\MF}_n\), \(\ol{\calG}\in \ol{\MF}_k\) we have:
  \[\bbll^\circ(KN^\circ(\ol{\calF})\star \inc_{\ge k}(KN^\circ(\ol{\calG}))))=
    \rho_{G*}(\CE_{\frn}(\ol{j}^{\circ,*}_{l,\ge k}(\ol{\calF})\ot \ol{p}^{\circ,*}_{r,\ge k}(\ol{\calG)}^*)^T)
  \]
\end{proposition}
\begin{proof}
  In the diagram below the dashed arrows correspond to the unique
  maps that makes diagram commute:
  \[\begin{tikzcd}
      \frn_n\ti\frb_n\arrow[r,equal]&\arrow[r,equal] \frn_n\ti\frb_n& \frn_n\ti\frb_n\\ \arrow[d,"j_{l,\ge k}\ti p_{r,\ge k}"]
      \frn_n\ti\frb_n^3\ti P_{\ge k}\arrow[u,"\rho_G"]&\arrow[l,equal]\arrow[r,dashed]\frn_n\ti\frb_n^3\ti P_{\ge k}\arrow[u,"\rho_G"]\arrow[d,dashed]&
      \ol{\calX}_n(P_{\ge k})\arrow[u,"\rho_G"]\arrow[d,"\bar{j}_{l,\ge k}\ti \bar{p}_{r,\ge k}"]\\
      \calX_n^\circ\ti \calX_k^\circ&(\frn_n\ti \GG_n\ti \frb_n^2)\ti
     ( \frn_k\ti \GG_k\ti \frb_k^2)\arrow[r,"p^\circ\ti p^\circ"]\arrow[l,"j^\circ\ti j^\circ"]&\ol{\calX}_n\ti\ol{\calX}_k
    \end{tikzcd}
  \]
\end{proof}

The  \(B^2_n\)-equivariant projection
\[\pi:\frn_n\ti \frb_n\to \frb_{n-1}\ti \frb_{n-1},\]
that projects matrices to the subspace spanned by the matrix units \(E_{ij}\), \(i,j>1\) descends to the projection
\(\pi\) between the free flag Hilbert scheme \eqref{eq:pi-proj}. On the other hand the projection
\[\pi_{\ge k}:\frn_n\ti\frb_n\to \frn_k\ti \frb_k.\]
that projects matrices to the subspace spanned by the matrix units \(E_{ij}\), \(i,j\le k\) does not descend to
to the map between the free flag Hilbert schemes. The  failure is due to the fact that the stability condition is incompatible
with the projection. However, we can use this projection to understand the homology of the graphs that appear in the
the Markov move theorem:

\begin{lemma}\label{lem:Markov}
  Let \(\overline{\calF}\in \overline{\MF}_{n-1}\), \(\overline{\calG}\in \overline{\MF}_2\) and \(V=\CC^{n-2}\langle \epsilon_1\rangle \)  is a vector space
  with the  \(B_n\)-action given by the character \(\chi_1=\exp(\epsilon_1)\).
  Suppose we have \[\bbll(\overline{\calF})=(M,d)\in \rmD^{\per}_{B^2}(\frn_{n-1}\ti\frb_{n-1})\]
then
  there is a deformed
  \(\ZZ_2\)-graded complex \[\overline{\calF}'=(\pi^*M\ot \Lambda^{even}V\oplus \pi^*M\ot \Lambda^{odd}V,D_V+\pi^* d)\in \rmD^{\per}_{B}(\frn_n\ti \frb_n)\]
  \begin{equation*}\label{dia:big-1}
 D_V:    \begin{tikzcd}[column sep=small]
             \pi^* M\ar[r]\ar[rrr,bend right=15]\ar[rrrrr,bend right=20]&
      \pi^* M\ot V\ar[r]\ar[rrr,bend right=15]&
       \pi^*M\ot\Lambda^2V\ar[r]\ar[rrr,bend right=15]&
       \pi^*M\ot\Lambda^3V\ar[r]&
      \pi^*M\ot\Lambda^4 V\ar[r]&\cdots
  \end{tikzcd}
\end{equation*}
such that
\[\bbll^\circ\bigg(\ind_1\circ KN^\circ(\overline{\calF})\star \inc_2\circ KN^\circ(\overline{\calG})\bigg)=\overline{\calF}'\ot \pi^*_{\ge 2}(\bbll(\overline{\calG}^*)).\]
\end{lemma}

In the proof of the lemma we use the reduced version of the induction and inclusion functors:
\[\begin{tikzcd}\overline{\calX}_n(P_{\ge k})\arrow[r,"\bar{i}_{\ge k}"]\arrow[d,"\bar{p}_{\ge k}"]&\overline{\calX}_n\\
    \overline{\calX}_{k}&
  \end{tikzcd} \quad\quad
\begin{tikzcd}\overline{\calX}_n(P_k)\arrow[r,"\bar{i}_{k}"]\arrow[d,"\bar{p}_{ k}"]&\overline{\calX}_n\\
    \overline{\calX}_{k}\ti\overline{\calX}_{n-k}&
  \end{tikzcd}\quad\quad
\begin{tikzcd}
  \overline{\inc}_{\ge k}=\bar{i}_{\ge k,*}\circ \bar{p}_{\ge k}^*,\\\overline{\ind}_{ k}=\bar{i}_{ k,*}\circ \bar{p}_{k}^*
\end{tikzcd}
\]
where here and everywhere below
\[\overline{\calX}_n(P)=\frn_n\ti P\ti\frb_n.\]

The  Knorrer functor \(KN^\circ\) intertwines these functors with the functors
\(\inc_{\ge k}\) and \(\ind_k\) thus the lemma above is equivalent to the formula:
\begin{equation}\label{eq:inc-ind}
\bbll^\circ\bigg( KN^\circ\circ\ol{\ind_1}(\overline{\calF})\star  KN^\circ\circ\ol{\inc}_2(\overline{\calG})\bigg)=\overline{\calF}'\ot \pi^*_{\ge 2}(\bbll(\overline{\calG}^*)).
\end{equation}

\begin{proof}[Proof of Lemma~\ref{lem:Markov}]
  Let us denote by \(P'\subset \GG_n\) the \(B^2\) invariant
  subspace defined by
  \[P'=\{ g\in \GG_n| g_{1i}=0, i=3,\dots,n\}.\]
  The natural inclusions
  \[j'_1: P_1\to P',\quad j_{\ge 2}: P_{\ge 2}\to P', \quad i': P'\to \GG_n,\]
  induces the dashed arrow maps in the diagram:
  \[\begin{tikzcd}
&\overline{\calX}_n(P')\ar[rd,dashed,"\bar{i}'"]&\ar[d,"\bar{j}_1"]\ar[r,"\bar{p}_1"]\ar[rd,dashed,"\bar{j}'_1"]\overline{\calX}_n(P_1)&\overline{\calX}_{n-1}\ti\overline{\calX}_1\\
\overline{\calX}_2&\ar[l,"\bar{p}_{\ge 2}"]\ar[r,"\bar{i}_{\ge 2}"]\ar[u,dashed,"\bar{j}'_{\ge 2}"]\overline{\calX}_n(P_{\ge 2})&\ol{\calX}_n&\ol{\calX}_n(P')\ar[l,dashed,"\bar{i}'"]
    \end{tikzcd}
  \]
  There are natural maps \(\rho_G\)  from the spaces in the second and third  diagram to \(\frn_n\ti \frb_n\) which intertwine the maps in the diagram.

  Now we show how one can use base-change and projection formula
  to prove \eqref{eq:inc-ind}.
  According to proposition~\ref{prop:last-redu} the  LHS of the formula is obtained by application of
  \(\CE_{\frn}(\cdot)^T\) to
  \[
    \rho_{G*}(\bar{j}_{1*}\circ \bar{p}_1^*(\ol{\calF})\ot
    \bar{i}_{\ge 2*}\circ\bar{p}_{\ge 2}^*(\ol{\calG}^*))=
\rho_{G*}(\bar{i}^{\prime*}\circ\bar{j}_{1*}\circ \bar{p}_1^*(\ol{\calF})\ot
    \bar{j}'_{\ge 2*}\circ\bar{p}_{\ge 2}^*(\ol{\calG}^*))
  \]
  here we used that \(\bar{i}_{\ge 2,*}=\bar{i}'_*\circ \bar{j}'_{\ge 2}\) and the projection formula for the map \(\bar{i}'\).
  Similarly we have \(\bar{j}_{1*}=\bar{i}'_*\circ \bar{j}'_{1*}\), hence
  \[
    \bar{i}^{\prime*}\circ\bar{j}_{1*}\circ \bar{p}_1^*(\ol{\calF})=\bar{i}^{\prime *}\circ \bar{i}'_*\circ \bar{j}'_{1*}\circ \bar{p}_1^*(\ol{\calF})= \bar{i}^{\prime *}\circ \bar{i}'_*(\calF'),\quad \calF'=\bar{j}'_{1*}\circ \bar{p}_1^*(\ol{\calF}).\]
  The key observation follows from the construction of the push-forward  for matrix factorizations from \cite{OblomkovRozansky16}. Indeed, if \(\calF'=(M',d')\) then
  the composition
  \(\bar{i}^{\prime *}\circ \bar{i}'_*(\calF')\) is equal to the matrix factorization
  \[(M'\ot \Lambda^\bullet V, d'+D'_V),\quad D'_V\in \oplus_{k,i\ge 0}
    \Hom(M'\ot \Lambda^i V, M'\ot \Lambda^{i+2k}V), \]
  where \(V\) is a vector space spanned by the matrix elements
  \(g_{1i}\), \(i\ge 3\). These are exactly the matrix elements
  vanish on \(P'\). We note the last matrix factorization by
  \(\calF'\ot \Lambda^\bullet V\) for brevity.

  To complete our proof we need to apply the base-change relation one
  more time. We work with the diagram:
  \[
\begin{tikzcd}    &\ar[d,dashed,"\gamma"]\overline{\calX}_n(B_n)\ar[r,dashed,"\delta"]&\ar[d,"\bar{j}'_1"]\ar[r,"\bar{p}_1"]\overline{\calX}_n(P_1)&\overline{\calX}_{n-1}\ti\overline{\calX}_1\\
\overline{\calX}_2&\ar[l,"\bar{p}_{\ge 2}"]\ar[r,"\bar{j}'_{\ge 2}"]\overline{\calX}_n(P_{\ge 2} )&\ol{\calX}_n(P')&
    \end{tikzcd}
  \]

  Thus by the previous argument:
  \begin{multline*}
  \rho_{G*}(\bar{j}'_{1*}\circ \bar{p}_1^*(\ol{\calF})\ot \Lambda^\bullet V\ot
    \bar{j}'_{\ge 2*}\circ\bar{p}_{\ge 2}^*(\ol{\calG}^*))=
  \rho_{G*}( \bar{p}_1^*(\ol{\calF})\ot \bar{j}^{\prime,*}(\Lambda V^\bullet)\ot
  \bar{j}^{\prime *}\circ \bar{j}'_{\ge 2*}\circ\bar{p}_{\ge 2}^*(\ol{\calG}^*))\\=
    \rho_{G*}( \bar{p}_1^*(\ol{\calF})\ot \bar{j}^{\prime,*}(\Lambda V^\bullet)\ot
 \delta_*\circ \gamma^*\circ\bar{p}_{\ge 2}^*(\ol{\calG}^*))=
    \rho_{G*}(\delta^*\circ \bar{p}_1^*(\ol{\calF})\ot \delta^*\circ \bar{j}^{\prime *}(\Lambda V^\bullet)\ot
  \gamma^*\circ\bar{p}_{\ge 2}^*(\ol{\calG}^*)).
\end{multline*}

The \(B_n\)-action on \(\ol{\calX}_n(B_n)\)  is free. Thus the functor
\(\CE_{\frn}(\cdot)^T\) on this space is equivalent to
the functor of restriction on the \(B_n\)-action slice:
\[\bar{j}_e:\ol{\calX}_n\to\frn_n\ti \frb_n.\]

Thus by the previous argument the LHS of \eqref{eq:inc-ind} is equal
\[\bar{j}_{e}^*(\delta^*\circ \bar{p}_1^*(\ol{\calF})\ot \delta^*\circ \bar{j}^{\prime *}(\Lambda V^\bullet)\ot
  \gamma^*\circ\bar{p}_{\ge 2}^*(\ol{\calG}^*)).
\]

The  statement follows the last formula because
\[\bar{j}_e\circ\gamma^*\circ \bar{p}_{\ge 2}^*(\ol{\calG}^*)=
  \pi_{\ge 2}^*(\bbll(\ol{\calG}^*)), \quad \bar{j}_{e}^*(\delta^*\circ \bar{p}_1^*(\ol{\calF})\ot \delta^*\circ \bar{j}^{\prime *}(\Lambda V^\bullet))=\ol{\calF}',\]
where \(D_V=\bar{j}_e^*\circ \delta^*\circ \bar{j}^{\prime*}(D'_V)\).

\end{proof}
\subsection{Proof  of Markov move relations}
\label{sec:proof-markov-move}
Let us set \[\ol{\calF}_\star=\ol{\calF}_{D\circ D_\star^{(n-1,n)}}\in\MF(\ol{\calX}_n,\Wr)\] and \(\star\) can be either \(\bullet\) or \(\parallel\).

The  lemma~\ref{lem:Markov} is directly applicable in the
setting of the proof, that is:
\[\bbll(\ol{\calF}_\star)=\ol{\calF}'_D\ot \pi_{\ge 2}^*(\bbll(\ol{\calC}_\star^*)).\]

To describe the second factor in the last formula let us
fix coordinates on \(\ol{\calX}_n=\frn_n\ti \GG_n\ti \frb_n\) as \(X\in\frn_n\), \(Y\in\frb_n \). Then the pull-backs are the
following Koszul complexes:
\[\pi_{\ge 2}^*(\bbll(\ol{\calC}_\bullet^*))=\mathrm{K}_\bullet=\mathrm{K}[x_{12}],\quad
  \pi_{\ge 2}^*(\bbll(\ol{\calC}_\parallel^*))=\mathrm\mathrm{K}_\parallel={K}[x_{12}y_{11}].\]

After restricting to the stable locus these Koszul complexes
define the divisors  \(\FHilb_{n,\star}\subset\FHilb_n^{free}\). Let us describe these divisors.

 The  projection \(\pi\) is surjective on the divisors \(\FHilb_{n,\star}\). Moreover,  {\it every} fiber of the projection
 \(\pi|_{\FHilb_{n,\bullet}}\) is a projective space \(\mathbb{P}^{n-2}\). The  divisor \(\FHilb_{n,\parallel}\) is bigger,
 \[\FHilb_{n,\parallel}=\FHilb_{n,\bullet}\cup \pi^{-1}(0\times \FHilb_{n-1}^{free}).\]
 That is the fibers of  \(\pi|_{\FHilb_{n,\parallel}}\) are \(\mathbb{P}^{n-2}\) outside the divisor \(x_{11}=0\) and
 the fiber is \(\mathbb{P}^{n-1}\) over the divisor.

The vector space
 \(V\) in the definition of complex \(\ol{\calF}'_D\) becomes the vector bundle \(\calO_n(-1)^{\oplus n-2}\).
Hence  to compute \(\pi_*(\ol{\calF}_{\star}\ot \Lambda^k\calB_n))\), we need to modify the previous corollary.

We denote \(\pi_\bullet\) and \(\pi_\parallel\) the restriction of the projection \(\pi\) on the corresponding
divisor:
\[\pi_{\star*}(\calC)=\pi_{*}(\mathrm{K}_\star\ot \calC).\]

Let us discuss the case \(\star=\bullet\). Then for any \(z\in \CC\ti \FHilb_{n-1}\) the fiber \(F_z=\pi^{-1}(z)\)
is the projective space \(\mathbb{P}^{n-2}\). Hence
\begin{itemize}
  \item \(\pi_{\bullet*}(\Lambda^k\calB_n)=\Lambda^k\calB_{n-1}\)
  \item \(\pi_{\bullet*}(\calO_n(m)\ot \Lambda^k\calB_n)=0\) if \(m\in [-n+3,-1]\).
    \item \(\pi_{\bullet*}(\calO_n(-n+2)\ot \Lambda^k\calB_n)=\Lambda^{k-1}\calB_{n-1}[n-2]\)
  \end{itemize}

  The statement is similar for \(\pi_{\parallel}\) except we have take into account that over the divisor
  \(\mathscr{D}=\{x_{11}=0\}\) the dimension of the fiber jumps:
\begin{itemize}
  \item \(\pi_{\parallel*}(\Lambda^k\calB_n)=\Lambda^k\calB_{n-1}\)
  \item \(\pi_{\parallel*}(\calO_n(m)\ot \Lambda^k\calB_n)=0\) if \(m\in [-n+3,-1]\).
    \item \(\pi_{\parallel*}(\calO_n(-n+2)\ot \Lambda^k\calB_n)=\calO(-\mathscr{D})\otimes\Lambda^{k-1}\calB_{n-1}[n-2]\).
  \end{itemize}
Here we used the derived version of the push-forward \(\pi_{\parallel *}(\calF):=\pi_*(\calF\otimes \calO_{\FHilb_{n,\parallel}})\).

Thus in both case then only the left and the right extreme terms of \(\Lambda^\bullet(V)\)  survive the push-forward \(\pi_*\).  The contraction of the \(\pi_*\)-acyclic terms
  could potential lead to appearance of new correction arrows. However, this does not happen since only two extreme ends of the complex \(\Lambda^\bullet V,D_V\)
  survive the push-forward. Hence the standard spectral sequence (or Gauss elimination lemma ) argument implies that
  \[\pi_*(\ol{\calF}_{\bullet}\ot \Lambda^k\calB_n)= \CC[x_{11}]\ot\bigg(\bbll(\ol{\calF}_{D'})\ot \Lambda^k\calB_{n-1}\oplus
    \bbll(\ol{\calF}_{D'})\ot \Lambda^{k-1}\calB_{n-1}[n-2]\bigg).\]
  \[\pi_*(\ol{\calF}_{\parallel}\ot \Lambda^k\calB_n)= \CC[x_{11}]\ot \bbll(\ol{\calF}_{D'})\ot \Lambda^k\calB_{n-1}\oplus
 x_{11}\CC[x_{11}]\ot    \bbll(\ol{\calF}_{D'})\ot \Lambda^{k-1}\calB_{n-1}[n-2].\]

\section{Properties of the functor}
\label{sec:properties-functor}


\subsection{Geometric Soergel category}
\label{sec:geom-soerg-categ}



The  general braid graph  \(D\) is a composition of the elementary graphs \(D_\bullet^{(i,i+1)}\), \(D=D_\bullet^{(i_1,i_1+1)}\circ \dots
\circ D_\bullet^{(i_m,i_m+1)}\). Thus \(\Phi_n^\flat\) assigns to the graph \(D\) the matrix factorization:
\[\calF_D:=\calC_\bullet^{(i_1)}\star\dots \star \calC_\bullet^{(i_m)}.\]

Categories of matrix factorizations are triangulated
\cite{Orlov04} hence additive.
We defined an additive a \(\MF_n^\flat\) as Karoubi envelope of the smallest full subcategory of additive category  \(\MF_n\) that contains the
elementary matrix factorizations \(\calC_\bullet^{(i)}\), \(i=1,\dots,n-1\) and closed under convolution \(\star\). The  matrix factorizations
\(\calF_D\) are objects of this category but it contains more objects that correspond to more singular diagrams.

The main technical statement of the paper now follows from the
MOY relations by application of the  technique developed by Hao Wu \cite{Wu08}.
The same technique used for example in the work of Jake Rasmussen \cite{Rasmussen15}
we refer for the details to his work and below we the key steps of the proof

\begin{proof}[Proof of Theorem~\ref{thm:main-cat}]
  The first part of the theorem is proven in \ref{prop:monoid}.
  The second part follows from the MOY relations from the section~\ref{sec:relations-mfflat}, see lemma~\ref{lem:MOY1} and lemma~\ref{lem:MOY2}.

  For third part we use Hao Wu technique \cite{Wu08}. He shows (see also \cite{Rasmussen15}) that
  if one has homology theory for braid graphs such that the  MOY relation from lemma~\ref{lem:MOY1}, lemma~\ref{lem:MOY2} and
  proposition~\ref{prop:Markov-2} hold
  then there is a unique inductive procedure for the reduction to the unknot.
  It was shown in \cite{Rasmussen15} that
  \[\Ext_{\rmD(\CC^n\ti\CC^n)}(\Phi_S(\xId),\Phi_S(\gamma))\]
  satisfies the MOY relations. Thus we have the statement for the case \(\gamma_1=\xId\).  The general case follows from the special since the
  corollary~\ref{cor:cycle}.

  The part follows from the previous ones after combining with
  \(\calB^\vee\otimes \det(\calB)=\calB\).
\end{proof}

Thus we immediately obtain

\begin{corollary}
  The functor \(\bb\) is  a faithful functor when restricted to \(\MF_n^\flat\):
  \[\bb: \MF_n^\flat\to \sbim_n.
  \]
\end{corollary}

The  Hochschild homology functor defines the natural traces on the  category \(\sbim_n\):
\[ B\to \Tor^i_{R_n\ot R_n}(B,R_n).\]

As we have shown in \cite{OblomkovRozansky16} the category \(\MF_n \) also has natural trace functors
\[\Tr^i(\calF):=\CE_{\frn^2}\Big(\mathcal{E}xt(\calF,\calF_{\parallel}\ot\Lambda^i\calB)\Big)^{T^2\ti \GG_n}.\]
Thus the theorem~\ref{thm:main-cat} and the properties of the duality
functor from the section \ref{sec:duality}
imply that these functors are intertwined by the functor \(\bb\):

\begin{corollary}
  For any \(\calF\in \MF^\flat_n\) we have an isomorphism of graded vector spaces
  \[\Tr^i(\calF)=\Tor^i_{R_n\ot R_n}(\bb(\calF),R_n).\]
\end{corollary}

Let us notice that element \(\Tr^i(\calF_D)\) naturally has gradings: \(\mathbf{q}\) and \(\mathbf{t}\)-gradings. On the
other hand the Soergel bimodule side of the equality has only one grading: \(\mathbf{q}\)-grading. Thus we see that
the elements \(\calF_D\) are actually pure:

\begin{corollary}
  For any diagram \(D\) all elements of \(\Tr^i(\calF_D)\) are of the same \(\mathbf{t}\)-degree.
\end{corollary}

The  trace functor also computes the space of morphisms between the objects \(\MF_n^\flat\) since:
\[\Tr^0(\calF_D\star\calF_{D'})=\Hom(\calF_D,\calF_{D'}).\]

Similarly, any two Soergel bimodules \(B,B'\) we have
\[\Tor^0_{R_n}(B\ot_{R_n}B', R_n)=\Hom(B,B').\]

Thus the matching of the traces implies
\[\Hom(\bb(\calF),\bb(\calF'))=\Hom(\calF,\calF'),\]
for any \(\calF,\calF'\in \MF_n^\flat\).

\subsection{Extension to Rouquier complexes}
\label{sec:extens-rouq-compl}
Let us denote by \(\ho(\sbim_n)\) and \(\ho(\MF^\flat)\) the homotopy categories of the bounded complexes of
the objects of these two categories. Rouquier constructed the homomorphism from the braid group \(\Br_n\) to
the monoidal category \(\ho(\sbim_n)\), \(\Phi_R:\Br_n\to \ho(\sbim_n)\):
\[\Br_n\ni\beta\mapsto \calC_\beta=(C_{\beta,\bullet},d^R)\in\ho(\sbim_n).\]

Imitating Rouquier's construction we define the geometric version of the homomorphism
\[\Phi_{ho}:\Br_n\to \ho(\MF^\flat_n), \quad \beta\mapsto C_\beta,\]
and this homomorphism is compatible with the Rouquier's construction
\begin{equation}
  \label{eq:Rou-mf}
  \Phi_{R}=\bb\circ\Phi_{ho}.
\end{equation}

Probably, the most important result is the homotopy of the complexes

\begin{theorem}
  For any \(\beta\in\Br_n\) and \(i\) we have a homotopy of the complexes of graded vector spaces
  \[\Tr^i(\Phi_{ho}(\beta))\sim \Tor^i_{R_n}(R_n,\Phi_R(\beta)).\]
\end{theorem}

There is a natural projection \(\Pi\) from the homotopy category \(\ho(\MF_n^\flat)\) to the triangulated category
\(\MF_n\). This projection turns the \(\ZZ\)-grading of complexes to two-periodic grading of matrix factorizations.
In our previous work we constructed the homomorphism
\[\Phi: \Br_n\to \MF_n\]
and we show that our new construction is consistent with the previous result:
\begin{equation}
  \label{eq:Pi}
  \Phi=\Pi\circ\Phi_{ho}.
\end{equation}

It was shown in \cite{OblomkovRozansky16} that the total homology
\[\mathrm{H}^\bullet(\beta)=\mathbb{H}(\mathcal{E}xt(\Phi(\beta),\Phi(1)\ot \Lambda^\bullet\calB)),\]
is an isotopy invariant of the closure of  \(\beta\). In particular, the triply-graded invariant
from the introduction is
\[\HHH_{geo}(\beta)=\oplus_i \mathrm{H}^i(\beta).\]

Finally, using purity result from above we show

\begin{lemma}
  For any \(\beta\) we have
  \[\mathbb{H}(\CE_{\frn^2}(\mathcal{E}xt(\Phi_{ho}(\beta),\Phi_{ho}(\xId)\ot\Lambda^i\calB))^T)=H^*(\Tr^i(\Phi_{ho}(\beta)),d^R),\]
  and homological grading \(*\) matches with \(\mathbf{t}\)-grading on the other side as well as \(\mathbf{q}\)-grading
  matches with the polynomial grading.
\end{lemma}
\begin{proof}

  Let us fix the notation the two-periodic complex:
  \[\mathcal{S}^i_{\beta,\bullet}=
    (\Phi_{ho}(\beta)_\bullet\otimes \Phi_{ho}(\xId)^\dagger\ot \Lambda^i\calB, d^R),\]
  where \(d^R: \mathcal{S}^i_{\beta,\bullet}\to \mathcal{S}^i_{\beta,\bullet+1}\)
  is the Rouquier differential.

  In general,
  given an element \(\calF\in \MF_n\) the two-periodic complex \[C=\CE_{\frn^2}(\mathcal{E}xt(\calF,\calF_\parallel\ot \Lambda^i\calB))^{T^2}\] is
  has \(\mathbf{t}\)-degree one. That is if we decompose the \(C\) into the \(\mathbf{t}\)-homogeneous pieces
  \(C=\oplus_i C[i]\) then the differential of the complex \(d^{mf}\) is graded \(d^{mf}: C[i]\to C[i+1]\).
  The differential \(d^{mf}\) is the sum of the matrix factorization differential and \ChE   differential with the corrections.

  On the other hand the complex \(\Tr^i(\Phi_{ho}(\beta))\) is naturally a bi-complex. Indeed, \(\Phi_{ho}(\beta)\) is
  a complex whose terms are elements of \(\MF_n^\flat\), we denote these differential by \(d^R\) since we imitate Rouquier
  construction here:
  \[\Phi_{ho}(\beta)=(\oplus_j C_{\beta,j}, d^R),\quad C_{\beta,j}\in \MF_n^\flat\]
  where  the differential \(d^R\) is of \(\mathbf{t}\)-degree \(1\).

  On the other hand \(\Tr^i(C_{\beta,\bullet})\) has its own differential \(d^{mf,\bullet,i}_\beta\) and the
  total differential \(d^{tot,i}\) is the sum \(d^{tot,i}=d^R+d^{mf,\bullet,i}\).

  Tautologically we have
  \[\mathbb{H}(\mathcal{S}^i_\beta)=H^*(\mathcal{S}^i_\beta,d^{tot,i}).\]
  Since \(H^*(\mathcal{S}^i_\beta,d^{mf,\bullet,i})=\Tr^i(C_{\beta,\bullet})\)  we need to show that \(H^*(S^i_\beta,d^{tot,i})=H^*(H^*(\mathcal{S}_\beta^i,d^{nf,\bullet,i}),d^R))\).
  Now we let us recall that both differentials are of \(\mathbf{t}\)-degree \(1\).
  Thus we need to study the spectral sequence of the bi-complex
 \[ \begin{tikzcd}
    &\dar  & \dar &\dar& \dar&\\
    \rar&\calS^i_{\beta,j}[1]\arrow{r}\arrow{d} &\calS^i_{\beta,j}[2]\dar\rar & \calS^i_{\beta,j}[3]\dar\rar&  \calS^i_{\beta,j}[4]\dar &  \\
      \rar&  \calS^i_{\beta,j+1}[2]\arrow{r}\arrow{d} &\calS^i_{\beta,j+1}[3]\dar\rar & \calS^i_{\beta,j+1}[4]\dar\rar&  \calS^i_{\beta,j+1}[5]\dar & \\
      \rar&\calS^i_{\beta,j+2}[3]\arrow{r} &\calS^i_{\beta,j+2}[4]\arrow{r}&\calS^i_{\beta,j+2}[5]\arrow{r}&\calS^i_{\beta,j+2}[6]&
  \end{tikzcd}.
 \]

The horizontal differential in the bi-complex is \(d^{mf,\bullet,\bullet}\) and
the vertical differential is \(d^R\). Thus to compute \(E_1\) page of the spectral sequence  we compute the homology
of the \(d^{mf}\) differential and we obtain:
 \[ \begin{tikzcd}
    & \dar & \dar\arrow[ddl,dashed] &\dar\arrow[ddl,dashed]& \dar\arrow[ddl,dashed]&\\
    &\Tr^i(C_{\beta,1})[1]\arrow{d} &\Tr^i(C_{\beta,1})[2]\dar\arrow[ddl,dashed] & \Tr^i(C_{\beta,1})[3]\dar\arrow[ddl,dashed]&  \Tr^i(C_{\beta,1})[4]\dar\arrow[ddl,dashed] &  \\
      &  \Tr^i(C_{\beta,2})[2]\arrow{d} &\Tr^i(C_{\beta,2})[3]\dar & \Tr^i(C_{\beta,2})[4]\dar&  \Tr^i(C_{\beta,2})[5]\dar & \\
      &\Tr^i(C_{\beta,3})[3] &\Tr^i(C_{\beta,3})[4]&\Tr^i(C_{\beta,3})[5]&\Tr^i(C_{\beta,3})[6]&
  \end{tikzcd}.
 \]
 where the solid arrows are the differentials \(d^R\) and the dashed arrows are
 the induced differentials that govern page \(E_3\) of the spectral sequence.

 Since the complexes \(C_{\beta,j}\) are elements of \(\MF_n^\flat\) shifted
 by \(\mathbf{t}^j\), we conclude that the implies that
 \(\Tr^i(C_{\beta,j})[k]\) is zero if \(k\ne j\). Thus only the first column of the
 last diagram is non-zero and the spectral sequence converges at \(E_2\) page.
\end{proof}

\subsection{Proof of the comparison theorem}
\label{sec:proof-comp-theor}

Thus we can conclude our main result since

\begin{multline*}
\HHH_{alg}(\beta)=\oplus_{i,j} H^j(\Tor^i_{R_n}(\Phi_R(\beta),R_n))=\oplus_{i,j}H^j(\Tr(\Phi_{ho}(\beta)))\\=\oplus_i\mathbb{H}(\CE_{\frn^2}(\mathcal{E}xt(\Phi_{ho}(\beta),\Phi_{ho}(1)\ot\Lambda^i\calB)))^T=\HHH_{geo}(\beta).
\end{multline*}

\section{Applications and further directions}
\label{sec:appl-furth-direct}

\subsection{Torus knots}
\label{sec:torus-links}

In this subsection we explain how the results of the current paper
provide a geometric proof of the conjectures for the torus knots.
More algebraic approach  is used in work
\cite{Hogancamp17}, \cite{Mellit17}.

Here we only treat the case of the torus knots \(T_{n,nk+i}\).
The case of the general torus link follows from the localization
computation in our earlier work \cite{OblomkovRozansky17a} if know
the parity statement for the homology. The parity statement
for the torus link was shown algebraically in the work of Hogancamp and Mellit.
There is a geometric approach to the parity that will be discussed
in our forthcoming work.

The Hilbert-Chow map sends an ideal \(I\) to the support of the
quotient \(\CC[x,y]/I\):
\[HC: \Hilb_n(\CC^2)\to \Sym^n(\CC^2).\]
Thus we define the punctual Hilbert scheme as \(\Hilb_n(\CC^2,0):=HC^{-1}(0^n)\).
Respectively, we define \(\Hilb_n(\CC^2,0)\times \CC_{\mathbf{q}^2}\subset
\Hilb_n(\CC^2)\) to be the \(HC\) preimage of the diagonal
embedding of \(\CC_{\mathbf{q}^2}\)
Let us remind that \(\CC^2=
 \CC_{\mathbf{q}^2}\times \CC_{\mathbf{q}^{-2}\mathbf{t}^2}\) is the \(T_{qt}\)-weight decomposition.

 The results of \cite{OblomkovRozansky17a} allowed us to compute the geometric
 trace for the simplest braid that closes to unknot:
 \[\mathcal{T}r(cox_n)=\mathcal{O}_{\Hilb_n(\CC^2,0)
   \times \CC_{\mathbf{q}^2}},\quad
   cox_n=\sigma_1\cdot \sigma_2\cdots\sigma_{n-1}\in \Br_n\]

 The torus link \(T_{n,m}\) is defined as \(L(cox_n^m)\). The  full twist \(
 FT=cox_n^n\) generates the center of the braid group \(\Br_n\). Hence
 \(T_{n,nk+1}=L(cox_n\cdot FT^k)\). Thus the relation \eqref{eq:FT} implies
 \[(1-\mathbf{q}^2)\cdot\HHH_{alg}(T_{n,nk+1})=H^*(\Hilb_n(\CC^2,0),\det(\calB)^k\ot \Lambda(\calB)).\]
 The vanishing of the higher cohomology of this sheaf is proven by Haiman
 \cite{Haiman02}. Thus the statement of \ref{thm:torus-knots} follows.
 In the same paper of Haiman one can find a localization
 for formula for \(T_{qt}\)-character of the space of global sections.

 \subsection{Poincare duality for knot homology}
 In our earlier work \cite{OblomkovRozansky19a} we constructed a categorical invariant of the
 link:
 \[\calE(L)\in \mathrm{D}^{\per}_{T_{qt}}(R(L))\]
 where \(R(L)=\CC[x_1,\dots,x_\ell,y_1,\cdot,y_\ell]\), \(\ell=|\pi_0(L)|\) and
 \(\deg(x_i)=q^2\), \(\deg(y_i)=t^2/q^2\). The object \(\calE(L)\) also
 has additional \(a\)-grading and it is a direct sum of \(a\)-isotypical
 components.

 The category \(\mathrm{D}^{\per}_{T_{qt}}(R(T))\) has a natural automorphism
 \(\mathfrak{F}\) that is induced by switching \(x_i\) and \(y_i\) for
 all \(i\). We show in \cite{OblomkovRozansky19a} that \(\calE(L)\) has
 the following properties:
 \begin{enumerate}
 \item \(\mathfrak{F}(\calE(L))=\calE(L).\)
 \item \(\HHH_{geo}(\beta)=\calE\otimes^{\mathrm{L}}_{R(L(\beta))}\CC[x_1,\dots,x_\ell]\),
   \(\ell=|\pi_0(L(\beta))|\).
 \item \(\calE(L)\) is the sum of free \(\CC[x,y]\)-modules if
   \(|\pi_0(L(\beta))|=1\).
   \end{enumerate}

   The combination of these there properties together with the comparison property
   \eqref{eq:iso} implies the duality theorem \ref{cor:duality}.

   \subsection{Conjectures}
   \label{sec:conjectures}

   As we show the functor \(\bb\) is fully faithful. Thus we have
   a realization of the category of category Soergel bimodules
   \(\sbim_n\) inside of
   the category of matrix factorizations \(\MF_n^{st}\). It is natural
   to ask whether the functor \(\bb\) is surjective:

   \begin{conjecture}
     The functor \(\bb\) extends to the equivalence of categories:
     \[\MF_n^{st}\cong \mathrm{Ho}(\sbim_n).\]
   \end{conjecture}



   Finally, let us mention that the object \(\calE(L)\) allows
   one to define {\it dualizable homology} \cite{OblomkovRozansky19a} of the link as
   \[\mathrm{HXY}(L)=R\Gamma(\calE(L)).\]

   On the other hand Gorsky and Hogancamp \cite{GorskyHogancamp17}constructed a deformation of the
   homology theory \(\mathrm{HY}\) which they call \(y\)-fied homology
   such that \(\mathrm{HY}(L)\) is a module over \(R(L)\).
   They conjecture that their homology are preserved by the involution
   \(\mathfrak{F}\). Thus we expect the following

   \begin{conjecture} For any link \(L\) we have
     \[\mathrm{HY}(L)=\mathrm{HXY}(L).\]
   \end{conjecture}

   To prove this conjecture we need to study the analog of
   the functor \(\bb\) for the deformations of the categories \(\MF_n^{st}\)
   from \cite{OblomkovRozansky19a}.


\begin{thebibliography}{GZDC99}

\bibitem[AK15a]{ArkhipovKanstrup15}
S.~Arkhipov and T.~Kanstrup.
\newblock Braid group actions on matrix factorizations, 2015.

\bibitem[AK15b]{ArkhipovKanstrup15a}
S.~Arkhipov and T.~Kanstrup.
\newblock {Equivariant Matrix Factorizations and Hamiltonian reduction}.
\newblock {\em Bull. of Korean Math. Soc.}, (5):1803--1825, 2015.

\bibitem[AS12]{AganagicShakirov12}
M.~Aganagic and S.~Shakirov.
\newblock Refined {Chern-Simons} theory and knot homology.
\newblock {\em Proceedings of Symposia in Pure Mathematics}, 85:3--31, 2012.

\bibitem[BR12]{BezrukavnikovRiche2012}
Roman Bezrukavnikov and Simon Riche.
\newblock Affine braid group actions on derived categories of springer
  resolutions.
\newblock {\em Annales scientifiques de l’École normale supérieure},
  45(4):535–599, 2012.

\bibitem[CG97]{ChrisGinzburg}
Neil Chris and Victor Ginzburg.
\newblock {\em Representation Theory and Complex Geometry}.
\newblock Birkhauser, 1997.

\bibitem[DGR06]{DunfieldGukovRasmussen06}
N.~Dunfield, S.~Gukov, and J.~Rasmussen.
\newblock The superpolynomial for knot homologies.
\newblock {\em Experimental Mathematics}, 15(2):129--159, Jan 2006.

\bibitem[EH16]{EliasHogancamp16}
Ben Elias and Matt Hogancamp.
\newblock On the computation of torus link homology, 2016.

\bibitem[EH17]{EliasHogancamp17}
B.~Elias and M.~Hogancamp.
\newblock {Categorical diagonalization}, 2017.

\bibitem[Eis80]{Eisenbud80}
D.~Eisenbud.
\newblock Homological algebra on a complete intersection, with an application
  to group representations.
\newblock {\em Trans. Amer. Math. Soc.}, (1):35--64, 1980.

\bibitem[GH17]{GorskyHogancamp17}
E.~Gorsky and M.~Hogancamp.
\newblock {Hilbert schemes and $y$-ification of Khovanov-Rozansky homology}.
\newblock 2017.

\bibitem[GN15]{GorskyNegut15}
E.~Gorsky and A.~Negu\c{t}.
\newblock Refined knot invariants and hilbert schemes.
\newblock {\em J. Math. Pures Appl.}, 9:403--435, 2015.

\bibitem[Gor12]{Gorsky12a}
E.~Gorsky.
\newblock q,t-catalan numbers and knot homology.
\newblock {\em Contemporary Mathematics}, page 213–232, 2012.

\bibitem[GORS14]{GorskyOblomkovRasmussenShende14}
E.~Gorsky, A.~Oblomkov, J.~Rasmussen, and V.~Shende.
\newblock Torus knots and the rational {DAHA}.
\newblock {\em Duke Mathematical Journal}, 163:2709--2794, 2014.

\bibitem[GRN16]{GorskyNegutRasmussen16}
E.~Gorsky, J.~Rasmussen, and A.~Negu\c{t}.
\newblock Flag {Hilbert} schemes, colored projectors and {Khovanov-Rozansky}
  homology, 2016.

\bibitem[GS06]{GordonStafford06}
I.~Gordon and J.~T. Stafford.
\newblock {Rational Cherednik algebras and Hilbert schemes, II: Representations
  and sheaves}.
\newblock {\em Duke Mathematical Journal}, 132(1):73--135, Mar 2006.

\bibitem[GZDC99]{CampiloDelgadoGuseinZade99}
S~M Gusein-Zade, F~Delgado, and A~Campillo.
\newblock The alexander polynomial of plane curve singularities and rings of
  functions on curves.
\newblock {\em Russian Mathematical Surveys}, 54(3):634–635, Jun 1999.

\bibitem[Hai02a]{Haiman02a}
M.~Haiman.
\newblock {Notes on Macdonald Polynomials and the Geometry of Hilbert Schemes}.
\newblock {\em Symmetric Functions 2001: Surveys of Developments and
  Perspectives}, pages 1--64, 2002.

\bibitem[Hai02b]{Haiman02}
M.~Haiman.
\newblock {Vanishing theorems and character formulas for the Hilbert scheme of
  points in the plane}.
\newblock {\em Invent. Math.}, 149(2):371--407, 2002.

\bibitem[HL14]{HalpernLeistner14}
D.~Halpern-Leistner.
\newblock The derived category of a {GIT} quotient.
\newblock {\em Journal of the American Mathematical Society}, 28(3):871--912,
  Oct 2014.

\bibitem[Hog17]{Hogancamp17}
M.~Hogancamp.
\newblock {Khovanov-Rozansky homology and higher Catalan sequences }.
\newblock {\em arXiv:1704.01562}, 2017.

\bibitem[Kho07]{Khovanov07}
M.~Khovanov.
\newblock Triply-graded link homology and { Hochschild} homology of {Soergel}
  bimodules.
\newblock {\em International Journal of Mathematics}, 18(08):869--885, Sep
  2007.

\bibitem[Kir96]{Kirillov96}
Alexander~A. Kirillov.
\newblock On an inner product in modular tensor categories.
\newblock {\em Journal of the American Mathematical Society}, 9(4):1135–1169,
  1996.

\bibitem[KR08]{KhovanovRozansky08b}
M.~Khovanov and L.~Rozansky.
\newblock Matrix factorizations and link homology {II}.
\newblock {\em Geometry and Topology}, 12:1387--1425, 2008.

\bibitem[Mel17]{Mellit17}
A.~Mellit.
\newblock Homology of torus knots.
\newblock {\em arXiv:1704.07630}, 2017.

\bibitem[MOY98]{MurakamiOhtsukiYamada98}
H.~Murakami, T.~Ohtsuki, and Sh. Yamada.
\newblock {HOMFLY} polynomial via an invariant of colored plane graphs.
\newblock {\em Enseign. Math.}, 44(2):325--360, 1998.

\bibitem[OR17]{OblomkovRozansky17a}
A.~Oblomkov and L.~Rozansky.
\newblock {HOMFLYPT homology of Coxeter links}, 2017.

\bibitem[OR18a]{OblomkovRozansky18b}
A.~Oblomkov and L.~Rozansky.
\newblock {3D TQFT and HOMFLYPT homology}, 2018.

\bibitem[OR18b]{OblomkovRozansky18d}
A.~Oblomkov and L.~Rozansky.
\newblock {3D TQFT and HOMFLYPT homology}.
\newblock {\em arXiv:1812.06340}, 2018.

\bibitem[OR18c]{OblomkovRozansky18}
A.~Oblomkov and L.~Rozansky.
\newblock {A categorification of a cyclotomic {Hecke} algebra}, 2018.

\bibitem[OR18d]{OblomkovRozansky17}
A.~Oblomkov and L.~Rozansky.
\newblock Affine braid group, {JM} elements and knot homology.
\newblock {\em Transformation Groups}, Jan 2018.

\bibitem[OR18e]{OblomkovRozansky18a}
A.~Oblomkov and L.~Rozansky.
\newblock {Categorical Chern character and braid groups}, 2018.

\bibitem[OR18f]{OblomkovRozansky16}
A.~Oblomkov and L.~Rozansky.
\newblock Knot homology and sheaves on the {Hilbert} scheme of points on the
  plane.
\newblock {\em Selecta Mathematica}, 24(3):2351--2454, Jan 2018.

\bibitem[OR19]{OblomkovRozansky19a}
A.~Oblomkov and L.~Rozansky.
\newblock {Dualizable link homology}, 2019.

\bibitem[Orl04]{Orlov04}
Dmitri Orlov.
\newblock Triangulated categories of singularities and {D}-branes in
  {Landau-Ginzburg} models.
\newblock {\em Proc. Steklov Inst. Math.}, 246(3):227--248, 2004.

\bibitem[ORS18]{OblomkovRasmussenShende12}
A.~Oblomkov, J.~Rasmussen, and V.~Shende.
\newblock The {Hilbert} scheme of a plane curve singularity and the {HOMFLY}
  homology of its link, with {Appendix by E. Gorsky}.
\newblock {\em Geometry and Topology}, 22(2):645--691, Jan 2018.

\bibitem[OS12]{OblomkovShende12}
A.~Oblomkov and V.~Shende.
\newblock The {Hilbert} scheme of a plane curve singularity and the {HOMFLY}
  polynomial of its link.
\newblock {\em Duke Mathematical Journal}, 161(7):1277--1303, May 2012.

\bibitem[OY16]{OblomkovYun16}
A.~Oblomkov and Zh. Yun.
\newblock {Geometric representations of graded and rational Cherednik
  algebras}.
\newblock {\em Advances in Mathematics}, 292:601--706, Apr 2016.

\bibitem[PV11]{PolishchukVaintrob11}
A.~Polishchuk and Arkady Vaintrob.
\newblock Matrix factorizations and singularity categories for stacks.
\newblock {\em Annales de l’institut Fourier}, 61(7):2609--2642, 2011.

\bibitem[Ras15]{Rasmussen15}
Jacob Rasmussen.
\newblock Some differentials on khovanov–rozansky homology.
\newblock {\em Geometry and Topology}, 19(6):3031--3104, Dec 2015.

\bibitem[Ric08]{Riche08}
Simon Riche.
\newblock Geometric braid group action on derived categories of coherent
  sheaves (with a joint appendix with {Roman Bezrukavnikov}).
\newblock {\em Representation theory}, 12:131--169, 2008.

\bibitem[Ric10]{Riche10}
Simon Riche.
\newblock Koszul duality and modular representations of semisimple lie
  algebras.
\newblock {\em Duke Mathematical Journal}, 154(1):31–134, Jul 2010.

\bibitem[Rou04]{Rouquier04}
R.~Rouquier.
\newblock {Categorification of the braid groups}, 2004.

\bibitem[Soe01]{Soergel00}
W.~Soergel.
\newblock Langlands' philosophy and {Koszul} duality.
\newblock In {\em Algebra-Representation Theory (Constanta,2000)}, pages
  379--414, 2001.

\bibitem[Wu08]{Wu08}
Hao Wu.
\newblock Braids, transversal links and the khovanov-rozansky theory.
\newblock {\em Transactions of the American Mathematical Society},
  360(07):3365–3390, Feb 2008.

\end{thebibliography}


 \end{document}